\documentclass[10pt,letterpaper]{article}
\usepackage{float, graphicx}

\usepackage[]{epsfig}
\usepackage{amsmath, amsthm, amssymb,}
\usepackage{epsfig}
\usepackage{verbatim}
\usepackage{multicol}
\usepackage{url}
\usepackage{latexsym}
\usepackage{mathrsfs}
\usepackage[colorlinks, bookmarks=true, citecolor=blue]{hyperref}
\usepackage{amsmath}
\usepackage{enumerate}
\usepackage[normalem]{ulem}
\usepackage{bm}
\usepackage{dsfont}
\usepackage{stmaryrd}
\usepackage{enumitem}
\usepackage{tikz}
\usepackage{cleveref}
\usepackage{color}
\usepackage{xcolor}

\usepackage{authblk}
\numberwithin{equation}{section}

\usepackage[margin=1.1in]{geometry}

\makeatletter

\def\@settitle{\begin{center}%
    \bfseries
 \normalfont\LARGE\@title
  \end{center}%
}
\def\@setauthors{\begin{center}%
 \normalsize\@author
  \end{center}%
}

\makeatother

\newcommand{\sfb}{{\mathsf b}}

\newcommand{\sft}{{\mathsf t}}

\newcommand{\sfx}{{\mathsf x}}
\newcommand{\sfy}{{\mathsf y}}
\newcommand{\sfz}{{\mathsf z}}

\newcommand{\sfp}{{\mathsf p}}

\newcommand{\sfS}{{\mathsf S}}
\newcommand{\sfL}{{\mathsf L}}
\newcommand{\sfT}{{\mathsf T}}
\newcommand{\sfN}{{\mathsf N}}

\renewcommand{\cal}{\mathcal}
\newcommand\cA{{\mathcal A}}
\newcommand\cB{{\mathcal B}}
\newcommand{\cC}{{\cal C}}
\newcommand{\cD}{{\cal D}}

\newcommand{\cF}{{\cal F}}
\newcommand{\cG}{{\cal G}}
\newcommand\cH{{\mathcal H}}

\newcommand{\cJ}{{\cal J}}
\newcommand{\cM}{{\cal M}}

\newcommand{\cP}{{\cal P}}

\newcommand{\cT}{{\mathcal T}}

\newcommand\cW{{\mathcal W}}

\newcommand{\cI}{{\mathcal I}}

\newcommand{\fR}{{\mathfrak R}}

\newcommand{\fP}{{\frak P}}

\newcommand{\bmb}{{\bm{b}}}

\newcommand{\bme}{{\bm{e}}}

\newcommand{\bmr}{{\bm{r}}}

\newcommand{\bmx}{{\bm{x}}}
\newcommand{\bmy}{{\bm{y}}}
\newcommand{\bmz} {{\bm {z}}}

\newcommand{\bmmu}{{\bm \mu}}
\newcommand{\bmla}{{\bm \la}}
\newcommand{\bmnu}{{\bm \nu}}

\newcommand{\bml}{{\bm l}}

\newcommand{\rd}{{\rm d}}

\newcommand{\ri}{\mathrm{i}}

\newcommand{\bB}{{\mathbb B}}
\newcommand{\bC}{{\mathbb C}}

\newcommand{\bE}{\mathbb{E}}
\newcommand{\bH}{\mathbb{H}}
\newcommand{\bN}{\mathbb{N}}
\newcommand{\bP}{\mathbb{P}}
\newcommand{\bQ}{\mathbb{Q}}
\newcommand{\bR}{{\mathbb R}}

\newcommand{\bY}{\mathbb Y}
\newcommand{\bZ}{\mathbb{Z}}
\newcommand{\bW}{\mathbb{W}}

\newcommand{\cin}{{\omega_-}}

\newcommand{\al}{\alpha}

\newcommand{\la}{\lambda}

\newcommand{\eps}{\varepsilon}

\DeclareMathOperator{\supp}{supp}
\DeclareMathOperator{\dist}{dist}

\DeclareMathOperator{\OO}{O}
\DeclareMathOperator{\oo}{o}

\DeclareMathOperator{\argmin}{argmin}

\DeclareMathOperator{\Sym}{Sym}
\DeclareMathOperator{\Li}{Li}

\DeclareMathOperator{\Hib}{Hib}
\DeclareMathOperator{\Adm}{Adm}
\DeclareMathOperator{\PV}{PV}

\renewcommand{\Re}{\mathop{\mathrm{Re}}}
\renewcommand{\Im}{\mathop{\mathrm{Im}}}

\newcommand{\deq}{\mathrel{\mathop:}=} 

\renewcommand{\leq}{\leqslant}
\renewcommand{\geq}{\geqslant}

\newcommand{\del}{\partial}

\newcommand{\wt}{\widetilde}

\newcommand{\qq}[1]{[\![{#1}]\!]}

\newcommand{\beq}{\begin{equation}}
\newcommand{\eeq}{\end{equation}}

\newcommand{\bl}{{\bm l}}
\newcommand{\br}{{\bm r}}

\theoremstyle{plain} 
\newtheorem{theorem}{Theorem}[section]
\newtheorem*{theorem*}{Theorem}
\newtheorem{lemma}[theorem]{Lemma}
\newtheorem*{lemma*}{Lemma}

\newtheorem*{corollary*}{Corollary}
\newtheorem{proposition}[theorem]{Proposition}
\newtheorem*{proposition*}{Proposition}
\newtheorem{assumption}[theorem]{Assumption}
\newtheorem*{assumption*}{Assumption}
\newtheorem{claim}[theorem]{Claim}

\newtheorem{definition}[theorem]{Definition}
\newtheorem*{definition*}{Definition}

\newtheorem*{example*}{Example}
\newtheorem{remark}[theorem]{Remark}

\newtheorem*{remark*}{Remark}
\newtheorem*{remarks*}{Remarks}

\title{Asymptotics of Symmetric Polynomials: A Dynamical Point of view}

    \author[1]{Alice~Guionnet\thanks{aguionne@ens-lyon.fr}}

  \author[2]{Jiaoyang~Huang\thanks{huangjy@wharton.upenn.edu}}

\affil[1]{CNRS-ENS Lyon, France}
\affil[2]{University of Pennsylvania, United States of America}

\date{}

\begin{document}

\maketitle

\begin{abstract}
In this paper we study the asymptotic behavior of the (skew) Macdonald and Jack symmetric polynomials as the number of variables grows to infinity. We characterize their limits in terms of certain variational problems.  As an intermediate step, we establish a large deviation principle for the $\theta$ analogue of non-intersecting Bernoulli random walks. When $\theta=1$, these walks are equivalent to random Lozenges tilings of strip domains, where the variational principle (with general domains and boundary conditions) has been proven in the seminal work \cite{cohn2001variational} by Cohn, Kenyon, and Propp.  Our result gives a new argument of this variational principle, and also extends it to non-intersecting $\theta$-Bernoulli random walks for any $\theta \in (0,\infty)$. Remarkably, the rate functions remain identical, differing only by a factor of $1/\theta$.
\end{abstract}

\setcounter{tocdepth}{1}
{
\hypersetup{linkcolor=black}
\tableofcontents
}

\newpage

\section{Introduction}

Macdonald symmetric functions, initially discovered by Macdonald in the late 1980s \cite{macdonald1998symmetric}, are central to a number of key developments in mathematics and mathematical physics. They are indexed by Young diagrams and implicitly depend on two parameters $q,t\in (0,1)$. Their degeneracy, by taking $t=q^\theta$ and sending $q$ to approach one, gives the Jack polynomials. This family of symmetric functions depends on the positive parameter $\theta>0$, and was first introduced in \cite{jack1972xxv,jack1970class}. Notably, when $\theta=1$, the Jack polynomials coincide with the well-known Schur polynomials, up to some multiplicative constants. 

These symmetric polynomials appear in algebraic combinatorics as generating functions, and provide a systematic way to enumerate combinatorial objects. In representation theory, these symmetric polynomials offer a crucial bridge between algebraic structures, combinatorial objects and group representations. Recently, symmetric polynomials have become central tools in the development of integrable stochastic models. This theory has found numerous applications in areas such as random partitions, random matrix theory, and directed polymers. In this article, we study the asymptotic behaviors of symmetric polynomials, namely (skew) Macdonald and Jack polynomials, as the number of parameters goes to infinity. In this introduction, we discuss the results for skew Jack polynomials. The results on skew Macdonald polynomials are more involved and are presented in \Cref{s:Macdonaldintro}. Recall that skew Jack polynomials $J_{\bmla\setminus \bmmu}(\bmx; \theta)$ are symmetric functions in infinitely many variables $\bmx=(x_i)_{1\leq i\leq \infty}$ and parametrized by two Young diagrams $\bm\mu\subset \bmla$, where $\bmmu=(\mu_1\geq \mu_2\geq \cdots\geq \mu_{\ell(\bmmu)}), \bmla=(\la_1\geq \la_2\geq \cdots\geq \la_{\ell(\bmla)})$ are arrays of integer numbers  with length $\ell(\bmmu), \ell(\bmla)$ respectively.  We refer to \Cref{s:YoungD} and \Cref{s:Jack} for more detailed discussion on (skew) Jack polynomials. We consider the limit of skew Jack polynomials along a sequence of Young diagrams $\bmla^{(N)}, \bmmu^{(N)}$ and parameters $\bmb^{(N)}$, where the column and row sizes of Young diagrams and the number of parameters grow linearly in $N$.
We show the limit $N^{-2}\ln J_{\bmla^{(N)}\setminus \bmmu^{(N)}}(\bmb^{(N)}; \theta)$ exists and give an explicit characterization of the limit in terms of a variational problem.

\subsection{Nonintersecting $\theta$-Bernoulli Walk Ensembles}
Understanding the asymptotic behavior of symmetric polynomials is fundamental to modern combinatorics,  statistical mechanics and representation theory.  However,  there are few results in this direction, as symmetric polynomials tend to have complicated structures and thus are not easily approachable by algebraic combinatorics techniques. In this article, we will study the asymptotic behaviors of symmetric polynomials via a dynamical approach.

Many symmetric polynomials arise naturally as the partition functions of interacting particle systems. Some notable examples are
\begin{enumerate}
\item The partition function of nonintersecting Brownian bridges is given by the Harish-Chandra-Itzykson-Zuber integral formula \cite{guionnet2002large}. It is closely related with Schur polynomials.
\item Macdonald processes as introduced by Borodin and Corwin \cite{borodin2014macdonald}, by definition have Macdonald polynomials as their partition functions.  Degenerations of Macdonald processes include the Schur processes \cite{okounkov2003correlation,okounkov2001infinite},  Jack processes \cite{gorin2015multilevel,huang2021law}, Hall-Littlewood processes \cite{borodin2016between} and Whittaker processes \cite{borodin2014macdonald}.
\item Partition functions of vertex models give families of symmetric rational functions, which generalize Schur symmetric polynomials, as well as some of their variations \cite{aggarwal2023free,aggarwal2023colored,borodin2017family,borodin2018higher}.
\end{enumerate}
On one hand, the asymptotic behavior of symmetric polynomials provides a valuable tool for characterizing the dynamics of interacting particle systems, where these polynomials serve as partition functions. On the other hand, this also offers an alternative approach to investigating the asymptotic behaviors of symmetric polynomials by exploring the asymptotic behaviors of the associated interacting particle systems using dynamical approaches.

Fix large $N$, and take Young diagrams $\bm\la=(\la_1\geq \la_2\geq \cdots\geq\la_N), \bmmu=(\mu_1\geq \mu_2\geq \cdots\geq\mu_N)$, and parameters $\bmb=(b_0,b_1, b_2,\cdots, b_{\sfT-1})$. 
The skew Jack symmetric polynomials $J_{\bmla'\setminus \bmmu'}(\bmb; \theta^{-1})$ (where $\bmla', \bmmu'$ are the transposes of $\bmla, \bmmu$),  can be interpreted as the partition functions of $N$-particle nonintersecting $\theta$-Bernoulli walks. For any Young diagrams $\bmnu=(\nu_1\geq \nu_2\geq \cdots\geq \nu_N)$ with at most $N$ rows, we encode it by a particle configuration $\bm\sfx=\bm\sfx(\bmnu)=(\sfx_1, \sfx_2,\cdots, \sfx_N)$, where
\begin{align}\label{e:xtonu}
\sfx_i=\nu_i-(i-1)\theta,\quad 1\leq i\leq N. 
\end{align}
From the construction, this particle configuration lives on the following $\theta$-dependent lattice 
	\begin{align}\label{e:defW}
	\bW_\theta^N\deq \bigl\{(\sfx_1,\dots,\sfx_N)\in\mathbb R^N\mid \sfx_1\in \mathbb Z,\quad \sfx_{i}-\sfx_{i+1}\in \theta+\mathbb Z_{\geq 0}, \quad i=1,2,\dots,N-1\bigr\}.
	\end{align}

\begin{definition}[non-intersecting $\theta$-Bernoulli walk ensembles]	\label{d:bernoulliwalk}
	An $n$-particle non-intersecting $\theta$-Bernoulli walk from time $0$ to time  $\sfS$ is a sequence of particle configurations $\mathsf{p} = \big( \bm{\mathsf{x}} (0),\bm{\mathsf{x}} (1), \ldots , \bm{\mathsf{\sfx}} (\sfS) \big) \in (\bW^n_\theta)^{\sfS}$ such that $\bm\sfx(\sft)=(\sfx_1(\sft), \sfx_2(\sft),\cdots, \sfx_n(\sft))$, and $\mathsf{x}_i (\sft + 1) - \mathsf{x}_i (\sft) \in \{ 0, 1 \}$ for each $\sft \in [0, \sfS]$ and $i\in\{1,\ldots,n\}$; viewing $\{\sfx_i(\sft)\}_{0\leq \sft\leq \sfS}$ as the space-time trajectory for the $i$-th particle, which may either jump to the right or not move at each step. From the construction of the lattice \eqref{e:defW}, the paths are non-intersecting, since we have $\mathsf{x}_i (\sft) > \mathsf{x}_j (\sft)$ whenever $1\leq i<j\leq n$ and $0\leq \sft\leq\sfS$. \end{definition}

Assume $\bmmu,\bmla$ are Young diagrams with at most $N$ rows. We identify them as particle configurations $\bm\sfy=(\sfy_1>\sfy_2>\cdots>\sfy_N), \bm\sfz=(\sfz_1>\sfz_2>\cdots>\sfz_N)\in \bW_\theta^N$
\begin{equation}\label{e:x2}
\sfy_i=\mu_i-\theta(i-1),\quad \sfz_i=\lambda_i-\theta(i-1), \quad1\leq i\leq N.
\end{equation}


We denote the set of non-intersecting $\theta$-Bernoulli walks from $\bm\sfy=(\sfy_1>\sfy_2>\cdots>\sfy_N)\in \bW_\theta^N$ to $\bm\sfz=(\sfz_1>\sfz_2>\cdots>\sfz_N)\in \bW_\theta^N$ as
\begin{align}\label{e:pathset}
\cP(\bm\sfy, \bm\sfz;\sfT)
=\{\sfp=\{\bm{\sfx}(\sft)\}_{0\leq \sft\leq \sfT}\in (\bW^N_\theta)^{\sfT} : \bm\sfx(0)=\bm\sfy,\bm\sfx(\sfT)=\bm\sfz\}.
\end{align}
The set $\cP(\bm\sfy, \bm\sfz;\sfT)$ is nonempty if there exists a non-intersecting Bernoulli walk $\{\bm\sfx(\sft)\}_{0\leq \sft\leq \sfT}$ from $\bm\sfx(0)=\bm\sfy$ to $\bm\sfx(\sfT)=\bm\sfz$, which is equivalent to 
\begin{align}\label{e:tilable}
\sfy_i\leq \sfz_i\leq \sfy_i+\sfT, \quad 1\leq i\leq N.
\end{align}
Indeed, if \eqref{e:tilable} holds we can take 
\begin{align*}
\sfx_i(\sft)=\max\{\sfy_i, \sfz_i-(\sfT-\sft)\},\quad 1\leq i\leq N.
\end{align*}

Then we can rewrite the skew Jack polynomial $J_{\bmla'/\bmmu'}(\cdot; \theta^{-1})$  evaluated at $(b_0,b_1,\cdots, b_{\sfT-1})$, in terms of the weights \eqref{e:weightpb}   as follows
\begin{align}\label{e:Jackexp}
J_{\bmla'/\bmmu'}(b_0, \cdots, b_{\sfT-1};\theta^{-1})
=\frac{J_{\bmmu}(1^N;\theta)}{J_{\bmla}(1^N;\theta)}\sum_{\sfp\in \cP(\bm\sfy;\bm\sfz;\sfT)}\cW(\sfp;\bmb),
\end{align}
where $\bm\la', \bmmu'$ are the transposes of $\bmla,\bmmu$, and for non-intersecting $\theta$-Bernoulli walk  $\sfp=\{\bm\sfx(\sft)\}_{0\leq \sft\leq \sfT}$ and $\bmb=(b_0, b_1,\cdots, b_{\sfT-1})$, we define their weights as
\begin{align}\label{e:weightpb}
\cW(\sfp;\bmb)=\prod_{0\leq \sft\leq \sfT-1}\prod_{1\leq i<j\leq n}\frac{(\sfx_i(\sft)+\theta e_i(\sft))-(\sfx_j(\sft)+\theta e_j(\sft))}{{\sfx_i(\sft)}-{\sfx_j(\sft)}} \prod_{1\leq i\leq N}\sfb_\sft^{ e_i(\sft)},
\end{align}
where $\bme(\sft)=(e_1(\sft), e_2(\sft),\cdots, e_N(\sft))$ is the increment at time $\sft$, 
\begin{align*}
e_i(\sft)=\sfx_i(\sft+1)-\sfx_i(\sft)\in \{0,1\}, \quad 1\leq i\leq N.
\end{align*}

The formula \eqref{e:Jackexp} is derived from the construction of the Jack process, as explained in more detail in \Cref{s:Jprocess}. On the left-hand side of \eqref{e:Jackexp}, there are explicit formulas for the Jack symmetric polynomials evaluated at $1^N$, i.e. $J_\bmmu(1^N;\theta), J_{\bm\lambda}(1^N;\theta)$,  making them amenable for asymptotic analysis. The primary challenge lies in analyzing the sum of the weights $\cW(\sfp;\bmb)$. In the special case where $b_i=1$, this sum of weights can be interpreted as the partition function of non-intersecting $\theta$-Bernoulli random walks from $\bm\sfy$ to $\bm\sfz$, with the transition probability given by\begin{align}\label{e:mdensity2}
\bP(\bm \sfx(\sft+1)=\bm\sfx+\bme|\bm\sfx(\sft)=\bm\sfx)=\frac{1}{2^N} 
    \frac{V(\bm\sfx+\theta \bme)}{V(\bm\sfx)}=
\frac{1}{2^N}
\prod_{1\leq i<j\leq N}\frac{(\sfx_i+\theta e_i)-(\sfx_j+\theta e_j)}{{\sfx_i}-{\sfx_j}},
\end{align}
for $0\leq \sft\leq \sfT-1$, where $V$ is the Vandermonde determinant, and $\bme=(e_1, e_2,\cdots,e_N)\in \{0,1\}^N$, see \Cref{f:Young_diagram}.
We remark that the above transition probability is given by pairwise interactions, and the interactions between adjacent particles are singular to prevent colliding. More precisely, from the above transition probability, if $\bm\sfx\in \bW_\theta^N$ and $\sfx_i-\sfx_{i+1}\geq \theta+1$, then $(\sfx_i+e_i)-(\sfx_{i+1}+e_{i+1})\in \theta+\bZ_{\geq 0}$. If $\sfx_i-\sfx_{i+1}=\theta$, then \eqref{e:mdensity2} is nonzero only if $(e_i, e_{i+1})\in \{(0,0), (1,0), (1,1)\}$. In these cases, $(\sfx_i+e_i)-(\sfx_{i+1}+e_{i+1})\in \theta+\bZ_{\geq 0}$. Therefore the Markov process \eqref{e:mdensity2} stays in the lattice $\bW_\theta^N$.

\begin{figure}
	\begin{center}
	 \includegraphics[scale=0.4,trim={0cm 8cm 0cm 9cm},clip]{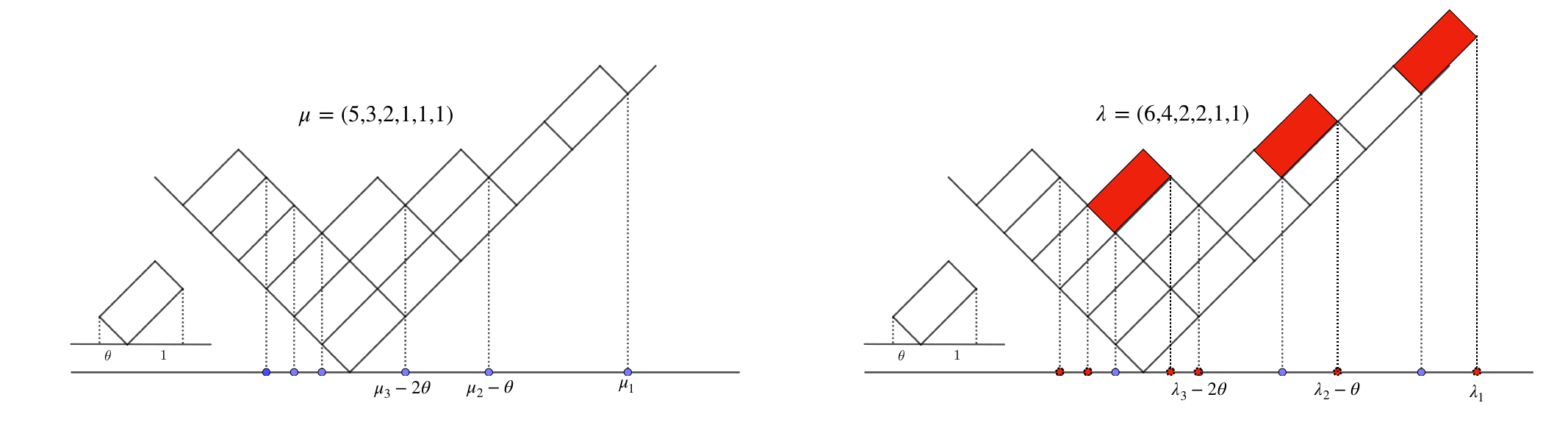}
	 \caption{Shown above is the transition from the Young diagram $\bm\mu=(5,3,2,1,1,1)$ to the Young diagram $\bm\la=(6,4,2,2,1,1)$, which are encoded by the particle systems through the relation \eqref{e:xtonu}. }
	 \label{f:Young_diagram}
	 \end{center}
	 \end{figure}	

Our first main result establishes a large deviation principle for the non-intersecting $\theta$-Bernoulli random walks \eqref{e:mdensity2}.  Notably, when $\theta=1$, these walks reduce to non-intersecting Bernoulli random walks, which is equivalent to random Lozenge tiling of strip domains. Cohn, Kenyon, and Propp \cite{cohn2001variational} previously proved a variational principle for random tiling models (domino tiling and lozenge tiling) with general boundary conditions.   Our result extends this variational principle to non-intersecting $\theta$-Bernoulli random walks for any $\theta \in (0,\infty)$. Remarkably, the rate functions remain identical, differing only by a factor of $1/\theta$. To state the large deviation principle, we need to introduce the notations of height function and surface tension.

\subsection{Height Function and Surface Tension}
Given any $N$-particle nonintersecting $\theta$-Bernoulli walk ensemble $ \sfp=\big( \bm{\mathsf{x}} (0),\bm{\mathsf{x}} (1), \ldots , \bm{\mathsf{\sfx}} (\sfT) \big) \in (\bW^N_\theta)^{\sfT}$ from time $0$ to time $\sfT$, we denote the increment at time $\sft$ as $\bme(\sft)=(e_1(\sft), e_2(\sft), \cdots, e_N(\sft))$, 
\begin{align*}
\bme(\sft)=\bm\sfx(\sft+1)-\bm\sfx(\sft)\in \{0,1\}^N.
\end{align*}
We denote the rescaled particle configuration (by rescaling space and time by $1/N$),
\begin{align}\label{e:scaling}
\bmx(t)=(x_1(t), x_2(t),\cdots, x_N(t))=\frac{\bm\sfx(\sft)}{N},\quad t=\frac{\sft}{N}, \quad 0\leq \sft\leq \sfT.
\end{align}
In the rest of the paper, we will use the scaling above (we will denote $\bm\sfx(\sft)$ for particle configurations before scaling, and $\bmx(t)$ for after scaling). 
For $   \sft/N\leq t\leq   (\sft+1)/N$, we linearly interpolate the particle configurations
\begin{align*}
\bmx(t)
&= \left(\sft+1-Nt\right)\bmx\left(\frac{    \sft}{N}\right)+ \left(Nt-\sft\right)\bmx\left(\frac{    \sft+1}{N}\right)=\bmx\left(\frac{\sft}{N}\right)+\left(Nt-\sft\right)\frac{\bme(\sft)}{N}.
\end{align*}
We encode the particle densities of $\bmx(t)$ as 
\begin{align}\label{eq_rho_x_def}
    \rho(x;\bmx(t))=\sum_{i=1}^N \bm1(x\in [x_i(t), x_i(t)+\theta/N]),\quad 0\leq t\leq \sfT/N,
\end{align}
and define the associated height function 
\begin{align}\label{e:Hxt}
H(x, t)=\int_{-\infty}^x   \rho(y;\bmx(t))\rd y, \quad (x,t)\in \bR\times [0,\sfT/N].
\end{align}
We notice that the total mass of $\rho$  in \eqref{eq_rho_x_def} is $\theta$ for any $0\leq t\leq \sfT/N$, and $H(x,t)$ is a non-decreasing function of $x$ from $0$ to $\theta$. The walk ensemble $\sfp=\big( \bm{\mathsf{x}} (0),\bm{\mathsf{x}} (1), \ldots , \bm{\mathsf{\sfx}} (\sfT) \big)$, the rescaled particle configuration $\bmx(t)$ can be recovered easily by either the empirical particle density $\varrho(x;\bmx(t))$, or the height function $H(x,t)$. In the rest of the paper, we will not distinguish them.

Given an $N$-particle nonintersecting $\theta$-Bernoulli walk ensemble $\sfp=\{\bm{\sfx}(\sft)\}_{0\leq \sft\leq \sfT}$, we recall its weight $\cW(\sfp;\bmb)$ from \eqref{e:weightpb}. For $\bmb=1^\sfT$, we simply denote the weight as $\cW(\sfp)$,
\begin{align}\label{e:weightp}
\cW(\sfp)=\cW(\sfp; 1^\sfT)=\prod_{0\leq \sft\leq \sfT}\prod_{1\leq i<j\leq N}\frac{(\sfx_i(\sft)+\theta e_i(\sft))-(\sfx_j(\sft)+\theta e_j(\sft))}{{\sfx_i(\sft)}-{\sfx_j(\sft)}}. 
\end{align}
Under the transition probability \eqref{e:mdensity2}, the probability of the walk $\sfp$ is given by $2^{-N\sfT}\cW(\sfp)$.
We also define the weight of the height function $H$ associated with $\sfp=\{\bm{\sfx}(\sft)\}_{0\leq \sft\leq \sfT}$ as
$
\cW(H)=\cW(\sfp).
$

From the construction of the height function \eqref{e:Hxt}, $H(x,t)$ is $2$-Lipschitz, and almost surely, 
\begin{align*}
\nabla H(x,t)=(\partial_{x}H(x,t),\partial_{t} H(x,t))\in \{(0,0), (1,0), (1,-1)\}.
\end{align*}
To analyze the limits of height functions of these interacting particle systems, it will be useful to introduce continuum analogs of the height functions considered in \eqref{e:Hxt}. So, set 
	\begin{flalign} 
		\label{t}
		\mathcal{T} = \big\{ (s, t) \in (0, 1) \times \mathbb{R}_{< 0}: s + t > 0 \big\} \subset \mathbb{R}^2, 
	\end{flalign} 

	\noindent and its closure $\overline{\mathcal{T}} = \big\{ (s, t) \in [0, 1] \times \mathbb{R}_{\leq 0} :  s + t \geq 0 \big\}$. We interpret $\overline{\mathcal{T}}$ as the set of possible gradients, also called \emph{slopes}, for a continuum height function. The height functions $H(x,t)$ associated with particle configurations (from \eqref{e:Hxt}) is $2$-Lipschitz and satisfies $\nabla H(x,t)\in \overline{\cT}$ for $(x,t)\in \bR\times [0,\sfT/N]$. \
	
\begin{definition}
Fix time $T>0$, and denote $\fR=\bR\times [0,T]$. We say that a function $H : \fR\rightarrow \mathbb{R}$ is \emph{admissible} if $H$ is $2$-Lipschitz and $\nabla H(u) \in \overline{\mathcal{T}}$ for almost all $u \in \fR$. We further say that the boundary height function $h=(h(x,0), h(x,T)): \del\fR \rightarrow \mathbb{R}$ \emph{admits an admissible extension to $\fR$} if $\Adm (\mathfrak{R}; h)$, the set of admissible functions $H: \fR\rightarrow \mathbb{R}$ with $H |_{\bR\times\{0,T\} } = h$, is not empty.
\end{definition}
	
%

	For any $x \in \mathbb{R}_{\geq 0}$ and $(s, t) \in \overline{\mathcal{T}}$ we denote the \emph{Lobachevsky function} $L: \mathbb{R}_{\geq 0} \rightarrow \mathbb{R}$ and the \emph{surface tension} $\sigma : \overline{\mathcal{T}} \rightarrow \mathbb{R}^2$ by 
	\begin{flalign}
		\label{sigmal} 
		L(x) = - \displaystyle\int_0^x \log |2 \sin z| \mathrm{d} z; \qquad \sigma (s, t) = \displaystyle\frac{1}{\pi} \Big( L \big(\pi (1-s) \big) + L (- \pi t) + L \big( \pi (s + t) \big) \Big).
	\end{flalign}
	
	\noindent For any admissible height function $H$ on $\fR$, we further denote the \emph{entropy functional}
	\begin{flalign}
		\label{efunctionh} 
		\mathcal{E} (H) = \displaystyle\iint_{\mathfrak{R}} \sigma \big( \nabla H (x,t) \big) \mathrm{d}x\rd t.
	\end{flalign}

The entropy functional \eqref{efunctionh} also governs the logarithm of the number of lozenge tilings for a given domain and boundary height function \cite{cohn2001variational}. The surface tension $\sigma$ is strictly concave within the interior of $\cT$ and linear along the boundary of $\overline{\cT}$. As a consequence, the maximizer of the energy functional \eqref{efunctionh} possesses what are known as ``liquid regions", where the solution is real analytic, and ``frozen regions", where the solution is piecewise linear. Variational problems associated with \eqref{efunctionh}, and in more general setting, have been explored in previous studies \cite{10.1007/s11511-007-0021-0, 10.1215/00127094-2010-004, astala2020dimer}.

\subsection{Large Deviation Principle for non-intersecting $\theta$-Bernoulli random walks}
We prove a large deviation for non-intersecting $\theta$-Bernoulli random walks \eqref{e:mdensity2} with given boundary condition at time $0$ and time $\sfT$. We make the following assumption on the height profile of the limiting boundary condition. 

\begin{definition}\label{a:boundaryheight}
Fix any $\theta>0$. We denote by $\Adm^{\partial}_{\theta}(\fR)$ the set of boundary height functions $h=(h(x,0),h(x,T)): \del \fR\mapsto \bR$ such that 
\begin{enumerate}
\item $\del_x h(x,0), \del_x h(x,T)$ are two compactly supported  positive measures with density bounded by $1$ and total mass $\theta$. 
\item  Moreover, $h=(h(x,0),h(x,T))$ admits an admissible extension to $\fR$.
 \end{enumerate}
 
\end{definition}
We equip $\Adm(\fR, h)$ with the uniform topology. 
First, we remark that since we assumed $H$ is Lipschitz, the uniform topology in the spatial variable is equivalent to the weak topology on its derivative, as can be easily checked by integration by parts. Moreover, it is evident that  $\Adm(\fR, h)$ is a compact space. Moreover,   if $H\in\Adm(\fR, h)$, then for every $t\in [0,T]$, $\partial_{x}H(.,t)$ is compactly supported, bounded by one and , because $\nabla H\in  \overline{\mathcal{T}}$, $\partial_{x }H(.,t)$ is a non-negative measure for all time $t$. 
\begin{remark}
We 
remark that $\Adm(\fR, h)$  depends on $\theta$ due to the fact that  $h\in  \Adm^{\partial}_{\theta}(\fR)$. We  want to stress this dependency because it has important consequences, in particular $H(x,t)$
is a non-decreasing function of $x$ from $0$ to $\theta$ for  all times $0\leq t\leq T$. Indeed, since $\partial_{t} H(x,t)$ is non-positive for every $x$, $H(x,T)\leq H(x,t)\leq  H(x, 0)$. On the other hand, $H(x,T)=h(x,T)$ and $H(x,0)=h(x,0)$ are equal 
 to $0$ for $x$ sufficiently small, and $\theta$ for $x$ sufficiently large since $h\in  \Adm^{\partial}_{\theta}(\fR)$. Therefore, for every time $t\in [0,T]$, $H(x,t)$  is equal
 to $0$ for $x$ sufficiently small, and $\theta$ for $x$ sufficiently large.
\end{remark}

To state our main result, let us precise our running assumption concerning the particle configurations. 
\begin{definition}\label{assume} Let $\theta,T>0$ and  $h\in \Adm^{\partial}_{\theta}(\fR)$. 
 Given two  sequences  of particle configurations $(\bm\sfy^{(N)},\bm\sfz^{(N)})$, 
\begin{align*}
\bm\sfy^{(N)}=(\sfy_1^{(N)}\geq \sfy_2^{(N)}\geq\cdots\geq \sfy_N^{(N)})\in \bW_\theta^N,\quad
 \bm\sfz^{(N)}=(\sfz_1^{(N)}\geq \sfz_2^{(N)}\geq\cdots\geq \sfz_N^{(N)})\in \bW_\theta^N, \quad N\geq 1\,.
 \end{align*} 
 we say that $(\bm\sfy^{(N)},\bm\sfz^{(N)},\sfT)$ are $(h,\theta,T)$-admissible  iff 
\begin{enumerate}
\item There exists a constant $C>0$, $|\sfy_i^{(N)}|, |\sfz_i^{(N)}|\leq CN$ for $1\leq i\leq N$.
\item The empirical density (recall from \eqref{eq_rho_x_def}) of $\bm\sfy^{(N)},  \bm\sfz^{(N)}$ converges, namely when $N\rightarrow \infty$
\begin{align}\label{e:rhoconverge}
\varrho(x;\bm\sfy^{(N)}/N)\rightarrow \del_x h(x,0),\quad
\varrho(x;\bm\sfz^{(N)}/N)\rightarrow \del_x h(x,T),
\end{align}
in distribution.
\item As $N$ goes to infinity, $\sfT/N$ goes to $T$. 
\end{enumerate}
We further assume that for $N$ large enough  the set of non-intersecting $\theta$-Bernoulli walks of length $\sfT$  from $\bm\sfy^{(N)}$ to $\bm\sfz^{(N)}$ is nonempty, namely $\cP(\bm\sfy^{(N)}, \bm\sfz^{(N)};\sfT)\neq\emptyset$ (see \eqref{e:pathset}
). 
\end{definition}
We remark that a height function $G\in\cP(\bm\sfy^{(N)}, \bm\sfz^{(N)};\sfT)$ is defined on $\bR\times [0, \sfT/N]$. If $\sfT/N<T$, we can extend $G$ to $\fR$ by setting $G(x,t)=G(x,T)$ for $t\geq T$. For any $H\in \Adm(\fR,h)$, the distance $\|G-H\|_\infty=\sup_{(x,t)\in \fR}|G(x,t)-H(x,t)|=\sup_{(x,t)\in \bR\times [0, \sfT/N]}|G(x,t)-H(x,t)|+\OO(|\sfT/N-T|)$, as $H$ and $G$ are Lipschitz. We can now state the main result of this article concerning the large deviations for the distribution of the height function under the law $\bP$ defined in \eqref{e:mdensity2}.

\begin{theorem}\label{t:main1} Let $\theta,T>0$ and $h\in \Adm^{\partial}_{\theta}(\fR)$. Consider  two  sequences  of particle configurations and time $(\bm\sfy^{(N)},\bm\sfz^{(N)},\sfT)$ which are $(h,\theta,T)$-admissible, and non-intersecting $\theta$-Bernoulli random walks \eqref{e:mdensity2} starting from $\bm\sfy^{(N)}$ at time $0$.

\begin{enumerate}
\item  Then, 
for any $H\in \Adm(\fR, h)$,  the following holds
\begin{equation}\label{limsupwldp}
\lim_{\varepsilon\rightarrow 0}\limsup_{N\rightarrow\infty}\frac{1}{ N^2}\ln \bP(\{G\in \cP(\bm\sfy^{(N)}, \bm\sfz^{(N)};\sfT): \|G-H\|_\infty\leq \varepsilon\})= -\frac{1}{\theta}
\cJ_{h}(H){-}T\ln 2,
\end{equation}
where 
$\cJ_{h}$ is the functional  defined for $H\in  \Adm(\fR, h)$ by
\begin{equation}\label{GRF}
\cJ_{h}(H)=-\iint_{\fR} \sigma(\nabla H(x,t))\rd x\rd t-\frac{1}{2}\left.\iint_{\fR} \ln|x-y|\rd h(x,t)\rd h(x,t)\right|_0^T.
\end{equation}
Moreover, \eqref{limsupwldp}  holds if we replace the $\limsup$ by a  $\liminf$.
\item $\cJ_{h}$ has compact level sets in $\Adm (\fR,h)$.
\item The law of the height functions conditioned to remain in $\cP(\bm\sfy^{(N)}, \bm\sfz^{(N)};\sfT)$ satisfies a large deviation principle with speed $N^{2}$ and good rate function $\cI_{h}(H)$ which is infinite outside $\Adm(\fR,h)$ and otherwise given by 
\begin{align}\label{e:Ih}
\cI_{h}(H)=\frac{1}{\theta}( \cJ_{h}(H)-\inf_{G\in \Adm(\fR,h)}\cJ_{h}(G))\,.
\end{align}
Moreover, $\cI_{h}$ is minimized at a unique minimizer $H^h\in  \Adm(\fR, h)$. 
\end{enumerate}
\end{theorem}

The surface tension $\sigma$ is the same as that of lozenge tiling. It has been proven in \cite[Theorem 7.5]{gorin2021lectures}, that $\cJ_h$ is lower semicontinuous over $\Adm(\fR, h)$ with the uniform topology. Since $\Adm(\fR, h)$ is compact, it follows that $\cJ_h$ has compact level sets in $\Adm(\fR,h)$.  The uniqueness of the minimizer for the variational principle \eqref{e:Ih} has been proven in \cite[Proposition 4.5]{10.1215/00127094-2010-004}. We remark that even though the rate function $\cJ_{h}(H)$ does not depend on $\theta$, the space $ \Adm(\fR, h)$ over which we minimize it depends on $\theta$, as $h\in \Adm^{\partial}_{\theta}(\fR)$.

Our second main result derives the large deviation asymptotics of skew Jack symmetric polynomials as the number of variables goes to infinity. As discussed after \eqref{e:Jackexp}, the main challenge lies in analyzing the sum of the weights $\cW(\sfp;\bmb)$ from \eqref{e:weightpb}. It can be interpreted as the partition function  of non-intersecting $\theta$-Bernoulli random walks from $\bm\sfy$ to $\bm\sfz$, with time dependent drift. Explicitly, the transition probability is given by
\begin{align}\label{e:mdensitydrift}
\bP^{\sfb}(\bm \sfx(\sft+1)=\bm\sfx+\bme|\bm\sfx(\sft)=\bm\sfx)=
\frac{1}{(1+\sfb_\sft)^N}
\prod_{1\leq i<j\leq N}\frac{(\sfx_i+\theta e_i)-(\sfx_j+\theta e_j)}{{\sfx_i}-{\sfx_j}}\prod_{1\leq i\leq N}\sfb_\sft^{e_i(\sft)},
\end{align}
for $0\leq \sft\leq \sfT-1$ and $\bme=(e_1, e_2,\cdots,e_N)\in \{0,1\}^N$. 
The large deviation of the aforementioned non-intersecting $\theta$-Bernoulli random walks with time-dependent drift can be deduced from \Cref{t:main1} through the application of Varadhan's lemma. An immediate consequence of this is the following result concerning the asymptotics of skew Jack polynomials following from the formula \eqref{e:Jackexp}. The proof for this result is provided in \Cref{s:proofmain2}. We next state our large deviation results under the law $\bP^{\sfb}$ and deduce the asymptotics of skew Jack polynomials in the related scaling. 

\begin{theorem}\label{t:main2}  Let $\theta,T>0$ and $h\in \Adm^{\partial}_{\theta}(\fR)$.
We consider a sequence of Young diagrams
\begin{align*}
\bm\la^{(N)}=(\la_1^{(N)}\geq \la_2^{(N)}\geq\cdots\geq \la_N^{(N)}),\quad
 \bm\mu^{(N)}=(\mu_1^{(N)}\geq \mu_2^{(N)}\geq\cdots\geq \mu_N^{(N)}), \quad N\geq 1,
 \end{align*}
 such that $(\bm\sfy^{(N)},\bm\sfz^{(N)})$ are given by \eqref{e:x2}. Assume $(\bm\sfy^{(N)},\bm\sfz^{(N)}, \sfT)$ are $(h,\theta,T)$-admissible and  for $0\leq \sft\leq \sfT-1$ take  $\sfb_{\sft}=e^{f(\sft/N)}$ for a  continuously differentiable function $f$.
 \begin{enumerate}
\item The asymptotics of the skew Jack polynomials are given by
 \begin{align*}
\lim_{N\rightarrow\infty}\frac{1}{N^2}\ln J_{(\bmla^{(N)})'\setminus(\bmmu^{(N)})'}(b_0, b_1, \cdots, b_{\sfT-1};\theta^{-1})= \frac{1}{\theta}\cJ^f_h,
\end{align*}
where 
\begin{equation}\label{lim1}\cJ^f_h
:= \sup_{H\in  \Adm(\fR,h)}\left\{ \iint_{\fR} \sigma(\nabla H(x,t))\rd x\rd t +\cF^{f}(H)\right\},\end{equation}
and the linear functional $\cF^f(H)$ is given by 
\begin{equation}\label{lim2}
\mathcal F^{f}(H)
:=-\iint_\fR f(s)\del_s H(y,s)\rd y\rd s.\end{equation}
\item The law of height functions under $\bP^{\sfb}$ defined in  \eqref{e:mdensitydrift}  conditioned to remain in $\cP(\bm\sfy^{(N)}, \bm\sfz^{(N)};\sfT)$ satisfies a large deviation principle with speed $N^{2}$ and good rate function $\cI^{f}_{h}(H)/\theta$ which is infinite outside $\Adm(\fR,h)$ and otherwise given by 
\begin{equation}\label{ratef}\cI^{f}_{h}(H)=-\left( \iint_\fR \sigma(\nabla H(x,t))\rd x\rd t   +\cF^{f}(H)\right)+\cJ^{f}_{h}\,.\end{equation}
Moreover, $\cI^f_h$ achieves its minimal value at a unique height function  $H^{f}_h$. 

\end{enumerate}

\end{theorem}

\subsection{Asymptotics for Skew Macdonald Polynomials}\label{s:Macdonaldintro}

The goal of this section is to study the asymptotics of the skew Macdonald polynomials $P_{\bmla\setminus \bmmu}(\bmx; q,t)$, which are symmetric functions in infinitely many variables $\bmx=(x_i)_{1\leq i\leq \infty}$ and parametrized by two Young diagrams $\bm\mu\subset \bmla$. We refer to  \Cref{s:Macdonald} for more detailed discussion on (skew) Macdonald polynomials. 

Fix $\theta>0$, and take $t=q^\theta$. Assume $\bmla,\bmmu$ are Young diagrams with at most $N$ rows, and take any $\bmb=(b_0,b_1,\cdots, b_{\sfT-1})$. Then the same as in \eqref{e:x2}, we identify $\bmmu, \bmla$ as particle configurations $\bm\sfy, \bm\sfz$.
For non-intersecting Bernoulli walks  $\sfp=\{\bm\sfx(\sft)\}_{0\leq \sft\leq \sfT}$ from $\bm\sfx(0)=\bm\sfy$ to $\bm\sfx(\sfT)=\bm\sfz$, we define their weights as
\begin{align}\label{e:weightpbmac}
\widetilde \cW(\sfp;\bmb)=\prod_{0\leq \sft\leq \sfT-1}\prod_{1\leq i<j\leq n}\frac{q^{\sfx_i(\sft)+\theta e_i(\sft)}-q^{\sfx_j(\sft)+\theta e_j(\sft)}}{q^{\sfx_i(\sft)}-q^{\sfx_j(\sft)}} \prod_{1\leq i\leq N}\sfb_\sft^{ e_i(\sft)}.
\end{align}
Then we can rewrite the skew Macdonald polynomial evaluated at $(b_0,b_1,\cdots, b_{\sfT-1})$ as follows 
\begin{align*}
P_{\bmla'/\bmmu'}(b_0, \cdots, b_{\sfT-1};t,q)
=\frac{P_{\bmmu}(1, t, t^2,\cdots, t^{N-1};q,t)}{P_{\bmla}(1, t, t^2,\cdots, t^{N-1};q,t)}\sum_{\sfp\in \cP(\bm\sfy;\bm\sfz;\sfT)}\widetilde\cW(\sfp;\bmb),
\end{align*}
where $\bm\la', \bmmu'$ are the transposes of $\bmla,\bmmu$.

Similarly to the skew Jack polynomials, there are explicit formulas for the Macdonald symmetric polynomials evaluated at the principal specialization, i.e. $P_{\bmmu}(1, t, t^2,\cdots, t^{N-1};q,t), P_{\bmla}(1, t, t^2,\cdots, t^{N-1};q,t)$,  making them amenable for asymptotic analysis. The sum of the weights $\widetilde \cW(\sfp;\bmb)$ can be interpreted as the partition function of non-intersecting $\theta$-Bernoulli random walks from $\bm\sfy$ to $\bm\sfz$, with the transition probability given by
\begin{align}\label{e:Macdonaldp}
\bP^{\sfb,q}(\bm\sfx(\sft+1)=\bm\sfx+\bme|\bm\sfx(\sft)=\bm\sfx)\propto \prod_{1\leq i<j\leq N}\frac{q^{\sfx_i+\theta e_i}-q^{\sfx_j+\theta e_j}}{q^{\sfx_i}-q^{\sfx_j}}\prod_{1\leq i\leq N}
b_\sft^{e_i}
.
\end{align}
for $0\leq \sft\leq \sfT-1$, and $\bme=(e_1, e_2,\cdots,e_N)\in \{0,1\}^N$. The above Markov process \eqref{e:Macdonaldp} is a special case of ascending Macdonald process as constructed in \cite{borodin2014macdonald}.


 The Macdonald processes \eqref{e:Macdonaldp} are different from non-intersecting $\theta$-Bernoulli random walks \eqref{e:mdensity2}, and also have singular interactions among adjacent particles. However, by taking their ratios, we observe the cancellation of singularities. This observation allows us to establish the large deviation principle for the Macdonald process through the application of Varadhan's lemma. The following result on the asymptotics of skew Macdonald polynomials is a consequence of the large deviation principle of the Macdonald process \eqref{e:Macdonaldp}. The proof is given in \Cref{s:proofmain2}. {
\begin{theorem}\label{t:main3} Let $\theta,T>0$ and $h\in \Adm^{\partial}_{\theta}(\fR)$.
We consider a sequence of Young diagrams $(\bm\la^{(N)},\bm\mu^{(N)})$ such that $(\bm\sfy^{(N)},\bm\sfz^{(N)})$ are given by \eqref{e:x2}. Assume $(\bm\sfy^{(N)},\bm\sfz^{(N)},\sfT)$ are $(h,\theta,T)$-admissible. For $0\leq \sft\leq \sfT-1$, we  take  $\sfb_{\sft}=e^{f(\sft/N)}$ for a  continuously differentiable function $f$. Moreover,  
we take $t=q^\theta, q=e^{\kappa/N}$ for some 
$\kappa<0$.  \begin{enumerate}
\item The asymptotics of the skew Macdonald polynomials are given by
 \begin{align*}
&\phantom{{}={}}\lim_{N\rightarrow\infty}\frac{1}{N^2}\ln P_{(\bmla^{(N)})'\setminus(\bmmu^{(N)})'}(b_0, b_1, \cdots, b_{\sfT-1};q,t)=\frac{1}{\theta}\cJ^f_h,
\end{align*}
where 
$\cJ^f_h$ is defined in \eqref{lim1}.
\item The law of height functions under $\bP^{\sfb,q}$ defined in  \eqref{e:Macdonaldp} conditioned to remain in $\cP(\bm\sfy^{(N)}, \bm\sfz^{(N)};\sfT)$ satisfies a large deviation principle with speed $N^{2}$ and good rate function $\cI^{f}_{h}(H)/\theta$  which is infinite outside $\Adm(\fR,h)$ and otherwise given by \eqref{ratef}. 
Moreover, $\cI^f_h$ achieves its minimal value at a unique height function  $H^{f}_h$. 
\end{enumerate}
\end{theorem}}
We remark that the rate function surprisingly does not depend on $\kappa$. This indicates that $q$ is converging toward one too rapidly to induce changes at the large $N$ limit. However, other scalings seem difficult to deduce by Varadhan's lemma and require other kinds of large deviations estimates. 

\subsection{Related works on the Asymptotics of Symmetric Polynomials}

The asymptotic behavior of symmetric polynomials can be studied under other scaling regimes. One intensively studied regime is to take limit along a sequence of Vershik-Kerov partitions, where the size and rows of the partitions all grow linearly in $N$. In this way, the normalized Schur functions approximate a character of the infinite unitary group. The asymptotics of normalized Schur functions in this scaling regime were originally derived by Vershik and Kerov in \cite{vershik1982characters}. Generalizations of these results to the asymptotics of Jack and Macdonald polynomials under the same scaling can be found in \cite{10.1155/S1073792898000403, cuenca2018pieri, cuenca2018asymptotic}.

While the Vershik-Kerov partitions involve thin partitions, this paper explores a different scaling regime. Specifically, we consider the limit along a sequence of partitions where the number of rows and columns grows linearly in $N$,  but the size grows quadratically in $N^2$. Gorin and Panova's work in \cite{10.1214/14-AOP955} presents an example within this regime. They obtained the asymptotics of Schur polynomials with all but finitely many parameters set to one, see also \cite{GuMa05,GuHu22}. These asymptotic results played an important role in the study of extreme characters of the infinite-dimensional unitary group \cite{10.1214/14-AOP955, bufetov2015representations}, as well as applications in random lozenge tilings \cite{bufetov2018fluctuations,10.1215/00127094-2019-0023} and the six-vertex model \cite{10.1214/14-AOP955,gorin2023boundary}.

%
%
%

%
%
%
%

The structure constants of symmetric polynomials, such as Kostka number, Kronecker coefficients and Littlewood-Richardson coefficients are fundamental quantities in algebraic combinatorics. Despite the absence of explicit formulas, the asymptotics of specific extreme Kronecker and Littlewood-Richardson coefficients have been derived in \cite{pak2019largest, stanley2011enumerative,pak2020bounds,pak2023durfee}.   Using the asymptotics of Schur symmetric polynomials, the large deviation asymptotics of Kostka numbers and the large deviation upper bound for the Littlewood-Richardson coefficients have been obtained in \cite{belinschi2022large}. These results were generalized in \cite{huang2023asymptotics} to the asymptotic behavior
of weight multiplicities of irreducible representations of compact or complex simple Lie algebras in the limit of large rank. The asymptotics of Jack polynomials and Macdonald polynomials should allow to address similar questions for the large deviations of their structure constants.

\subsection{Ideas of the proofs }
In the continuous setting, the asymptotics of the Harish-Chandra-Itzykson-Zuber integral has been discovered through studying the large deviations of Dyson's Brownian motion by the first named author and Zeitouni \cite{MR1883414,MR2091363,MR2034487}. Subsequently, analogous findings were extended to the rectangular spherical integral and Generalized Bessel Functions in \cite{huang2023asymptotics,guionnet2023asymptotics}.  Large deviations are established thanks to the standard tilting argument by martingales.
However, these martingales are constructed thanks to  stochastic calculus and the  specific structures of Dyson's Brownian motion.

In this paper, we focus on the discrete setting, where tools from stochastic calculus are not available. We 
 summarize the main  ideas of the proof of  our main Theorems,  \Cref{t:main1}.
\Cref{t:main2} and \Cref{t:main3}, can then be deduced by Varadhan's lemma.
 In the special case when $\theta=1$,  the non-intersecting $\theta$-Bernoulli random walks \eqref{e:mdensity2} are equivalent to random Lozenge tiling of strip domains, where the variational principle of the height function has been proven in \cite{cohn2001variational}. In fact their results apply to random tilings of more general domains. The proof of the variational principle for  tilings in \cite{cohn2001variational} relies on the exact computation of the partition function for tilings on a large torus, utilizing Kasteleyn's matrix. Unfortunately, this method is not applicable in our context when $\theta\neq 1$. 
Instead, our proof relies on the recently introduced dynamical loop equations by the second author and Gorin \cite{gorin2022dynamical}, based on Nekrasov's equations \cite{MR3668648,Nek_PS,Nekrasov}. These equations have proven effective for analyzing the fluctuations of large families of two-dimensional interacting particle systems in both discrete and continuous settings. Examples include nonintersecting Bernoulli/Poisson random walks \cite{gorin2019universality,konig2002non, huang2017beta}, $\beta$-corner processes \cite{gorin2015multilevel, borodin2015general}, measures on Gelfand-Tsetlin patterns \cite{bufetov2018asymptotics,petrov2015asymptotics,petrov2014asymptotics}, and Macdonald processes \cite{borodin2014macdonald}. In particular, a version of dynamical loop equations has been derived for the non-intersecting $\theta$-Bernoulli random walks.

We establish the large deviation principle using Cram{\'e}r's method, by tilting the measure of nonintersecting $\theta$-Bernoulli walk using exponential martingales. A key point is that these martingales can be identified as smooth functions of the height functions thanks to  dynamical loop equations.
Our proof is divided into an upper bound and a lower bound. The normalization constant for these exponential martingales can be determined using dynamical loop equations. Subsequently, the large deviation upper bound can be obtained straightforwardly by applying Markov's inequality. The measure tilted by the exponential martingale coincides with the drifted non-intersecting $\theta$-Bernoulli random walk. Utilizing the dynamical loop equation once again, we establish a limit shape theorem for the drifted non-intersecting $\theta$-Bernoulli random walk. As a consequence, for any targeting measure process, we can construct an exponential martingale such that the tilted measure concentrates around it. This provides the large deviation lower bound.

In the above outlined proof, we utilize dynamical loop equations to investigate nonintersecting $\theta$-Bernoulli walks with drift. However, two challenges hinder their direct application.  Firstly, dynamical loop equations require that the drift is analytic. To overcome this, we convolve the targeting measure process with a small Cauchy distribution (which is reminiscent of the strategy followed in  \cite{MR1883414}). This convolution enables the analytical extension of the density to a strip region around the real axis, allowing for a similar extension of the drift terms. As a trade-off, we are required to analyze measure-valued processes supported on the entire real axis. This challenge is addressed through precise truncations and the utilization of the explicit form of the rate function.

Secondly and more critically, dynamical loop equations require certain non-criticality conditions. Specifically, the particle system cannot be too dense, meaning the density cannot be too close to one. To tackle this challenge, rather than studying the nonintersecting $\theta$-Bernoulli walks as a whole, we partition space-time into small regions of parallelogram shape. Within each region, we ensure that both the density and velocity remain nearly constant. In each region, the system resembles nonintersecting $\theta$-Bernoulli walks but with a reduced number of particles and possibly left and right boundaries. If the density and velocity are non-extremal, we can establish a large deviation principle using the dynamical loop equation as outlined above. However, if the density or velocity are extremal, we prove a large deviation upper bound by directly analyzing the walk. Importantly, it is observed that the contribution from these regions is negligible. The original nonintersecting $\theta$-Bernoulli walks have long-range pairwise interactions. After partitioning, particles from different regions also interact with each other. Fortunately, these interactions occur between particles which are far from each other and are not singular, which can be approximated as smooth weights. In this manner, we can integrate all regions together and derive the variational principle for the original nonintersecting $\theta$-Bernoulli walks.

 Loop equations and Nekrasov's equations have been key to establish central limit theorems \cite{MR1487983,GGG,borot-guionnet2, MR3010191, bourgade2021optimal}. To our knowledge, it is the first time it is used to derive large deviations.

\subsection{Organization of the paper}
In \Cref{s:setup}, we collect various facts about height functions, Vandermonde determinants and the free entropy. They will be used repeatedly in the rest of this paper. \Cref{s:outline} provides an overview of the proof for \Cref{t:main1}, which relies on both a large deviation upper bound (\Cref{p:LDP}) and a corresponding lower bound (\Cref{p:LDPlow}) for non-intersecting $\theta$-Bernoulli random walks with a constant slope.
The proof for \Cref{p:LDP} is detailed in \Cref{s:upB}, while the proof for \Cref{p:LDPlow} is presented in \Cref{s:lowB}. Moving forward, \Cref{s:Jprocess} contains the proof for \Cref{t:main2}, on the asymptotics of skew Jack polynomials. Similarly, the proof for \Cref{t:main3}, which explores the asymptotics of skew Macdonald polynomials, is outlined in \Cref{s:Macdonald}.
Finally in \Cref{s:DynamicalLoopE}, we collect the results on dynamical loop equations from \cite{gorin2022dynamical}.

\subsection{Notations}
We use $\bmla,\bmmu$ to represent Young diagrams. 
For macro particle locations we use mathsf letters: $\sfx_i(\sft),  \bm\sfx(\sft),  \bm\sfy, \bm\sfx$, where time $\sft$ ranges from $0$ to $\sfT$.  
For micro particle locations  we use standard letters $x_i(t), \bmx(t), \bmy, \bmz$, where the time $t$ ranges from $0$ to $T$. We use $H$ and $H^*$ for arbitrary limiting Height functions, , that is an element of $\Adm(\fR,h)$ that arises as the limit of height functions. $\widetilde H$ stands for  the smoothed version of $H$. And we use
$\cH$ for the associated height functions  of non-intersecting random walks.

For two quantities $X$ and $Y$ depending on $N$, 
we write that $X = \OO(Y )$ or $X\lesssim Y$ if there exists some universal constant $C>0$ such
that $|X| \leq C Y$ . We write $X = \oo(Y )$, or $X \ll Y$ if the ratio $|X|/Y\rightarrow \infty$ as $N$ goes to infinity. We write
$X\asymp Y$ if there exists a universal constant $C>0$ such that $ Y/C \leq |X| \leq  C Y$. We denote $\qq{a,b} = [a,b]\cap\bZ$ and $\qq{n} = \qq{1,n}$. 



 \subsection*{Acknowledgements.}
 The research of J.H. is supported by  NSF grant DMS-2331096 and DMS-2337795, and the Sloan research award. A.G.  has received funding from the European Research Council (ERC) under the European Union
Horizon 2020 research and innovation program (grant agreement No. 884584).
The authors want to thank Ofer Zeitouni for helpful discussions at the early stage of this project.


\section{Setup and Preliminary Results}\label{s:setup}
In this section, we collect various facts about height functions, the Vandermonde determinant, and free entropy. These facts will be used repeatedly throughout the rest of this paper. { In the rest of this paper, for simplicity of notations, we will  assume that $T=\sfT/N$. The more general case that $\sfT/N$ converges to $T$ as in \Cref{assume} can be proven in the same way, by noticing that the rate functions in \Cref{t:main1}, \Cref{t:main2}, \Cref{t:main3} are all continuous in $T$.}

\subsection{Approximating Height Function}\label{s:setup0}
For any height function $H^*\in \Adm(\fR;h)$,
the following lemma states that if $\cP(\bm\sfy^{(N)}, \bm\sfz^{(N)};\sfT)\neq\emptyset$, then there exists a non-interesecting $\theta$-Bernoulli walk from $\bm\sfy^{(N)}$ to $\bm\sfz^{(N)}$, and its height function is close to $H^*$. The proof is a consequence of 
 a careful discretization of $H^*$, and we postpone it to \Cref{s:setupproof}.

\begin{lemma}\label{e:constructH}  Let $\theta,T>0$ and $h\in \Adm^{\partial}_{\theta}(\fR)$. Consider  two  sequences  of particle configurations and time $\sfT$: $(\bm\sfy^{(N)},\bm\sfz^{(N)}, \sfT)$ which are $(h,\theta,T)$-admissible as defined in  \Cref{assume}. 
 For any $\varepsilon>0$, 
and any $H^*\in \Adm (\mathfrak{R}; h)$, there exists a non-intersecting Bernoulli paths $\{\bmx(t)\}_{0\leq t\leq T}$ from $\bmx(0)=\bm\sfy^{(N)}/N$ to $\bmx(T)=\bm\sfz^{(N)}/N$ with height function $\cH$, such that on $\fR$
\begin{align}\label{e:heightcloseH}
\|H^*-\cH\|_\infty\leq \varepsilon,
\end{align}
provided $N$ is large enough. 
\end{lemma}

We introduce the following $\ell$-mesh, and \Cref{l:Lipschitz} states that on most of the $\ell$-mesh, $H^*$ has an approximate linear approximation. 
\begin{definition}[$\ell$-mesh]\label{d:lemesh}
Adopt \Cref{a:boundaryheight}, there exists a large constant $A>0$, such that for any $H\in \Adm(\fR;h)$, 
\begin{align*}
H(x,t)=0,\quad x\leq -A, \quad H(x,t)=\theta, \quad x\geq A.
\end{align*}

\begin{figure}
	\begin{center}
	 \includegraphics[scale=8]{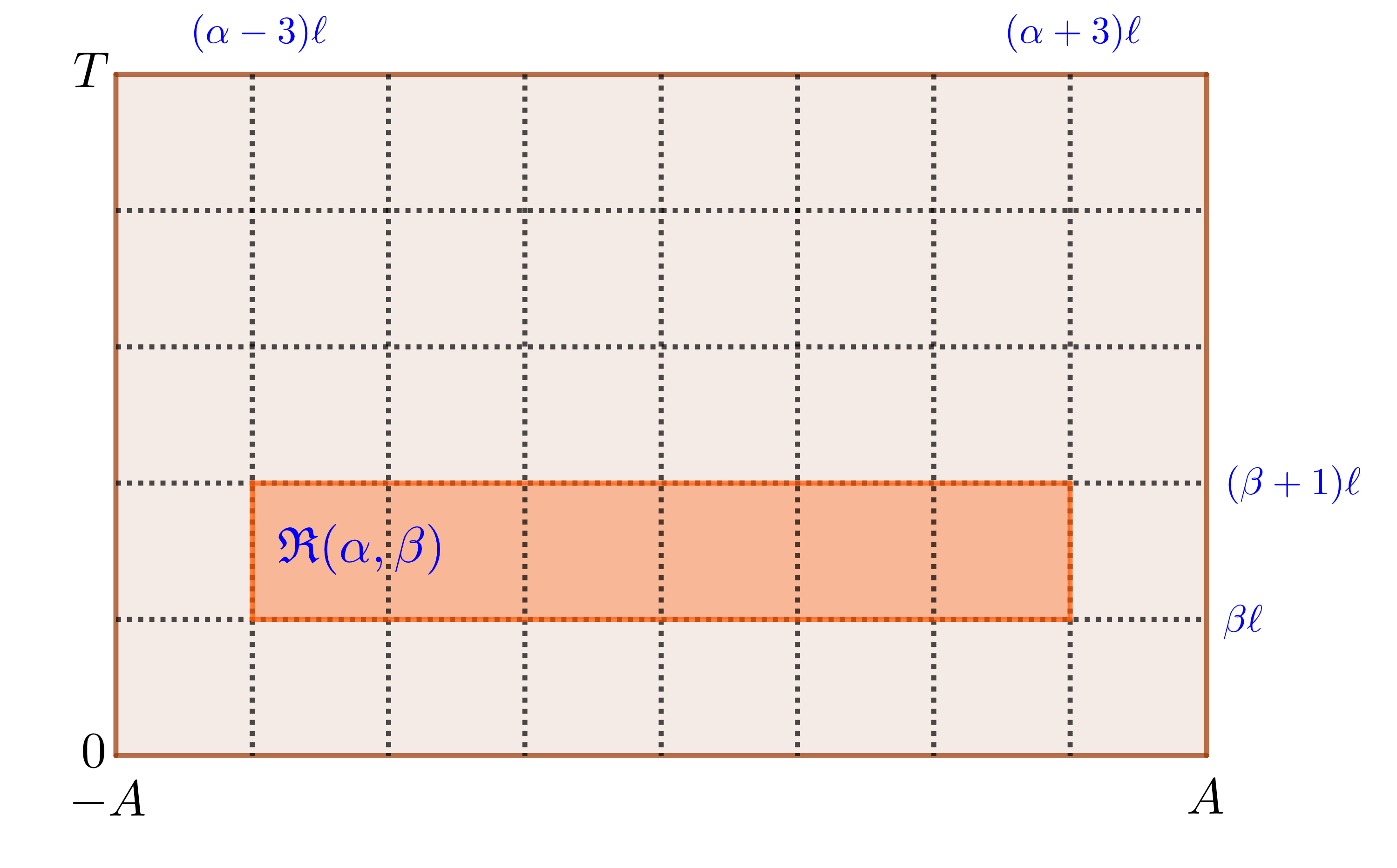}
	 \caption{Shown above is the $\ell$-mesh.}
	 \label{f:domain_R}
	 \end{center}
	 \end{figure}

 Take a small $\ell>0$ such that $A/\ell, T/\ell\in \bZ$. We cover the region $[-A,A]\times[0,T]$ by rectangles of size $6\ell\times \ell$ (see \Cref{f:domain_R}):
 \begin{align*}
[-A,A]\times[0,T]=\cup \fR(\al, \beta),\quad \fR(\al,\beta)=[(\al-3)\ell, (\al+3)\ell]\times [\beta\ell, (\beta+1)\ell], 
\end{align*}
where $\alpha,\beta$ are integer numbers so that 
${-A/\ell\leq \al\leq A/\ell,\quad 0\leq \beta\leq T/\ell}$.

\end{definition}

\begin{lemma}\label{l:Lipschitz}
Let $H$ be an asymptotic height function, namely a given element of $Adm(\fR,h)$, and let $\varepsilon_0>0$. If $\ell$ is sufficiently small then on at least a $1-\varepsilon_0$ fraction of the rectangles $\fR(\al, \beta)$ from the $\ell$-mesh (in \Cref{d:lemesh}) we have the following two properties 
\begin{enumerate}
\item On $\fR(\al,\beta)$, $H$ has a linear approximation within an error bounded by $\ell \varepsilon_0$, namely there exists $(\varrho, -\varrho v)\in \overline{\cT}$ such that for $(x,t)\in \fR(\al,\beta)$
\begin{align*}
|H(x,t)-(H(\al\ell,\beta\ell)+\varrho (x-\al\ell)-\varrho v(t-\beta\ell))|\leq \varepsilon_0\ell.
\end{align*}
\item For at least a $1-\varepsilon_0$ fraction (in
measure) of the points $x$ of $\fR(\al,\beta)$, the gradient $\nabla H(x)$ exists and is within $\varepsilon_0$ of $(\varrho, -\varrho v)$. 
\end{enumerate}
\end{lemma}

\begin{proof}
The proof is the same as in {\cite[Lemma 2.2]{cohn2001variational}}, so we omit.
\end{proof}

\Cref{l:Lipschitz} states that any height function $H\in \Adm(\fR, h)$ locally can be approximated by height functions with constant slope. We introduce the following height functions with constant slope. 
\begin{definition}\label{d:constantslope}
Given a slope $(\varrho, -\varrho v)\in \overline{\cT}$, we construct the constant slope height function $\cA$ on $\bR\times [0,\ell]$:
\begin{align}\label{e:defH*}
&\cA(x,t)= (x-tv)\varrho, \quad (x,t)\in \fP;\quad  \cA(x,t)=0,\quad x\leq tv; \quad \cA(x,t)=\ell \varrho, \quad x\geq \ell+tv.
\end{align} 
where $\fP$ is the parallelogram region:
\begin{align}\label{e:defp}
&\fP:=\{x,t\in \bR^2: 0\leq t\leq\ell, \quad tv\leq x\leq \ell+tv\}.
\end{align}
We say $\cA$ is the height function with slope given by $(\varrho, -\varrho v)$.
\end{definition}

\subsection{Vandermonde Determinant and Free Entropy}
\Cref{l:sumV} and \Cref{l:VVbb} concern certain useful identity and estimates of Vandermonde determinants. 

\begin{lemma}\label{l:sumV}
For any particle configuration $\bm\sfx\in \bW^n_\theta$, and $0\leq k\leq n$, the following holds
\begin{align}\label{e:sumV}
\sum_{e_1+e_2+\cdots+e_n=k}\frac{V(\bm\sfx+\theta \bme)}{V(\bm\sfx )}={n\choose k}.
\end{align}
\end{lemma}
\begin{proof}
Since Vandermonde determinant is anti-symmetric, it follows that
\begin{align}\label{e:sumV2}
\sum_{e_1+e_2+\cdots+e_n=k}V(\bm\sfx+\theta \bme),
\end{align}
as a function of $\sfx_1, \sfx_2, \cdots, \sfx_n$, is anti-symmetric. Therefore, it has the Vandermonde determinant $V(\bm\sfx)$ as a factor. We notice that for every $\theta$, the coefficient of the monomial $\sfx_1^{n-1} \sfx_2^{n-2} \cdots \sfx_{n-1}$ of degree $n(n-1)/2$ in $V(\bm\sfx+\theta \bme)$ is $1$. The claim \eqref{e:sumV} follows from that  the coefficient of $\sfx_1^{n-1} \sfx_2^{n-2} \cdots \sfx_{n-1}$ in \eqref{e:sumV2} is ${n\choose k}$, and in $V(\bm\sfx)$ is $1$.
\end{proof}

\begin{lemma}\label{l:VVbb}
There exists a constant $C>0$. For any particle configuration $\bm\sfx\in \bW^n_\theta$, and $\bm\sfy=\bm\sfx+\bme\in \bW^n_\theta$ (this is equivalent to that $V(\bm\sfx+\theta\bme)\neq 0$), 
the following hold 
\begin{align}\label{e:VVbb}
C^{-n}\leq \frac{V(\bm\sfx+\theta \bme)}{V(\bm\sfx)}\leq 2^n.
\end{align}
\end{lemma}
\begin{proof}
The upper bound in \eqref{e:VVbb} follows from \Cref{l:sumV} which implies
\begin{align}\label{e:ubb}
\sum_\bme \frac{V(\bm\sfx+\theta \bme)}{V(\bm\sfx)}=2^n,
\end{align}
and the fact that all terms in the left hand side are non-negative.
For the lower bound, we first show that there exists a constant $C=C(\theta)$ such that for every configuration $\bm\sfx$, 
\begin{align}\label{e:localeq1}
\left|\ln \frac{(\sfx_i+\theta e_i)-(\sfx_j+\theta e_j)}{{\sfx_i}-{\sfx_j}}
-\theta \ln \frac{(\sfx_i+ e_i)-(\sfx_j+ e_j)}{{\sfx_i}-{\sfx_j}}
\right|\leq C\frac{(e_i-e_j)^2}{({\sfx_i}-{\sfx_j})^2}.
\end{align}
If $e_i-e_j=0$, there is nothing to prove, the lefthand side is $0$. If $\sfx_i-\sfx_j=(j-i)\theta$, i.e. the particle configuration on $\{\sfx_i, \sfx_{i+1},\cdots, \sfx_{j}\}$ is tightly packed. The quantity $V(\bm\sfx+\theta\bme)$ is nonzero only if $e_i-e_j\in \{0,1\}$. Thus,
\begin{align*}
\frac{e_i-e_j}{\sfx_i-\sfx_j}\in \left\{0, \frac{1}{\theta(j-i)}\right\}, \quad\mbox{ so that } 0\leq \frac{e_i-e_j}{\sfx_i-\sfx_j}\leq \frac{1}{\theta}.
\end{align*}
Otherwise, $\sfx_i-\sfx_j\geq 1+\theta$, and we have 
\begin{align*}
\left|\frac{e_i-e_j}{\sfx_i-\sfx_j}\right|\leq \frac{1}{1+\theta}.
\end{align*}
The statement \eqref{e:localeq1} follows from the inequality below
\begin{align*}
|\ln(1+\theta x)-\theta\ln (1+x)|\leq Cx^2,\quad x\in [-(1+\theta)^{-1}, \theta^{-1}],
\end{align*}
which follows from the fact  that for $x\in [-(1+\theta)^{-1}, \theta^{-1}]$
\begin{align*}
|\del_x^2(\ln(1+\theta x)-\theta\ln (1+x))|=\theta(1-\theta)\frac{|1-\theta x^2|}{(1+x)^2(1+\theta x)^2}\leq \frac{(1+\theta)^5(1-\theta)}{\theta \min\{\theta,1\}}.
\end{align*}
Summing over $i,j$ in \eqref{e:localeq1}, we deduce
\begin{align}\begin{split}\label{e:localeq2}
\ln \frac{V(\bm\sfx+\theta \bme)}{V(\bm\sfx)}
&\geq \theta \ln \left(\frac{V(\bm\sfy)}{V(\bm\sfx)}\right)-C\sum_{i<j}\frac{4}{(\sfx_i-\sfx_j)^2}\\
&\geq \theta\ln\frac{V(\bm\sfy)}{V(\bm\sfx)} -C\sum_{1\leq i<j\leq n}\frac{4}{(\theta(i-j))^2}\geq \theta\ln\frac{V(\bm\sfy)}{V(\bm\sfx)} -C' n.
\end{split}\end{align}
By the same argument as for \eqref{e:localeq2} (replacing $\bme$ by $-\bme$), we also have that $\bm\sfx=\bm\sfy-\bme$ satisfies
\begin{align}\label{e:localeq3}
n\ln2 \geq \ln \frac{V(\bm\sfy-\theta \bme)}{V(\bm\sfy)}
&\geq  \theta  \ln \left(\frac{V(\bm\sfy-\bme)}{V(\bm\sfy)}\right)-C' n=
\theta \ln \left(\frac{V(\bm\sfx)}{V(\bm\sfy)}\right)-C' n,
\end{align}
where the upper bound by $n\ln 2$ comes from \eqref{e:VVbb}. 
Combining \eqref{e:localeq2} and \eqref{e:localeq3}, we obtain the lower bound in \eqref{e:VVbb}
\begin{align*}
\ln \frac{V(\bm\sfx+\theta \bme)}{V(\bm\sfx)}\geq \theta\ln\frac{V(\bm\sfy)}{V(\bm\sfx)} -C' n\geq -n\ln2-2C'n=-(\ln 2+2C')n.
\end{align*}

\end{proof}

In the following lemma we collect certain estimates on the limit of the Vandermonde determinants, and on the free entropy. 
\begin{lemma}\label{l:freeentropy}
Fix $\kappa< 0$. There exists a constant $C=C(\kappa)>0$, the following holds. 
For any particle configuration $\bm\sfx\in \bW_\theta^N$, let $\bmx\in \bm\sfx/N$. Assume $\bmx$ is supported on $[-A,A]$, and we denote its empirical density as $\varrho(x; \bmx)$, then 
\begin{align}\begin{split}\label{e:vandlimit}
&\frac{\theta^2}{N^2}\sum_{i<j}\ln(x_i-x_j) 
=\frac{1}{2}\iint_{\bR^2}\ln |x-y|\varrho(x;\bmx)\varrho(y;\bmx)\rd x\rd y +\OO\left(\frac{\ln N}{N}\right),\\
&\frac{\theta^2}{N^2}\sum_{i<j}\ln(1- e^{\kappa(x_i-x_j)})
=\frac{1}{2}\iint_{\bR^2}\ln(1- e^{\kappa|x-y|})\varrho(x;\bmx)\varrho(y;\bmx)\rd x\rd y +\OO\left(\frac{\ln N}{N}\right).
\end{split}\end{align}
For any two height functions $h_1, h_2$ such that $\|h_1-h_2\|_\infty\leq \varepsilon$ with $0\leq \del_x h_1, \del_x h_2\leq 1$ and their supports $\supp(\del_x h_1), \supp(\del_x h_2)\in [-A,A]$, then
\begin{align}\begin{split}\label{e:Hdiffbound}
&\left|\iint_{\bR^2}\ln|x-y|\rd h_1(x)\rd h_1(y)-\iint_{\bR^2}\ln|x-y|\rd h_2(x)\rd h_2(y)\right|\leq C\varepsilon,\\
&\left|\iint_{\bR^2}\ln(1- e^{\kappa|x-y|})\rd h_1(x)\rd h_1(y)-\iint_{\bR^2}\ln(1- e^{\kappa|x-y|})\rd h_2(x)\rd h_2(y)\right|\leq C\varepsilon.
\end{split}\end{align}
\end{lemma}

\begin{proof}
We will only prove the first statements in \eqref{e:vandlimit} and \eqref{e:Hdiffbound}, and the second statements follow from the same arguments. For any $i<j-1$, denote $x_i-x_j=r\geq 2\theta/N$, we have 
\begin{align*}
&\phantom{{}={}}\int_{x_i}^{x_i+\theta/N}\int_{x_j}^{x_j+\theta/N}  \ln |x-y|\rd x\rd y
=\int_{0}^{\theta/N}\int_{0}^{\theta/N}  \ln |r+x-y|\rd x\rd y\\
&=\int_{0}^{\theta/N}  (\ln(r-\tau)+\ln (r+\tau))\left(\frac{\theta}{N}-\tau\right)\rd \tau
=\frac{\theta^2}{N^2}\ln r+\int_{0}^{\theta/N}  \ln\left(1-\frac{\tau^2}{r^2}\right)\left(\frac{\theta}{N}-\tau\right)\rd \tau\\
&=\frac{\theta^2}{N^2}\ln r+\OO\left(\frac{1}{N^4r^2}\right).
\end{align*}
By summing over all the pairs $i<j-1$, and bound $\ln(x_i-x_{i+1})=\OO(\ln N)$, we get
\begin{align*}
\frac{\theta^2}{N^2}\sum_{i<j}\ln(x_i-x_j)
&=\sum_{i,j}\frac{1}{2}\int_{x_i}^{x_i+\theta/N}\int_{x_j}^{x_j+\theta/N}  \ln |x-y|\rd x\rd y +\OO\left(\frac{\ln N}{N}\right)\\
&=\frac{1}{2}\iint_{\bR^2}\ln |x-y|\varrho(x;\bmx)\varrho(y;\bmx)\rd x\rd y +\OO\left(\frac{\ln N}{N}\right). 
\end{align*}
This gives \eqref{e:vandlimit}.
The first statement in \eqref{e:Hdiffbound} follows from
\begin{align*}
\cal I:=\left|\iint_{\bR^2}\ln|x-y|\rd (h_1(x)-h_2(x))\rd h_1(y)\right|\leq C\varepsilon. 
\end{align*}
Denoting $\varrho_1(y)=\del_y h_1(y)$, we deduce by an integration by part
\begin{align*}
&\cal I=\left|\int \PV \int_\bR \frac{\varrho_1(y)}{x-y}\rd y (h_1(x)-h_2(x))\rd x\right|
\lesssim \varepsilon\int_{-A}^A |\Hib(\varrho_1) (y)|\rd y \\
&\leq \varepsilon\left(2A\int_{-A}^A |\Hib(\varrho_1) (y)|^2\rd y \right)^{1/2}
\leq   \varepsilon\left(2A\int_{\bR}\varrho_1^2 \rd y\right)^{1/2}=\varepsilon\sqrt{2A},
\end{align*}
where we used that the Hilbert transform preserves $L_2$ norm, and $\varrho_1\leq 1$.
\end{proof}

\section{Outline of the proof of  Theorem \ref{t:main1}}
\label{s:outline}
 In this section, we prove  Theorem \ref{t:main1} assuming it holds in the case where $H$ has constant slope. The proof of the later is derived in the next two sections.
We recall the boundary height function $h$ from \Cref{a:boundaryheight}, and the sequences of particle configurations $\bm\sfy^{(N)}, \bm\sfz^{(N)}$ from \Cref{t:main1}.
We recall the set of non-interesecting $\theta$-Bernoulli walks from $\bm\sfy^{(N)}$ to $\bm\sfz^{(N)}$ from \eqref{e:pathset}
\begin{align*}
\cP(\bm\sfy^{(N)}, \bm\sfz^{(N)};\sfT)
=\{\sfp=\{\bm{\sfx}(\sft)\}_{0\leq \sft\leq \sfT}: \bm\sfx_0=\bm\sfy^{(N)},\bm\sfx_\sfT=\bm\sfz^{(N)}\}.
\end{align*}
Fix an asymptotic height function $H^*\in \Adm(\fR, h)$, and  sufficiently small $\varepsilon_0, \ell>0$ such that \Cref{l:Lipschitz} holds and $\sfL=\ell N\in \bN$. We introduce some parameters $0< \varepsilon, \delta, \xi\ll \ell$ and $0<\zeta \ll 1$: 
\begin{align}\label{e:introepsilon}
\varepsilon=\varepsilon_0\ell, \quad \delta\ll \zeta \ell, \quad \varepsilon\ll \zeta \ell, \quad \xi\ll \zeta \ell\quad \varepsilon/\delta\ll  \zeta^2, \quad \delta\ln(\ell/\delta)\ll \zeta \xi.
\end{align}
We denote and  $\bB_\varepsilon(H^*)$ the set of non-intersecting $\theta$-Bernoulli walks with height function close to $H^*$
\begin{align*}
\bB_\varepsilon(H^*)=\{\sfp=\{\bm\sfx(\sft)\}_{0\leq \sft\leq \sfT}\in \cP(\bm\sfy^{(N)}, \bm\sfz^{(N)};\sfT) \text{ with height function $\cH$}: \|\cH-H^*\|_\infty\leq \varepsilon\}.
\end{align*}

\subsection{Upper Bound}
To prove the large deviation upper bound for $\bP(\bB_\varepsilon(H^*))$, we partition $\bB_\varepsilon(H^*)$ according to the possible particle configurations at times $\bN\sfL$:
\begin{align*}
\bB_\varepsilon(H^*:\{\bm\sfy(\sft)\}_{\sft\in\{0, \sfL, 2\sfL, \cdots, \sfT\}})
:=\{\sfp=\{\bm\sfx(\sft)\}_{0\leq \sft\leq \sfT}\in \bB_\varepsilon(H^*): \bm\sfx(\sft)=\bm\sfy(\sft)\text{ for } \sft\in\{0, \sfL, 2\sfL, \cdots, \sfT\}\},
\end{align*}
where $\bm\sfy(0)=\bm\sfy^{(N)}$ and $\bm\sfy(\sfT)=\bm\sfz^{(N)}$. 
By our assumption that $|\sfy_i^{(N)}|, |\sfz_i^{(N)}|\leq CN$, then there exists $A>0$ (say $A=C+\sfT$) such that the particle configuration $\bm\sfy(\sft)$ is supported on $ [-AN, AN]$ for all $0\leq \sft \leq \sfT$. The number of total choices of $\{\bm\sfy(\sft)\}_{\sft\in\{0, \sfL, 2\sfL, \cdots, \sfT\}}$ is given by
\begin{align*}
{2AN\choose N}^{T/\ell}\leq 2^{(2TA/\ell)N}=e^{\OO(N)},
\end{align*}
which is negligible. In the rest of this section, we will fix the particle configurations $\{\bm\sfy(\sft)\}_{\sft\in\{0, \sfL, 2\sfL, \cdots, \sfT\}}$.

We  fix integer numbers $\alpha, \beta$, and recall the  $\ell$-mesh and rectangle $\fR(\al,\beta)=[(\al-3)\ell, (\al+3)\ell]\times [\beta\ell, (\beta+1)\ell]$ from \Cref{d:lemesh}.  We consider the restriction of the particle configuration $\bmy(\beta \ell )$ on the interval $[\al \ell , (\al+1)\ell )$
\begin{align*}
y_{i-1}(\beta \ell)<\al\ell\leq y_i(\beta \ell )<\cdots<y_j(\beta \ell )<(\al+1)\ell \leq y_{j+1}(\beta \ell ).
\end{align*}
We denote the index set of the above restricted particle configuration as $I(\alpha, \beta)=\{i,i+1, i+2,\cdots, j\}$ and denote $n:=|I(\al,\beta)|=j-i+1$.
For any $\sfp\in \bB_\varepsilon(H^*:\{\bm\sfy(\sft)\}_{\sft\in\{0, \sfL, 2\sfL, \cdots, \sfT\}})$, its restriction on the index set $I(\al,\beta)$ and time interval $[\beta \ell, (\beta+1)\ell]$, $\{{\sf\bmx}_{k}(t)\}_{k\in I(\al, \beta), t\in [\beta\ell,(\beta+1)\ell]}$ forms an $n$-particle nonintersecting Bernoulli random walk inside the rectangle $\fR(\al, \beta)$. The next lemma states that its height function is close to a height function with constant slope. 

\begin{lemma}\label{l:slope}
There exists a large finite constant $C>1$ so that  the following holds. Assume that on $\fR(\al, \beta)$, the height function $H^*$ has a linear approximation with slope $(\varrho, -\varrho v)\in \cT$ and error $\varepsilon_0=\varepsilon/\ell$ in the sense of \Cref{l:Lipschitz}. 
For any  $\{\bm\sfx(\sft)\}_{0\leq \sft\leq \sfT}\in \bB_\varepsilon(H^*:\{\bm\sfy(\sft)\}_{\sft\in\{0, \sfL, 2\sfL, \cdots, \sfT\}})$, we denote the shifted and rescaled particle configuration $\{z_k(t)\}_{1\leq k\leq n, 0\leq t\leq \ell}$
\begin{align}\label{e:zk}
z_k(t)=\frac{\sfx_{k+i-1}((t+\beta\ell)N)}{N}-\al\ell,\quad 1\leq k\leq n.
\end{align}
Then $\{z_{k}(t)\}_{1\leq k\leq n, 0\leq t\leq \ell}$ form an $n$-particle nonintersecting Bernoulli random walk, where
\begin{align}\label{e:particlebound}
\left|\frac{\theta n}{N}-\ell\varrho\right|\leq C\varepsilon,
\end{align}
and its height function $\cH^{\rm s}$ (defined as in \eqref{e:Hxt}) satisfies
\begin{align}\label{e:heightclose}
\left|\cH^{\rm s}(x,t)-\cA(x,t)\right|\leq C\varepsilon,\quad 0\leq t\leq \ell, \quad x\in \fR.
\end{align}
where $\cA$ is the height function with slope given by $(\varrho, -\varrho v)$ from \Cref{d:constantslope}.
\end{lemma}
\begin{proof}
We denote the height function of the rescaled particle configuration $\{\bmx(t)\}_{0\leq t\leq T}$ as $\cH$, then from the construction \eqref{e:zk}, $\cH^{\rm s}$ is a shifted version of $\cH$. More precisely,  for any $0\leq t\leq \ell$,
\begin{align}\label{e:Hxtdiff}
\cH^{\rm s}(x,t)=\cH(x+\al\ell,t+\beta\ell)-\frac{(i-1)\theta}{N},\quad z_1(t)\leq x\leq z_n(t)+\frac{\theta}{N}.
\end{align}
And for $x\leq z_1(t)$ we have $\cH^{\rm s}(x,t)=0$; for $x\geq z_n(t)$ we have $\cH^{\rm s}(x,t)=n\theta/N$. 
Next, we show \eqref{e:particlebound}. Since $\cH\in \bB_\varepsilon(H^*)$, we have
\begin{align}\label{e:ntheta}
\frac{n\theta}{N}+\OO\left(\frac{1}{N}\right)=\cH((\alpha+1)\ell, \beta\ell)-\cH(\alpha\ell, \beta\ell)
=H^*((\alpha+1)\ell, \beta\ell)-H^*(\alpha\ell, \beta\ell)+\OO(\varepsilon)
=\ell\varrho +\OO(\varepsilon).
\end{align}
The claim \eqref{e:particlebound} follows by rearranging. 
We will prove \eqref{e:heightclose}, by dividing it into different regions.
\begin{enumerate}
\item For $x\leq \min\{z_1(t), tv\}$, $\cH^{\rm s}(x,t)=\cA(x,t)=0$.

\item For $x\geq \min\{z_n(t)+\theta/N, tv+\ell\}$, $\cH^{\rm s}(x,t)=\theta n/N$ and $\cA(x,t)=\ell\varrho$, and the claim \eqref{e:heightclose} follows from \eqref{e:particlebound} .

\item
For $z_1(t)\leq x\leq z_n(t)+\theta/N$, we have
\begin{align}\begin{split}\label{e:Hxt2}
&\phantom{{}={}} \cH^{\rm s}(x,t)=\cH(x+\al\ell,t+\beta\ell)-\frac{(i-1)\theta}{N}
=\cH(x+\al\ell,t+\beta\ell)-\cH(\al\ell,\beta\ell)\\
&=H^*(x+\al\ell,t+\beta\ell)-H^*(\al\ell,\beta\ell)+\OO(\varepsilon)\\
&=x\varrho -t \varrho v+\OO(\varepsilon).
\end{split}\end{align}
It follows that for $\max\{z_1(t), tv\}\leq x\leq \min\{z_n(t)+\theta/N, tv+\ell\}$, we have $|\cH^{\rm s}(x,t)-\cA(x,t)|=\OO(\varepsilon)$.
\item 

If $tv\leq z_1(t)$, then for $tv\leq x\leq z_1(t)$ we have
\begin{align*}
0=H(x,t)\geq \cH(x+\al\ell,t+\beta\ell)-\cH(\al\ell,\beta\ell)= x\varrho -t \varrho v+\OO(\varepsilon)=\cA(x,t)+\OO(\varepsilon),
\end{align*}
where in the second statement we used the construction of \eqref{e:Hxt2}; in the third statement we used that $H^*$ and $\cH$ are close, and $H^*$ has a linear approximation with slope $(\varrho, -\varrho v)$.
The claim \eqref{e:heightclose} follows.
\item 
If $z_1(t)\leq tv$, then for $z_1(t)\leq x\leq tv$ we have $\cA(x,t)=0$, and
\begin{align*}
0\leq \cH^{\rm s}(x,t)=x\varrho -t \varrho v+\OO(\varepsilon)\leq \OO(\varepsilon),
\end{align*}
where in the second statement we used \eqref{e:Hxt2}; in the last statement we used $x\leq tv$. 

\item If $z_n(t)+\theta /N\leq tv+\ell$, then for $z_n(t)+\theta/N\leq x\leq tv+\ell$ we have
\begin{align*}
\frac{\theta n}{N}=\cH^{\rm s}(x,t)\leq \cH(x+\al\ell,t+\beta\ell)-\cH(\al\ell,\beta\ell)= x\varrho -t \varrho v+\OO(\varepsilon)=\cA(x,t)+\OO(\varepsilon),
\end{align*}
where in the second statement we used the construction of \eqref{e:Hxt2}; in the third statement we used that $H^*$ and $\cH$ are close, and $H^*$ has a linear approximation with slope $(\varrho, -\varrho v)$.
The claim \eqref{e:heightclose} follows from the above estimate together with \eqref{e:ntheta}. 

\item If $tv+\ell\leq z_n(t)+\theta/N$, then for $tv+\ell\leq x\leq z_n(t)+\theta/N$ we have $\cA(x,t)=\varrho\ell$, and 
\begin{align}\label{e:Hextreme}
\frac{\theta n}{N}\geq \cH^{\rm s}(x,t)=x\varrho -t \varrho v+\OO(\varepsilon)\geq \ell\varrho+\OO(\varepsilon).
\end{align}
where in the second statement we used \eqref{e:Hxt2}; in the last statement we used $x\geq\ell+ tv$. The claim \eqref{e:heightclose} follows from combining \eqref{e:Hextreme} with \eqref{e:ntheta}. 

%
%
\end{enumerate}
\end{proof}

With the above decomposition, we can rewrite the probability $\bP(\bB_\varepsilon(H^*:\{\bm\sfy(\sft)\}_{\sft\in\{0, \sfL, 2\sfL, \cdots, \sfT\}}))$ as 
\begin{align}\begin{split}\label{e:ldpup}
\frac{1}{2^{N\sfT}}\sum_{\bB_\varepsilon(H^*:\{\bm\sfy(\sft)\}_{\sft\in\{0, \sfL, 2\sfL, \cdots, \sfT\}})}\prod_\beta \prod_{ \al}\prod_{\beta\sfL\leq \sft<(\beta+1)\sfL}\prod_{i, j\in I(\alpha, \beta):i<j}\left(1+\frac{\theta(e_i(\sft)-e_j(\sft))}{\sfx_i(\sft)-\sfx_j(\sft)}\right)\\
\times \prod_\beta \prod_{ \al<\al'}\prod_{\beta\sfL\leq \sft<(\beta+1)\sfL}\prod_{i\in I(\alpha, \beta), j\in I(\alpha', \beta)}\left(1+\frac{\theta(e_i(\sft)-e_j(\sft))}{\sfx_i(\sft)-\sfx_j(\sft)}\right),
\end{split}\end{align}
where the summation is over paths $\{\bm\sfx(\sft)\}_{0\leq \sft\leq \sfT}\in\bB_\varepsilon(H^*:\{\bm\sfy(\sft)\}_{\sft\in\{0, \sfL, 2\sfL, \cdots, \sfT\}})$.

The large deviation upper bound in \Cref{t:main1} is a consequence of the following two statements. We postpone their proofs in \Cref{s:upB}. 
\begin{proposition}\label{p:LDP}
Adopt the notations in \Cref{t:main1}. Fix any $H^*\in \Adm(\fR,h)$, and a particle configuration $\{\bm\sfy(\sft)\}_{0\leq \sft\leq \sfT}\in \bB_\varepsilon(H^*)$, with $\bm\sfy(0)=\bm\sfy^{(N)}$ and $\bm\sfy(\sfT)=\bm\sfz^{(N)}$.
There exists a constant $C>0$, 
\begin{align}\label{e:trivialub}
&\phantom{{}={}}\frac{1}{(\ell N)^2}\ln \sum \prod_{\beta\sfL\leq \sft<(\beta+1)\sfL}\prod_{i,j\in I(\alpha, \beta):i<j}\left(1+\frac{\theta(e_i(\sft)-e_j(\sft))}{\sfx_i(\sft)-\sfx_j(\sft)}\right)\leq C,
\end{align}
where the summation is over $\{{\sf\bmx}_{k}(\sft)\}_{k\in I(\al, \beta), t\in [\beta\sfL,(\beta+1)\sfL]}$ with ${\sf\bmx}_k(\beta\sfL)={\sf\bmy}_k(\beta\sfL)$.
If we assume that on $\fR(\al, \beta)$, the height function $H^*$ has a linear approximation  with slope $(\varrho, -\varrho v)\in \cT$ and error $\varepsilon_0=\varepsilon/\ell$ in the sense of \Cref{l:Lipschitz}, then{
\begin{align}\begin{split}\label{e:linearapub}
&\phantom{{}={}}\frac{1}{(\ell N)^2}\ln \sum \prod_{\beta\sfL\leq \sft<(\beta+1)\sfL}\prod_{i,j\in I(\alpha, \beta):i<j}\left(1+\frac{\theta(e_i(\sft)-e_j(\sft))}{\sfx_i(\sft)-\sfx_j(\sft)}\right)\\
&
\qquad=\frac{1}{\theta}\iint_{[\al\ell,  (\al+1)\ell]\times [\beta\ell, (\beta+1)\ell]}\sigma(\nabla H^*)\rd x\rd t+\OO(\varepsilon_0+\varepsilon^{1/2}\log(1/\varepsilon)),
\end{split}\end{align}
where the summation is over $\{{\sf\bmx}_{k}(\sft)\}_{k\in I(\al, \beta), t\in [\beta\sfL,(\beta+1)\sfL]}$ with ${\sf\bmx}_k(\beta\sfL)={\sf\bmy}_k(\beta\sfL)$ and \eqref{e:heightclose} holds. }
\end{proposition}

\begin{lemma}\label{l:offdiagonal}Adopt notations in \Cref{t:main1}. Fix any $H^*\in \Adm(\fR,h)$, and a particle configuration $\{\bm\sfy(\sft)\}_{0\leq \sft\leq \sfT}\in \bB_\varepsilon(H^*)$, with $\bm\sfy(0)=\bm\sfy^{(N)}$ and $\bm\sfy(\sfT)=\bm\sfz^{(N)}$.
For any $\{\bm\sfx(\sft)\}_{0\leq \sft\leq \sfT}\in \bB_\varepsilon(H^*:\{\bm\sfy(\sft)\}_{\sft\in\{0, \sfL, 2\sfL, \cdots, \sfT\}})$, the following holds 
\begin{align*}
&\phantom{{}={}} \sum_\beta \sum_{\al < \al'}\sum_{\beta \sfL\leq \sft<(\beta+1)\sfL}
\sum_{i\in I(\alpha, \beta), j\in I(\alpha', \beta)}\ln\left(1+\frac{\theta(e_i(\sft)-e_j(\sft))}{\sfx_i(\sft)-\sfx_j(\sft)}\right)\\
&=\sum_{1\leq i<j\leq N}\theta\ln\left(\frac{\sfy_i(\sfT)-\sfy_j(\sfT)}{\sfy_i(0)-\sfy_j(0)}\right)+\OO(\varepsilon_0 N^2).
\end{align*}
\end{lemma}

\begin{proof}[Proof of Large Deviation Upper bound in \Cref{t:main1}]
Thanks to \Cref{l:offdiagonal}, we can bound \eqref{e:ldpup} as 
\begin{align*}
\frac{1}{N^2}\ln\eqref{e:ldpup} 
&\leq-T\ln(2)+\frac{1}{N^{2}}\ln \sum_{\bB_\varepsilon(H^*:\{\bm\sfy(\sft)\}_{\sft\in\{0, \sfL, 2\sfL, \cdots, \sfT\}})} \prod_{\al, \beta}\prod_{\beta\sfL\leq \sft<(\beta+1)\sfL}\prod_{i, j\in I(\alpha, \beta):i<j}\left(1+\frac{\theta(e_i(\sft)-e_j(\sft))}{\sfx_i(\sft)-\sfx_j(\sft)}\right)\\
&+\frac{1}{N^2} \sum_{1\leq i<j\leq N}\theta\ln\left(\frac{\sfx_i(\sfT)-\sfx_j(\sfT)}{\sfx_i(0)-\sfx_j(0)}\right)+\OO(\varepsilon_0)\\
&=-T\ln(2)+\frac{1}{\theta}\sum_{\al,\beta} \iint_{[\al \ell, (\al+1)\ell]\times [\beta\ell, (\beta+1)\ell]} \sigma(\nabla H^*)\rd x\rd t+\OO(\varepsilon_0+\varepsilon^{1/2}\log(1/\varepsilon))\\
&+\frac{1}{N^2} \sum_{1\leq i<j\leq N}\theta\ln\left(\frac{\sfy_i(\sfT)-\sfy_j(\sfT)}{\sfy_i(0)-\sfy_j(0)}\right)+\OO(\varepsilon_0)\\
&=\frac{1}{\theta}\iint_\fR \sigma(\nabla H^*)\rd x\rd t+\frac{1}{2\theta}\left.\iint_{\bR^2} \ln|x-y|\rd h(x,t)\rd h(x,t)\right|_0^T -T\ln(2)+\OO(\varepsilon_0+\varepsilon^{1/2}\ln(1/(\varepsilon))),
\end{align*}
where in the second equality, we used  Proposition \ref{p:LDP}
and that by assumption, the number of $(\al,\beta)$ such that  on $\fR(\al,\beta)$, $H^*$ does not have a linear approximation,  is bounded by $\varepsilon_0$. In the last line we used \Cref{l:freeentropy}. The large deviation upper bound in \Cref{t:main1} follows by sending $N$ to infinity, and $\varepsilon_0, \varepsilon$ to zero.
\end{proof}

\subsection{Lower Bound}\label{s:lowbb}


To prove the large deviation lower bound, we recall the height function $\cH$ as constructed in \eqref{e:constructH}
\begin{align}\label{e:cHbound}
\|\cH-H^*\|_\infty \leq \varepsilon.
\end{align}
We denote $\{\bmy(t)\}_{t\in\{0, \ell, 2\ell, \cdots, T\}}$ the particle configurations corresponding to $\cH$ at times $0, \ell, 2\ell, \cdots, T$.

We recall the $\ell$-mesh from \Cref{d:lemesh}. Take any $\fR(\al,\beta)$, the same as in the upper bound, we consider the restriction of the particle configuration $\bmy(\beta \ell )$ on the interval $[\al \ell , (\al+1)\ell )$
\begin{align}\label{e:defIab}
y_{i-1}(\beta \ell)<\al\ell\leq y_i(\beta \ell )<\cdots<y_j(\beta \ell )<(\al+1)\ell \leq y_{j+1}(\beta \ell ).
\end{align}
Denote the index set $I(\alpha, \beta)=\{i,i+1, i+2,\cdots, j\}$. For any $\{\bmx(t)\}_{0\leq t\leq T}$ in $\bB_\varepsilon(H^*:\{\bm\sfy(\sft)\}_{\sft\in\{0, \sfL, 2\sfL, \cdots, \sfT\}})$, its restriction on the index set $I(\al,\beta):=n$ and time interval $[\beta \ell, (\beta+1)\ell]$, $\{x_{k}(t)\}_{k\in I(\al, \beta), t\in [\beta\ell,(\beta+1)\ell]}$ form an $n$-particle nonintersecting Bernoulli random walk inside $\fR(\al, \beta)$.
The shifted particle configuration
\begin{align}\label{e:zk3}
y_{k+i-1}(t+\beta\ell)-\al\ell,\quad 1\leq k\leq n,
\end{align}
form an $n$-particle nonintersecting Bernoulli random walk on the parallelogram shaped region
\begin{align}\label{e:deffP}
\widetilde \fP=\{x,t\in \bR^2: 0\leq t\leq \ell, y_{i}(t+\beta\ell)-\al \ell\leq x\leq y_{j+1}(t+\beta\ell)-\theta-\al\ell\}.
\end{align}
Fix a small parameter $0< \zeta\ll 1$ satisfying \eqref{e:introepsilon}, we define the sub-region for slopes $\cT_\zeta$
\begin{align}\label{e:defcT}
\cT_\zeta=\{(u,v)\in \cT: \zeta< u, -v, u+v\leq 1-\zeta\}.
\end{align}

To obtain a large deviation lower bound, we will restrict to the following set of $N$-particle nonintersecting Bernoulli random walks $\{\bmx(t)\}_{0\leq t\leq T}$: for any $\alpha, \beta$, 
\begin{enumerate}
\item If on $\fR(\al, \beta)$, the height function $H^*$ has a linear approximation with slope $(\varrho, -\varrho v)\in \cT_\zeta$ (recall from \eqref{e:defcT}) and error $\varepsilon_0=\varepsilon/\ell$ in the sense of \Cref{l:Lipschitz}. We denote the shifted particle configuration $\{z_k(t)\}_{1\leq k\leq n, 0\leq t\leq \ell}$
\begin{align}\label{e:zk2}
z_k(t)=y_{k+i-1}(t+\beta\ell)-\al\ell,\quad 0\leq t\leq \ell, \quad 1\leq k\leq n.
\end{align}
Then $z_k(t)$ form an $n$-particle nonintersecting Bernoulli walk inside $\widetilde \fP$ from \eqref{e:deffP}. Recall that $\cA$ is the height function with slope given by $(\varrho, -\varrho v)$ from \Cref{d:constantslope}. Then $\cA$ is a linear approximation of the height function of $\{z_k(t)\}_{1\leq k\leq n, 0\leq t\leq \ell}$, with $L_\infty$ distance bounded by $C\varepsilon$. 

We restrict to $\bmx(t)$ such that the  height function $\cH^{\rm s}$ of $\{x_{k+i-1}(t+\beta\ell)-\al\ell\}_{1\leq k\leq n, 0\leq t\leq \ell}$ satisfies 
\begin{align}\label{e:heightclose2}
(x_{k+i-1}(t+\beta\ell)-\al\ell,t)\in \widetilde \fP,\quad \left|\cH^{\rm s}(x,t)-\cA(x,t)\right|\leq C\varepsilon,\quad 0\leq t\leq \ell. 
\end{align}
\item Otherwise, we take 
\begin{align}\label{e:x=y}
\{x_{k}(t)\}_{k\in I(\al, \beta), t\in [\beta\ell,(\beta+1)\ell]}=\{y_{k}(t)\}_{k\in I(\al, \beta), t\in [\beta\ell,(\beta+1)\ell]}.
\end{align}
\end{enumerate}

\begin{lemma}
For any $N$-particle nonintersecting Bernoulli random walks $\{\bmx(t)\}_{0\leq t\leq T}$ satisfying \eqref{e:heightclose2} and \eqref{e:x=y}, its height function $\widehat\cH$ satisfies
\begin{align}\label{e:HHdiff}
\|\widehat\cH-H^*\|\leq C\varepsilon.
\end{align}
\end{lemma}
\begin{proof}
For any $\al,\beta$, we recall the index set $I(\alpha, \beta)=\{i,i+1, i+2,\cdots, j\}$ from \eqref{e:defIab}. 
There are two cases. If on $\fR(\al,\beta)$, $H^*$ has a linear approximation with slope $(\varrho, -\varrho v)\in \cT_\zeta$, then for $\beta\ell \leq t\leq (\beta+1)\ell$ and $y_i(t)\leq x\leq y_{j}(t)$, we have
\begin{align}\begin{split}\label{e:bb1}
&\phantom{{}={}}\varrho(x-y_i(t))-\varrho v(t-\beta\ell)=H^*(x,t)-H^*(y_i(t),\beta\ell)+\OO(\varepsilon)\\
&\quad =
H^*(x,t)- \cH(y_i(t),\beta\ell) +\OO(\varepsilon)
=H^*(x,t)-i\theta/N+\OO(\varepsilon)\,,
\end{split}\end{align}
where we used \eqref{e:cHbound}.
And using \eqref{e:heightclose2}, the same argument as in \eqref{e:Hxt2}
\begin{align}\begin{split}\label{e:Hxt3}
\widehat\cH(x,t)-i\theta/N=\widehat\cH(x,t)-\widehat\cH(y_i(t),\beta \ell)=\varrho(x-y_i(t))-\varrho v(t-\beta\ell)+\OO(\varepsilon).
\end{split}\end{align}
The claim \eqref{e:HHdiff} follows from combining \eqref{e:bb1} and \eqref{e:Hxt3}.

In the second case that on $\fR(\al,\beta)$,  either $H^*$ has a linear approximation with slope $(\varrho, -\varrho v)\not\in \cT_\zeta$, or $H^*$ does not have a linear approximation. Then our construction gives
\begin{align*}
\widehat\cH(x,t)=\cH(x,t)=H^*(x,t)+\OO(\varepsilon),
\end{align*}
and \eqref{e:HHdiff} holds trivially.

\end{proof}

The large deviation lower bound in \Cref{t:main1} is a consequence of the following statement. We postpone its proof in \Cref{s:lowB}. 

\begin{proposition}\label{p:LDPlow} 
Adopt notations from \Cref{t:main1}. Fix any $H^*\in \Adm(\fR,h)$, and denote $\{\bmy(t)\}_{0\leq t\leq T}$ the particle configuration constructed in \Cref{e:constructH}. There exists a constant $C>0$, 
\begin{align}\label{e:triviallow}
\frac{1}{(\ell N)^2}\ln \sum \prod_{\beta\sfL\leq \sft<(\beta+1)\sfL}\prod_{i,j\in I(\alpha, \beta):i\neq j}\left(1+\frac{\theta(e_i(\sft)-e_j(\sft))}{\sfx_i(\sft)-\sfx_j(\sft)}\right)\geq-C,
\end{align}
where the summation is over $\{{x}_{k}(t)=\sfx_k(Nt)/N\}_{k\in I(\al, \beta), t\in [\beta\ell,(\beta+1)\ell]}$  such that \eqref{e:heightclose2} holds.

If we assume that on $\fR(\al, \beta)$, the height function $H^*$ has a linear approximation with slope $(\varrho, -\varrho v)\in \cT_\zeta$  (recall from \eqref{e:defcT}) and error $\varepsilon_0=\varepsilon/\ell$ in the sense of \Cref{l:Lipschitz}. 
\begin{align}\begin{split}\label{e:linearaplow}
&\phantom{{}={}}\frac{1}{(\ell N)^2}\ln \sum \prod_{\beta\sfL\leq \sft<(\beta+1)\sfL}\prod_{i,j\in I(\alpha, \beta):i\neq j}\left(1+\frac{\theta(e_i(\sft)-e_j(\sft))}{\sfx_i(\sft)-\sfx_j(\sft)}\right)\\
&\geq \frac{1}{\theta}\iint_{[\al\ell,  (\al+1)\ell]\times [\beta\ell, (\beta+1)\ell]}\sigma(\nabla H^*)\rd x\rd t+\OO\left(\varepsilon_0+\frac{\varepsilon \ln(\ell/\delta) }{\delta\zeta^2} +\frac{\delta \ln(\ell/\delta)^2}{\ell}+\frac{\xi\ln(\ell/\xi)}{ \ell}\right),
\end{split}\end{align}
where the summation is over $\{x_{k}(t)=\sfx_k(Nt)/N\}_{k\in I(\al, \beta), t\in [\beta\ell,(\beta+1)\ell]}$  such that \eqref{e:heightclose2} holds.
\end{proposition}

\begin{proof}[Proof of Large Deviation Lower bound in \Cref{t:main1}]
Thanks to \Cref{l:offdiagonal} and \Cref{p:LDPlow}, we can lower bound \eqref{e:ldpup} as
\begin{align*}
\frac{1}{N^2}\ln\eqref{e:ldpup} 
&\geq -T\ln(2)+\frac{1}{N^{2}}\ln \sum_{\bB_\varepsilon(H^*:\{\bm\sfy(\sft)\}_{\sft\in\{0, \sfL, 2\sfL, \cdots, \sfT\}})} \prod_{\al, \beta}\prod_{\beta\sfL\leq \sft<(\beta+1)\sfL}\prod_{i, j\in I(\alpha, \beta):i<j}\left(1+\frac{\theta(e_i(\sft)-e_j(\sft))}{\sfx_i(\sft)-\sfx_j(\sft)}\right)\\
&+\frac{1}{N^2} \sum_{1\leq i<j\leq N}\theta\ln\left(\frac{\sfx_i(\sfT)-\sfx_j(\sfT)}{\sfx_i(0)-\sfx_j(0)}\right)+\OO(\varepsilon_0)\\
&=-T\ln(2)+\frac{1}{\theta}\sum_{\al,\beta} \int_{[\al \ell, (\al+1)\ell]\times [\beta\ell, (\beta+1)\ell]} \sigma(\nabla H^*)\rd x\rd t+\frac{1}{N^2} \sum_{1\leq i<j\leq N}\theta\ln\left(\frac{\sfx_i(\sfT)-\sfx_j(\sfT)}{\sfx_i(0)-\sfx_j(0)}\right)\\
&+\OO\left(\varepsilon_0+\frac{\varepsilon \ln(\ell/\delta) }{\delta\zeta^2} +\frac{\delta \ln(\ell/\delta)^2}{\ell}+\frac{\xi\ln(\ell/\xi)}{ \ell}\right)\\
&=\frac{1}{\theta}\iint_\fR \sigma(\nabla H^*)\rd x\rd t+\frac{1}{2\theta}\left.\iint_{\bR^2} \ln|x-y|\rd h(x,t)\rd h(x,t)\right|_0^T -T\ln(2)\\
&+\OO\left(\varepsilon_0+\frac{\varepsilon \ln(\ell/\delta) }{\delta\zeta^2} +\frac{\delta \ln(\ell/\delta)^2}{\ell}+\frac{\xi\ln(\ell/\xi)}{ \ell}\right),
\end{align*}
where in the second equality, we used that the number of $(\al,\beta)$ such that  on $\fR(\al,\beta)$, $H^*$ does not have a linear approximation is bounded by $\varepsilon_0$; in the last line we used \Cref{l:freeentropy}. The large deviation lower bound in \Cref{t:main1} follows by using the choice of parameters from \eqref{e:introepsilon}.

\end{proof}

\section{Large Deviation Upper Bound: Constant Slope Case}
\label{s:upB}

In this section we study the following $n$-particle non-intersecting $\theta$-Bernoulli walk from time $0$ to $\sfL=\ell N$ $ \big( \bm{\mathsf{x}} (0),\bm{\mathsf{x}} (1), \ldots , \bm{\mathsf{\sfx}} (\sfL) \big) \in (\bW^n_\theta)^{\sfL}$ (see \Cref{d:bernoulliwalk})
\begin{align}\label{e:Mkcopy}
    \bP(\bm\sfx(\sft+1)=\bm\sfx+\bme|\bmx(\sft)=\bm\sfx)
   =\frac{1}{2^n}
    \frac{V(\bm\sfx+\theta \bme)}{V(\bm\sfx)} ,\quad 0\leq \sft\leq \sfL.
\end{align}
Then we will prove \Cref{p:LDP}, the large deviation upper bound for the non-intersecting $\theta$-Bernoulli walks with height function approximately linear.

We recall the height function $\cA(x,t)$ with constant slope $\nabla \cA(x,t)=(\varrho, -\varrho v)\in \overline{\cT}$ from \Cref{d:constantslope},
\begin{align}\label{e:defH*}
&\cA(x,t)= (x-tv)\varrho, \quad (x,t)\in \fP;\quad  \cA(x,t)=0,\quad x\leq tv; \quad\cA(x,t)=\ell \varrho, \quad x\geq \ell+tv.
\end{align} 
where $\fP$ is the parallelogram region
\begin{align}\label{e:defp}
&\fP:=\{x,t\in \bR^2: 0\leq t\leq\ell, \quad tv\leq x\leq \ell+tv\}, 
\end{align}
We assume that the density satisfies \eqref{e:particlebound}
\begin{align}\label{e:densitybound}
\left|\frac{\theta n}{N}-\ell \varrho\right|\leq C\varepsilon.
\end{align}
Fix a small parameter $0<\zeta\ll 1$ satisfying \eqref{e:introepsilon}, we recall $\cT_\zeta$ from \eqref{e:defcT}
\begin{align*}
\cT_\zeta=\{(u,v)\in \cT: \zeta< u, -v, u+v\leq 1-\zeta\}.
\end{align*}
There are two cases, 
\begin{enumerate}
\item Extreme slope case where $(\varrho, -\varrho v)\not\in \cT_\zeta$. 
\item Interior slope case where $(\varrho, -\varrho v)\in \cT_\zeta$. 
\end{enumerate}

The following proposition gives the large deviation upper bound for $n$-particle non-intersecting $\theta$-Bernoulli walks with height function having approximately constant slope. Its proofs are given in \Cref{s:extremeslope} and \Cref{s:interiorslope}.
\begin{proposition}\label{p:LDPconstant}
We recall the rate function $\sigma$ from \eqref{sigmal}, and parameters $\zeta, \delta, \varepsilon, \ell$ from \eqref{e:introepsilon}.
Fix $(\varrho, -\varrho v)\in \overline{\cT}$ such that \eqref{e:densitybound} holds. Take the height function $\cA(x,t)$ with constant slope $\nabla \cA=(\varrho, -\varrho v)$ on the parallelogram region $\fP$, as in \eqref{e:defH*}. The height function $\cH$ of the $n$-particle non-intersecting $\theta$-Bernoulli walk \eqref{e:Mkcopy} satisfies
\begin{enumerate}
\item If $(\varrho, -\varrho v)\not\in \cT_\zeta$, then we have
\begin{align}\label{e:contribution}
\frac{1}{(\ell N)^2}\ln \bP(\|\cH-\cA\|_\infty\leq \varepsilon)=-\frac{n}{\ell N}\ln(2)+\OO(\zeta^{1/2} \ln (1/\zeta)).
\end{align}
\item
If $(\varrho, -\varrho v)\in \cT_\zeta$, then we have
\begin{align*}
\frac{1}{(\ell N)^2}\ln \bP(\|\cH-\cA\|_\infty\leq \varepsilon)\leq \frac{\sigma(\varrho, -\varrho v)}{\theta} -\frac{n}{\ell N}\ln(2)+\OO\left(\frac{\varepsilon  }{\delta\zeta^2 }+(\delta/\ell)\log^2(\ell/\delta)\right).
\end{align*}
\end{enumerate}
\end{proposition}

\begin{proof}[Proof of \Cref{p:LDP}]
The first statement \eqref{e:trivialub} in \Cref{p:LDP} follows from \Cref{l:VVbb}.

For the second statement \eqref{e:linearapub}, we recall from \Cref{l:Lipschitz} that at least $1-\varepsilon_0$ fraction of the points $x$ of $\fR(\al,\beta)$, the gradient $\nabla H^*$ exists and is within $\varepsilon_0$ of $(\varrho, -\varrho v)$. As a consequence, 
\begin{align}\label{e:integralform}
\frac{1}{ \ell^2}\iint_{[\al\ell,  (\al+1)\ell]\times [\beta\ell, (\beta+1)\ell]}\sigma(\nabla H^*)\rd x\rd t
=\sigma(\varrho,-\varrho v)+\OO(\varepsilon_0).
\end{align}
If $(\varrho,-\varrho v)\in \cT_\zeta$, then the claim \eqref{e:linearapub} follows from \eqref{e:integralform} and the second statement of \Cref{p:LDPconstant}.
If $(\varrho,-\varrho v)\notin \cT_\zeta$, without loss of generality we assume that $\varrho\leq \zeta$. The Lobachevsky function (recall from \eqref{sigmal}) is $\pi$-periodic, namely $L(x+\pi)=L(x)$ and $-L(\pi-x)=L(x)$. Moreover, because of the logarithmic singularity, the Lobachevsky function satisfies
\begin{align}\label{e:Lobdiff}
|L(x)-L(y)|\lesssim |x-y|\log (1/|x-y|).
\end{align}
and 
$L(\zeta)=\OO(\zeta \log (1/\zeta))$ for $\zeta\rightarrow 0+$.
 Therefore
\begin{align}\label{e:sigmasmall}
\sigma(\varrho, -\varrho v)=
-L(\varrho)+L(\varrho v)+L((1-v)\varrho)=\OO(\zeta\ln(1/\zeta)).
\end{align}
The claim \eqref{e:linearapub} follows from \eqref{e:contribution}
and  \eqref{e:sigmasmall}.
\end{proof}

\subsection{Outline of  the proof of \Cref{p:LDPconstant}}
For the extreme slope case, to estimate $\bP(\|\cH-\cA\|_\infty\leq \varepsilon)$, we will directly upper bound the number of height profiles satisfying $\|\cH-\cA\|_\infty\leq \varepsilon$ by a combinatorics argument, and the probability $\bP(\cH)$ for each of these height profile by a delicate analysis of Vandermonde determinants. It turns out both parts are negligible. The proof is given in \Cref{s:extremeslope}.

For the interior slope case, we recall the constant slope height function $\cA$ from \eqref{e:defH*}. We convolve it with a Cauchy distribution. Take small $\delta>0$, denote  
\begin{align}\label{e:deftH}
\widetilde H(x,t)=\frac{1}{\pi}\int \frac{\delta \cA(y,t)\rd y}{(y-x)^2+\delta^2}.
\end{align}
Since for $(x,t)\in \fP$, $\nabla \cA(x,t)=(\varrho, -\varrho v)\in\cT_\zeta$, namely, 
$
\zeta\leq \varrho, \varrho v, \varrho(1-v)\leq 1-\zeta,
$
it follows that $\nabla \widetilde H(x,t)\in \cT$ for $(x,t)\in \fR$.  Properties of the smoothed height function $\wt H(x,t)$ are collected in \Cref{s:heightslope}.

We define the associated \emph{complex slope} $\wt f:  \fR \rightarrow \mathbb{H}^-$ by setting, for any $(x,t) \in \fR$, $\wt f_t(x)=\wt f(x,t) \in \mathbb{H}^-$ to be the unique complex number satisfying 
	\begin{flalign}
		\label{fh}
		\arg^* \widetilde f_t(x) = - \pi \partial_x \widetilde H (x,t); \qquad \arg^* \big( \wt f_t(x) + 1 \big) = \pi \partial_t \widetilde H (x,t),
	\end{flalign}

	\noindent where for any $z \in \mathbb{H}^-$ we have set $\arg^* z = \theta \in (-\pi, 0)$ to be the unique number in $(-\pi, 0)$ satisfying $e^{-\mathrm{i} \theta} z \in \mathbb{R}$; see \Cref{slope1} for a depiction. By the law of sines \eqref{fh} implies that
	\begin{align}\label{e:sinelaw}
	\frac{|\wt f_t(x)|}{\sin(-\pi \del_t \widetilde H(x,t))}=\frac{|1+\wt f_t(x)|}{\sin(\pi \del_x \widetilde H(x,t))}=\frac{1}{\sin(\pi \del_x \widetilde H(x,t)+\pi \del_t \widetilde H(x,t))}.
	\end{align}
From \eqref{e:sinelaw}, we can solve the complex slope $\wt f_t(x)$,
\begin{align}\begin{split}\label{e:ftx}
& \widetilde f_t(x)=e^{-\ri\pi \partial_x  \widetilde H(x,t)}\exp(\ln \sin(-\pi \partial_t  \widetilde H(x,t))- \ln \sin(\pi \partial_x  \widetilde H(x,t)+\pi\partial_t  \widetilde H(x,t))),\\
&1+ \widetilde f_t(x)=e^{\ri\pi \partial_t \widetilde H(x,t)}\exp(\ln \sin(\pi \partial_x  \widetilde H(x,t))- \ln \sin(\pi \partial_x  \widetilde H(x,t)+\pi\partial_t  \widetilde H(x,t))).
\end{split}\end{align}

\begin{figure}
	\begin{center}
	 \includegraphics[scale=0.3,trim={0cm 0cm 16cm 14cm},clip]{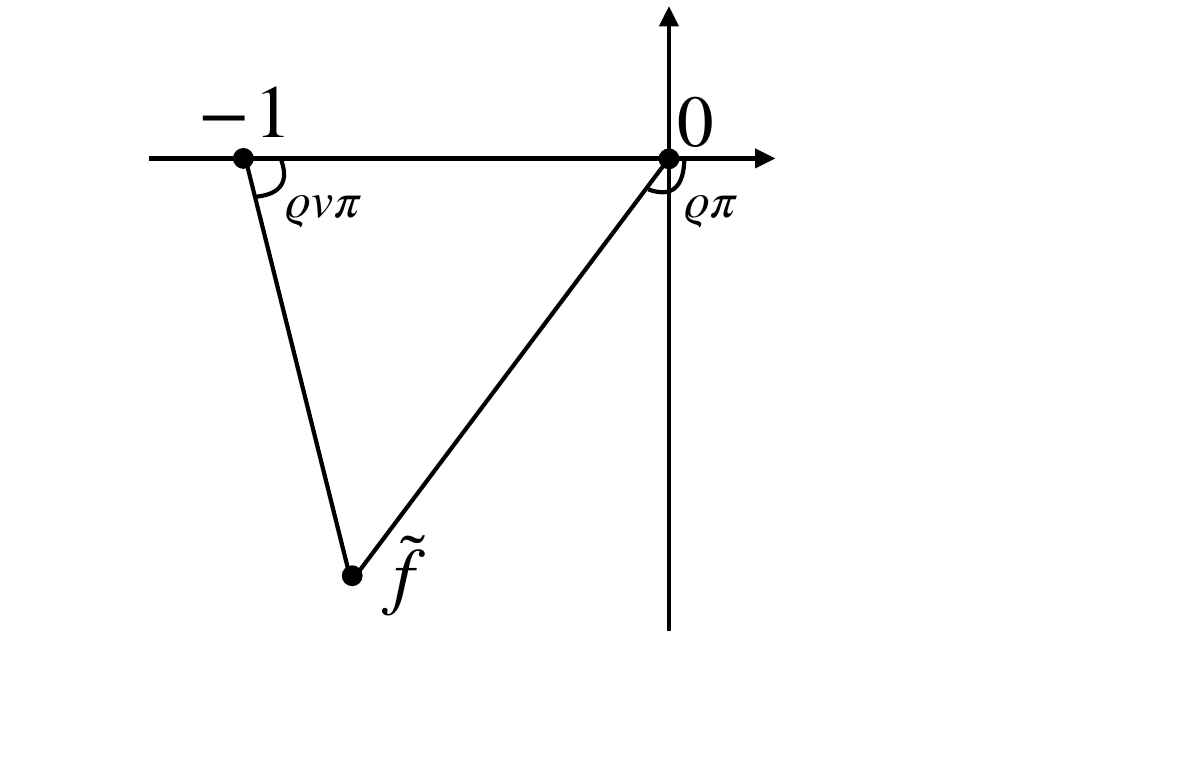}
	 \caption{Shown above is the complex slope $\widetilde f = \widetilde f (x,t)$.}
	 \label{slope1}
	 \end{center}
	 \end{figure}	

We take a drift function $g_t(x)$ given by
\begin{align}\label{e:gtx}
g_t(x)=\ln \sin(-\pi \partial_t  \widetilde H(x,t))- \ln \sin(\pi \partial_x  \widetilde H(x,t)+\pi\partial_t  \widetilde H(x,t))-\Hib( \partial_x  \widetilde H(x,t)).
\end{align}	
As we will show in \Cref{s:heightslope}, the complex slope $\wt f_t(x)$, and the drift function $g_t(x)$ can both be extended analytically to a strip neighborhood of the real axis, and on this strip region they are related by
\begin{align*}
\wt f_t(z)=e^{\wt m_t(z)+g_t(z)},\quad \wt m_t(z)=\int_\bR \frac{\del_x \wt H(x,t)\rd x}{z-x}.
\end{align*}

To prove the large deviation upper bound for the interior case in \Cref{p:LDP}, we  
tilt the law of the Markov process \eqref{e:Mkcopy} by an exponential martingale. 
For any $0\leq \sft \leq \sfL-1$, let
\begin{align*}
    \Delta \cM_t^g:=\frac{1}{N}\left(\sum_{i=1}^n e_i(\sft) g_t(x_i(t))-\bE\left[\left.\sum_{i=1}^n e_i(\sft) g_t(x_i(t))\right| \bmx(t)\right]\right),\quad t=\sft/N,
\end{align*}
and the following is a Martingale 
\begin{align}\label{e:Mt}
    \prod_{t\in N^{-1}\qq{0, \sfL-1}} \frac{e^{N\Delta\cM_{t}^g}}{\bE[e^{N \Delta\cM_{t}^g}|\bmx(t)]}.
\end{align}
In particular it gives the following identity
\begin{align}\label{e:ub1}
    1
    =\bE\left[  \prod_{t\in\qq{0, \sfL-1}/N} \frac{e^{N\Delta\cM_{t}^g}}{\bE[e^{N \Delta\cM_{t}^g}|\bmx(t)]}\right]
    &=\bE\left[\prod_{t\in\qq{0, \sfL-1}/N} \frac{e^{\sum_{i=1}^N e_i(Nt) g_t(x_i(t))}}{\bE[e^{\sum_{i=1}^N e_i(Nt) g_t(x_i(t))}|\bmx(t)]}\right].
\end{align}

The large deviation upper bound follows from restricting the above identity to the event $\{\|\cH-\cA\|_\infty\leq \varepsilon)\}$.
\begin{align}\label{e:Pg2}
    1
    \geq \bE\left[\prod_{t\in\qq{0, \sfL-1}/N} \frac{e^{\sum_{i=1}^N e_i(Nt) g_t(x_i(t))}}{\bE[e^{\sum_{i=1}^N e_i(Nt) g_t(x_i(t))}|\bmx(t)]}\bm1(\|\cH-\cA\|_\infty\leq \varepsilon)\right].
\end{align}
In \Cref{s:interiorslope}, we will show that on the event, $\{\|\cH(x,t)-\cA(x,t)\|_\infty\leq \varepsilon)\}$, up to a small error, the numerator in the exponential martingale \eqref{e:Mt} can be written in terms of $\cA(x,t)$. For the denominator, we will use the dynamical loop equation \Cref{t:loopstudy} to compute its expectation. They together lead to the following large deviation upper bound.

\begin{proposition}\label{p:ubb}
Adopt notations and assumptions from \Cref{p:LDPconstant}. Take the height function $\cA(x,t)$ with constant slope $\nabla \cA=(\varrho, -\varrho v)$ on the parallelogram region $\fP$, as in \eqref{e:defH*}.
 If $(\varrho, -\varrho v)\in \cT_\zeta$, then the height function $\cH$ of the $n$-particle non-intersecting $\theta$-Bernoulli walk \eqref{e:Mkcopy} satisfies
\begin{align*} \limsup_{N\rightarrow\infty}
\frac{1}{(\ell N)^2}\ln \bP(\|\cH-\cA\|_\infty\leq \varepsilon)
\leq -\frac{S(\cA;g)}{\theta \ell^2}+\OO\left(\frac{\varepsilon  }{\delta\zeta^2 }\right),
\end{align*}
where $g_t(x)$ is the drift function from \eqref{e:gtx}
\begin{align}\begin{split}\label{e:rateS}
    S(\cA;g)&=
     -\int_0^{\ell} \int_\bR \del_t \cA(x,t) g_t(x)\rd x \rd t-\frac{1}{2\pi\ri}\int_0^{\ell}\int_0^1\oint_\omega\ln (1+e^{m^*_t(z)+\tau g_t(z)})g_t(z)\rd z\rd \tau\rd t,
\end{split}\end{align}
where the contour $\omega$ encloses $[-3\ell, 3\ell]$ and $m^{*}_{t}$ is the Stieltjes transform of $\del_x \cA(x,t)$,
\begin{align}\label{defm*}
m^{*}_{t}(z)=\int_{tv}^{tv+\ell}\frac{\del_x \cA(x,t)\rd x}{x-z}.
\end{align}
\end{proposition}

We will then analyze the rate function $S(\cA;g)$ in \eqref{e:rateS} in \Cref{s:ratef}. Miraculously, after simplifying the rate function, we recover these
{Lobachevsky functions} (see \eqref{sigmal}). 
 
 \begin{proposition}\label{p:entropy}
Adopt the notations and assumptions in \Cref{p:ubb}.
The rate function satisfies
\begin{align}\label{e:entropy}
&S(\cA; g)=-\ell^2( \sigma(\varrho, -\varrho v)-\varrho\ln(2))+\OO( \delta \ln(\ell/\delta)^2\ell ).
\end{align}
\end{proposition}

\subsection{Proof of \Cref{p:LDPconstant}: Extreme Slope Case}
\label{s:extremeslope}

There are three cases for the extreme slope $(\varrho, -\varrho v)\not\in \cT_\zeta$ (recall from \eqref{e:defcT}): i) $-\varrho v\leq\zeta$ ii) $\varrho+\varrho v\geq1-\zeta$ iii) $\varrho\geq 1-\zeta$. We recall from \eqref{e:introepsilon} that  $\varepsilon\ll \zeta\ell$.

We start with the first two cases. Their proofs are the same, so we will only give the proof of the first case.
\begin{proof}[Proof of the First and Second Cases]
In the first case $-\zeta\leq \del_t \cA\leq 0$, and most of $e_i(\sft)$ are zero. 
In fact, we have  by integration by parts
\begin{align}\begin{split}\label{e:sume}
\frac{\theta}{N^2}\sum_{1\leq i\leq n}\sum_{1\leq \sft \leq \sfL} e_i(\sft)
&=\frac{\theta}{N^2}\left(\sum_{1\leq i\leq n}\sfx_i(\sfL)-\sum_{1\leq i\leq n}\sfx_i(0)\right)
=\int_\bR x \varrho_\ell(x; \bmx(\ell))\rd x -\int_\bR x \varrho_0(x; \bmx(0))\rd x \\
&=\int_\bR (\cH(x,0)-\cH(x,\ell))\rd x
=\int_\bR (\cA(x,0)-\cA(x,\ell))\rd x+\OO(\varepsilon \ell)\\
&=\int_\bR\int_0^\ell -\del_t \cA(x,t)\rd t\rd x+\OO(\varepsilon \ell)
=\OO(\zeta \ell^2+\varepsilon \ell )=\OO(\zeta \ell^2)\,,
\end{split}\end{align}
where we used that  $\varepsilon\ll \zeta\ell$ from \eqref{e:introepsilon}. 
We denote the time set
\begin{align}\label{e:defI}
I=\{1\leq \sft \leq \sfL: \sum_i e_i(\sft)\leq \zeta^{1/2} \ell N\},
\quad
I^\complement=\{1\leq \sft \leq \sfL: \sum_i e_i(\sft)>\zeta^{1/2} \ell N \}.
\end{align}
Then \eqref{e:sume} gives that $|I^\complement|=\OO(\zeta^{1/2}\ell N)$. There are ${\ell N \choose \OO(\zeta^{1/2} \ell N)}=e^{\OO(-\zeta^{1/2}\ln(\zeta) (\ell N))}$ ways to choose the set $I^\complement$. Recall the definition \eqref{e:weightp} of the weight $\cW$. Once we fix the sets $I, I^\complement$, the total weights of non-intersecting Bernoulli walks satisfying \eqref{e:defI} is 
\begin{align}\begin{split}\label{e:Wbound}
\sum_{\sfp \text{ satisfies } \eqref{e:defI}}\cW(\sfp)
&\leq \prod_{\sft\in I} \sum_{\sum_i e_i(\sft)\leq \zeta^{1/2} \ell N} \frac{V(\bm\sfx(\sft)+ \theta \bme(\sft))}{V(\bmx(\sft))}\prod_{\sft\in I^\complement} \sum_{\bme(\sft)}\frac{V(\bm\sfx(\sft)+{\theta }\bme(\sft))}{V(\bmx(\sft))}\\
&\leq \left(\sum_{k\leq \zeta^{1/2} \ell N}{n\choose k}\right)^{|I|} 2^{n |I^\complement|}=e^{\OO(-\zeta^{1/2}\ln(\zeta) (\ell N)^2)},
\end{split}\end{align}
where we used that
\begin{align*}
\sum_{\sum_i e_i=k}\frac{V(\bm\sfx+{ \theta }\bme)}{V(\bm\sfx )}={n\choose k}\leq n^k,
\end{align*}
from \Cref{l:sumV}.
It follows from \eqref{e:Wbound} and the fact that there at most $e^{\OO(-\zeta^{1/2}\ln(\zeta) (\ell N))}$ ways to choose the set $I^\complement$, we conclude \eqref{e:contribution} for $-\varrho v\leq \zeta$. 

In the second case that $1-\zeta\leq \del_x \cA+\del_t \cA\leq 1$, we have that most of $e_i(\sft)$ are one. By the same argument as the first case, we will have that \eqref{e:contribution} holds. 
\end{proof}

For the last case, we have $1-\zeta\leq \del_x \cA\leq 1$, and $(1-\zeta)\ell N\leq n\leq \ell N$, we need to bound the total weight of the paths $\sfp$ corresponding to $\cH$, 
\begin{align}\label{e:Wpsum}
\sum_{\sfp: \|\cH-\cA\|_\infty\leq \varepsilon }\cW(\sfp)=\sum_{\sfp: \|\cH-\cA\|_\infty\leq \varepsilon }\prod_{0\leq \sft\leq \sfT}\prod_{1\leq i<j\leq n}\frac{(\sfx_i(\sft)+\theta e_i(\sft))-(\sfx_j(\sft)+\theta e_j(\sft))}{{\sfx_i(\sft)}-{\sfx_j(\sft)}}. 
\end{align}
We remark that the normalization constant $2^{-n\sfL}$ gives the term  $-(n/\ell N)\ln 2$ in \eqref{e:contribution}. 
We denote the index sets
\begin{align}\label{e:defI012}
I_0=\qq{1, \zeta n}, \quad I_1=\qq{\zeta  n, n-\zeta n},\quad I_2=\qq{n-\zeta n,n}.
\end{align}
For $\|\cH-\cA\|_\infty\leq \varepsilon\leq \zeta \ell$, we have that for any index $i\in I_1$, $x_i(t)\in [tv, \ell+tv]$.
Fix any time $0\leq \sft \leq \sfT$. For $i\in I_0\cup I_2$, there are $2^{|I_0|+|I_2|}=2^{2\zeta \ell N}$ possible choices for $e_i(\sft)$. For the particle configuration $\{\sfx_i(\sft)+e_i(\sft)\}_{i\in I_1}$, using again that $\|\cH-\cA\|_\infty\leq \varepsilon\leq \zeta \ell$, the number of empty sites is bounded by
\begin{align*}
\sum_{i\in I_1} ((\sfx_{i+1}(\sft)+e_{i+1}(\sft))-(\sfx_{i}(\sft)+e_{i}(\sft))-\theta)=\OO(\zeta \ell N+\varepsilon N)=\OO(\zeta \ell N).
\end{align*}
Therefore, the number of configurations $\sfp$ satisfying $\|\cH-\cA\|_\infty\leq \varepsilon$ is at most 
\begin{align}\label{e:totalp}
2^{2\zeta N}{\ell N \choose \OO(\zeta\ell N)}=e^{\OO(-\zeta \ln(\zeta)\ell N)}.
\end{align}

The following lemma bounds each summand in \eqref{e:Wpsum}
\begin{lemma}For non-intersecting Bernoulli walks $\sfp=\{\sfx_i(\sft)\}_{1\leq i\leq n, 0\leq \sft \leq \sfL}$ satisfying that  $\|\cH-\cA\|_\infty\leq \varepsilon$, the following holds
\begin{align}\label{e:xi-xj}
\prod_{1\leq i<j\leq n}\frac{(\sfx_i(\sft)+\theta e_i(\sft))-(\sfx_j(\sft)+\theta e_j(\sft))}{{\sfx_i(\sft)}-{\sfx_j(\sft)}} 
=e^{\OO(\ell N\zeta\ln \frac{1}{\zeta} )}\left(\prod_{1\leq i<j\leq n}\frac{(\sfx_i(\sft)+ e_i(\sft))-(\sfx_j(\sft)+ e_j(\sft))}{{\sfx_i(\sft)}-{\sfx_j(\sft)}} \right)^\theta. 
\end{align}
\end{lemma}
\begin{proof}
To simplify the notation, we will omit the time $\sft$ dependence. We recall from \eqref{e:localeq1}
\begin{align}\label{e:localeq}
\left|\ln \frac{(\sfx_i+\theta e_i)-(\sfx_j+\theta e_j)}{{\sfx_i}-{\sfx_j}}
-\theta \ln \frac{(\sfx_i+ e_i)-(\sfx_j+ e_j)}{{\sfx_i}-{\sfx_j}}
\right|\leq C\frac{(e_i-e_j)^2}{({\sfx_i}-{\sfx_j})^2}.
\end{align}
Next, we show that for $\|\cH-\cA\|_\infty\leq \varepsilon$, 
\begin{align*}
\sum_{1\leq i<j\leq n}\frac{(e_i-e_j)^2}{({\sfx_i}-{\sfx_j})^2}=\sum_{i\in I, j\in J}\frac{1}{({\sfx_i}-{\sfx_j})^2}=\OO(\zeta\ln (\ell /\zeta) \ell N).
\end{align*}
Let $I=\{i\in \qq{1,n}: e_i=1\}$ and $J=\{j\in \qq{1,n}: e_j=0\}$.
The particle configuration $\{\sfx_i\}_{1\leq i\leq n}$ consists of at most $\OO(\zeta\ell N)$ tightly packed pieces. Namely, there exists $1=i_0<i_1<\cdots<i_K=n+1$ with $K=\OO(\zeta\ell N)$, such that for any $i_k\leq i<j<i_{k+1}$, $\sfx_{i}=\sfx_j+\theta(j-i)$. Denote the interval $T_k=\qq{i_{k}, i_{k+1}-1}$. Then either $I\cap T_k=\emptyset$ or $I\cap T_k=\{ i_k, i_{k}+1,\cdots, j_k\}$ for some $i_k\leq j_k<i_{k+1}$, and
\begin{align}\begin{split}\label{e:sumxixj}
\sum_{i\in I\cap T_k, j\in J}\frac{1}{|\sfx_i-\sfx_j|^2}
&\leq \sum_{i\in I\cap T_k}\sum_{d\geq 1+\min\{|i-j_k|, |i-i_k|\}}\frac{1}{(d\theta)^2}\\
&\lesssim \sum_{i\in I\cap T_k}\frac{1}{\theta^2(1+\min\{|i-j_k|, |i-i_k|\})}\lesssim \frac{1}{\theta^2}(1+\ln({1+i_k-j_k})).
\end{split}\end{align}
We can then sum over all the intervals $T_k$,
\begin{align*}
\sum_{i\in I, j\in J}\frac{1}{|\sfx_i-\sfx_j|^2}
&=\sum_{k=0}^{K-1}\sum_{i\in I\cap T_k, j\in J}\frac{1}{|\sfx_i-\sfx_j|^2}
\lesssim \sum_{k=0}^{K-1}\frac{1}{\theta^2}(1+\ln({1+i_k-j_k}))\\
&\lesssim \frac{K}{\theta^2} \left(1+\ln \left(\frac{\sum_k{1+i_k-j_k}}{K}\right)\right)
\lesssim \frac{K\ln (\ell N/K)}{\theta^2}
\lesssim \frac{\zeta \ell N\ln (1 /\zeta)}{\theta^2},
\end{align*}
where for the second statement we used \eqref{e:sumxixj}; the third statement we used Jensen's inequality; the last two statements we used that $\sum_k (1+i_k-j_k)\leq K+\ell N$.
\end{proof}

In the following we prove the third case, when $\varrho\geq 1-\zeta$. 
\begin{proof}[Proof of the Third Case.]
Plugging \eqref{e:xi-xj} into \eqref{e:Wpsum}, we get
\begin{align*}\begin{split}
&\phantom{{}={}}\prod_{0\leq \sft\leq  \sfL }\prod_{1\leq i<j\leq n}\frac{(\sfx_i(\sft)+\theta e_i(\sft))-(\sfx_j(\sft)+\theta e_j(\sft))}{{\sfx_i(\sft)}-{\sfx_j(\sft)}} \\
&=e^{\OO(-\zeta\ln(\zeta)  \sfL \ell N)}\prod_{0\leq \sft\leq  \sfL }\left(\prod_{1\leq i<j\leq n}\frac{(\sfx_i(\sft)+ e_i(\sft))-(\sfx_j(\sft)+ e_j(\sft))}{{\sfx_i(\sft)}-{\sfx_j(\sft)}} \right)^\theta\\
&=e^{\OO(-\zeta\ln(\zeta) (\ell N)^2)}\left(\frac{V(\sfx( \sfL ))}{V(\sfx(0))}\right)^\theta.
\end{split}\end{align*}
Finally we prove the following bound on the ratio of the Vandermonde determinant
\begin{align}\label{e:weightp2}
\frac{V(\sfx( \sfL ))}{V(\sfx(0))}\leq e^{\OO(\zeta(\ell N)^2)}.
\end{align}
The claim \eqref{e:contribution} follows from plugging \eqref{e:weightp}, \eqref{e:weightp2} into \eqref{e:Wpsum}, and using \eqref{e:totalp}.
To prove \eqref{e:weightp2}, we recall the three index sets $I_0,I_1,I_2$ from \eqref{e:defI012}.
It follows from showing that 
\begin{align}\label{e:farparticle}
\frac{V(\sfx( \sfL ))}{V(\sfx(0))}
\leq e^{\OO(\zeta (\ell N)^2)} \frac{V(\sfx( \sfL )|_{I_1})}{V(\sfx(0)|_{I_1})},
\end{align}
the relation $V( \theta, 2\theta,\cdots, |I_1|\theta)\leq V(\sfx( \sfL )|_{I_1})$, and 
\begin{align}\label{e:addempty}
\frac{V(\sfx( \sfL )|_{I_1})}{V( \theta, 2\theta,\cdots, |I_1|\theta)}=e^{\OO(\zeta(\ell N)^2)}.
\end{align}

The claim \eqref{e:farparticle} follows from the following one time estimate 
\begin{align}\label{e:onestep}
&\phantom{{}={}}\frac{V(\sfx( \sfL ))}{V(\sfx(0))} \frac{V(\sfx(0)|_{I_1})}{V(\sfx( \sfL )|_{I_1})}=\prod_{(i,j)\notin I_1\times I_1}\frac{\sfx_i( \sfL )-\sfx_j( \sfL )}{{\sfx_i(0)-{\sfx_j(0)}}}
\leq \phantom{{}={}}\prod_{i\in I_0\cup I_2, j\in \qq{1,n} \atop i<j} \left(1+\frac{2 \sfL }{|\sfx_i(0)-{\sfx_j(0)|}}\right) \\
&\leq \prod_{i\in I_0\cup I_2, j\in \qq{1,n}\atop i<j}  \left(1+\frac{2 \sfL }{|i-j|\theta}  \right)
\leq \prod_{i\in I_0\cup I_2} \prod_{1\leq k\leq \ell N} \left(\frac{3\ell N}{k\theta}  \right)^2
=e^{\OO((|I_0|+|I_2|)\ell N)}=e^{\OO(\zeta(\ell N)^2)},
\end{align}
where we used that $|\sfx_i(0)-\sfx_j(0)|\geq |i-j|\theta$ and $n\leq \ell N$.
Next we prove \eqref{e:addempty}. Since $\|\cH-\cA\|_\infty\leq \varepsilon$ and $1-\zeta\del_x \cA\leq 1$, we have that up to some shift $\sfx( \sfL )|_{I_1}$ is obtained from the particle configuration $\theta, 2\theta,\cdots, |I_1|\theta$ by adding at most $\OO(\zeta \ell N)$ empty sites. Next we show that adding an empty site increases the Vandermonde determinant by at most a factor $e^{\OO(\ell N)}$, then \eqref{e:addempty} follows. Say from the configuration $x_1> x_2> \cdots> x_{|I_1|}$, we add an empty site between $x_k$ and $x_{k+1}$, and get the particle configuration $x_1+1>x_2+1>\cdots>x_k+1>x_{k+1}>\cdots x_{|I_1|}$, then
\begin{align*}
\frac{V(x_1+1, \cdots, x_k+1, x_{k+1},\cdots x_{|I_1|})}{V(x_1, x_2,\cdots, x_{|I_1|})}
&=\prod_{i\leq k, j\geq k+1}\left(1+\frac{1}{x_i-x_j}\right)
\leq \prod_{i\leq k, j\geq k+1}\left(1+\frac{1}{|i-j|\theta}\right)\\
&\leq e^{\sum_{i\leq k, j\geq k+1}\frac{1}{|i-j|\theta}}
\leq e^{\frac{1}{\theta}\left(\frac{1}{1}+\frac{2}{2}+\cdots +\frac{\ell N}{\ell N}\right)}
=e^{\ell N/\theta},
\end{align*}
where we used that the multiset $\{|i-j|: i\leq k, j\geq k+1\}$ contains at most $m$ copies of $m$ for any $1\leq m\leq \ell N$.
\end{proof}

\subsection{Complex Slope Estimates}
\label{s:heightslope}

In this section,  we prove that the complex slope $\wt f_t(x)$ and the drift function $g_t(x)$ (as constructed in \eqref{fh} and \eqref{e:gtx}) can be extended analytically to a strip neighborhood of the real axis in \Cref{l:gtestimate}. This step is crucial to use the dynamical loop equations, since for instance \Cref{t:loopeq} requires the functions $\phi^{\pm}$ to be holomorphic.

In \Cref{c:nabHbound} and \Cref{l:kappaHib}, we collect some estimates of the smoothed height function $\wt H(x,t)$ (as constructed in \eqref{e:deftH}) . Their proofs follow from explicit computations, so we postpone them to \Cref{s:setupproof}.



%

\begin{lemma}\label{c:nabHbound}
Fix any $0\leq t\leq \ell$, $\delta\ll \ell$, and  denote 
\begin{align}\label{e:defkappat}
\kappa_t(x)=\frac{1}{\pi}\int_{tv}^{\ell+tv} \frac{\delta \rd y}{(y-x)^2+\delta^2},
\end{align}
then 
\begin{align}\begin{split}\label{e:dtH}
 \partial_x \widetilde H(x,t)=\varrho \kappa_t(x),
 \quad \partial_t \widetilde H(x,t)=-\varrho v \kappa_t(x).
\end{split}\end{align}
The function $\kappa_t(x)\in [0,1]$ and for $x\in [tv, \ell+tv]$, it holds
\begin{align*}
\kappa_t(x)\asymp 1,\quad 1-\kappa_t(x) \asymp \frac{C\delta}{\delta+\dist(x, \{tv, \ell+tv\})},
\end{align*}
for $x\not\in [tv, \ell+tv]$, it holds
\begin{align}\label{e:ktxbound}
1-\kappa_t(x)\asymp 1,\quad  \kappa_t(x) \asymp  \frac{C\delta}{\delta+\dist(x, \{tv, \ell+tv\})+\dist(x, \{tv, \ell+tv\})^2/\ell}.
\end{align}
%
\end{lemma}

We recall that the Hilbert transform $\Hib(u)$ of a function $u: \bR\mapsto \bR$ is given by $\Hib(u)(t)=\PV\int_\bR u(x)/(t-x)\rd x$ where $\PV$ denotes the principal value. 
\begin{lemma}\label{l:kappaHib}
There exists a large constant $C>0$, for any small $0<\delta\ll \ell$, recall $\kappa_t(x)$ from \eqref{e:defkappat}, the following holds
\begin{enumerate}
\item $\kappa_t(x)$  can be extended analytically to any $z=x+\ri \eta$ with $\Im[\eta]\leq \delta/2$, and
\begin{align}\label{e:ktdiff}
|\kappa_t(z)-\kappa_t(x)|\leq \frac{C\eta}{\delta}\min\{\kappa_t(x), 1-\kappa_t(x)\}.
\end{align}
\item
The Hilbert transform $\Hib(\kappa_t)$ extends analytically to  any $z=x+\ri\eta$ with $\Im[\eta]\leq \delta/2$,
\begin{align}\label{e:Hkappa}
\left|\Hib( \kappa_t)(z)\right|\leq
\left\{
\begin{array}{cc}
\frac{C\delta}{\dist(z, \{tv, \ell+tv\})} &\text{if $\dist(z, \{tv, \ell+tv\})\geq \ell$}\\
\ln(\ell/\delta)+C & \text{for all $z$}
\end{array}
\right.,
\quad \left|\Im \Hib(\kappa_t)(z)\right|\leq \frac{C\eta}{\delta}.
\end{align}
\end{enumerate}
\end{lemma}
If $\dist(z, \{tv,\ell+tv\})\gtrsim \ell$,  we have
\begin{align}\label{e:largezbound}
\del_z \kappa_t(z)\lesssim \frac{\delta \ell}{\dist(z, \{tv, \ell+tv\})^3},
\quad\Hib(\kappa_t(z)\lesssim \frac{\delta}{\dist(z, \{tv, \ell+tv\})}.
\end{align}

We recall the complex slope $\widetilde f_t(x)$ and drift function $g_t(x)$ from 
\eqref{e:ftx} and \eqref{e:gtx}
\begin{align}\begin{split}\label{e:ftgtx}
& \widetilde f_t(x)=e^{-\ri\pi \partial_x  \widetilde H(x,t)}\exp(\ln \sin(-\pi \partial_t  \widetilde H(x,t))- \ln \sin(\pi \partial_x  \widetilde H(x,t)+\pi\partial_t  \widetilde H(x,t))),\\
&g_t(x)=\ln \sin(-\pi \partial_t  \widetilde H(x,t))- \ln \sin(\pi \partial_x  \widetilde H(x,t)+\pi\partial_t  \widetilde H(x,t))-\Hib( \partial_x  \widetilde H(x,t)).
\end{split}\end{align}
The following lemma states that the complex slope $\wt f_t(x)$ and the drift function $g_t(x)$ can be extended analytically to a strip neighborhood of the real axis. Its proof follows from explicit computations, so we postpone it to \Cref{s:setupproof}.

\begin{lemma}\label{l:gtestimate}
There exists a large constant $C>0$, for any $0<\delta\ll\ell$, $\zeta\ll 1$, extreme slope $(\varrho, -\varrho v)\not\in \cT_\zeta$ (recall from \eqref{e:defcT}),  $\widetilde f_t(x)$ and $g_t(x)$ as in \eqref{e:ftgtx} can be extended analytically to the strip region 
$\cD=\cD(C)=\{x+\ri \eta: |\eta|\leq C^{-1}\delta\}$, and for any $z\in \cD$ the following holds
\begin{enumerate}
\item The norm of $\widetilde f_t(z)$ satisfies: $|\widetilde f_t(z)|\leq C/\zeta$.
\item The norm of $g_t(z)$ satisfies: $|g_t(z)|\leq \ln(1/\zeta)+\ln (\ell/\delta)+C$. If $\dist(z, \{tv, \ell +tv\})\geq \ell$, we have 
\begin{align}\label{e:glargez}
\ln \frac{\sin(\pi \varrho v \kappa_t(z))}{ \sin(\pi \varrho(1-v)\kappa_t(z))}
=\ln\frac{\sin( v)}{ \sin(1-v)}+\OO\left(\frac{\delta}{\dist(z, \{tv, \ell+tv\})}\right).
\end{align}
The imaginary part of $g_t(z)$ satisfies: $|\Im[g_t(z)]|\leq C\Im[z]/\delta$.
\item The derivative of $g_t(x)$ satisfies: for any $x\in \bR$, $|\del_x g_t(x)|, |\delta_t g_t(x)|\leq C/\delta$. 
\end{enumerate}
\end{lemma}

\subsection{Proof of \Cref{p:ubb}: Interior Slope Case}
In this section we prove \Cref{p:ubb}, which implies \Cref{p:LDP} for the interior slope case. It is a consequence of the following two lemmas.

\label{s:interiorslope}
\begin{lemma}\label{l:Nterm}
Adopt the notations and assumptions in \Cref{p:ubb}. 
Let $\cH$ be a height function with $\|\cH-\cA\|_\infty\leq \varepsilon$, and  denote $\{\bmx(t)\}_{0\leq t\leq \ell}$ the particle configuration associated with $\cH(x,t)$, then 
\begin{align}\begin{split}\label{e:expt1}
&\phantom{{}={}}\sum_{\sft=0}^{\sfL-1}{\sum_{i=1}^N e_i(\sft) g_t(x_i(\sft/N))}
=-\frac{N^2}{\theta}\int_0^{\ell} \int_\bR \del_t \cH(x, t)g_t(x)\rd x\rd t +\OO(N(\|g\|_\infty+\|\del_t g\|_\infty))\\
&=-\frac{N^2}{\theta}\int_0^{\ell} \int_\bR \del_t \cA(x, t)g_t(x)\rd x\rd t +\OO((\varepsilon/\delta)  (\ell N)^2+\varepsilon(\ln(1/\zeta)+\ln(\ell/\delta))\ell N^2).
\end{split}\end{align}
\end{lemma}

\begin{lemma}\label{l:DLE}
Adopt the notations and assumptions in \Cref{p:ubb}. 
Let $\cH(x,t)$ be a height function with $\|\cH-\cA\|_\infty\leq \varepsilon$, and  denote $\{\bmx(t)\}_{0\leq t\leq \ell}$ the particle configuration associated with $\cH(x,t)$, then for any $0\leq \sft=tN\leq \sfL-1$,
\begin{align}\begin{split}\label{e:DLE0}
    \ln   \bE[e^{\sum_i    e_i(\sft) g_t(   x_i(t)/N)}|  \bmx(t)]
    &=\frac{N}{2\pi\ri\theta}\int_0^1\oint_\omega\ln (1+e^{   m_t(z)+\tau  g_t (z)}) g_t (z)\rd z\rd \tau+\OO\left(1\right)\\
    &=\frac{N}{2\pi\ri\theta}\int_0^1\oint_\omega\ln (1+e^{   m_t^*(z)+\tau  g_t (z)}) g_t (z)\rd z\rd \tau+\OO\left(\frac{\varepsilon \ell N }{\delta\zeta^2 }\right),
\end{split}\end{align}
where $\omega$ is a contour enclosing the interval $[-3\ell, 3\ell]$, $m^{*}_{t}(z)$ is defined in \eqref{defm*} and  $ m_t(z)$ is the Stieltjes transform of the empirical measure of $   \bmx(t)$,
\begin{align*}
m_t(z)=\int_\bR\frac{\rho(x;\bmx(t))\rd x}{x-z}.
\end{align*}
\end{lemma}
\begin{proof}[Proof of \Cref{p:ubb}]
Recall from \eqref{e:Pg2} that 
\begin{align}\label{e:Pg3}
    \bE\left[\prod_{t\in\qq{0, \sfL-1}/N} \frac{e^{\sum_{i=1}^N e_i(Nt) g_t(x_i(t))}}{\bE[e^{\sum_{i=1}^N e_i(Nt) g_t(x_i(t))}|\bmx(t)]}\bm1(\|\cH-\cA\|_\infty\leq \varepsilon)\right]\leq 1.
\end{align}
\Cref{p:ubb} follows from plugging \Cref{l:Nterm} and \Cref{l:DLE} into \eqref{e:Pg3}, the error term simplifies using the relation \eqref{e:introepsilon}.
\end{proof}

\begin{proof}[Proof of Lemma \ref{l:Nterm}]

Let $G_t(z)=\int_0^z g_t(x)dx$. We first notice that, for $t=\sft/N$,
\begin{align*}\begin{split}
    &\frac{N^2}{\theta}\int_\bR (\rho(  x;\bmx(t+1/N))-\rho(  x;\bmx(t/N)))G_t(  x)\rd   x\\
    &=\sum_i   \frac{Ne_i(\sft)}{\theta} \int_{  x_i(t)}^{  x_i(t)+\theta/N}(G_t(  x+1/N)-G_t(  x))\rd   x=\sum_i   e_i(\sft) g_t(  x_i(t))+\OO\left(\|g_t\|_\infty+\|\del_t g\|_\infty\right).
\end{split}\end{align*}
Thus for the lefthand side of \eqref{e:expt1}, we have 
\begin{align}\begin{split}\label{e:replaceH}
&\phantom{{}={}}\sum_{\sft=0}^{\sfL-1}{\sum_{i=1}^N e_i(\sft) g_t(x_i(\sft)/N)}\\
&=\sum_{\sft=0}^{\sfL-1}\frac{N^2}{\theta}\int_\bR (\del_x \cH(x, t+1/N)-\del_x \cH(x, t))G_t(x)\rd x +\OO\left(N\|g_t\|_\infty\right)\\
&=-\sum_{\sft=0}^{\sfL-1}\frac{N^2}{\theta}\int_\bR (\cH(x, t+1/N)- \cH(x,t))g_t(x)\rd x +\OO\left(N\|g_t\|_\infty\right)\\
&=-\sum_{\sft=0}^{\sfL-1}\frac{N^2}{\theta}\int_{t}^{t+1/N} \del_s \cH(x, s) \rd s g_{t}(x)\rd x +\OO\left(N\|g_t\|_\infty\right)\\
&=-\frac{N^2}{\theta}\int_0^{\ell} \int_\bR \del_t \cH(x, t)g_t(x)\rd x\rd t +\OO(N(\|g_t\|_\infty+\|\del_t g\|_\infty)).
\end{split}\end{align}
This gives the first statement in \eqref{e:expt1}.
Next we replace $\cH$ by $\cA$, we can do an integration by part,  use $\|\cH-\cA\|_\infty\leq \varepsilon$, and that integration holds in the parallelogram $[0,\ell]\times [tv,tv+\ell]$, to deduce that
\begin{align}\begin{split}\label{e:replace}
&\phantom{{}={}}\left|\frac{N^2}{\theta}\int_0^{\ell} \int_\bR (\del_t \cH(x, t)-\del_t \cA(x, t))g_t(x)\rd x\rd t\right|\\
&\leq \frac{N^2}{\theta}\left|\left. \int_\bR (\cH(x,t)-\cA(x,t))g_t(x)\rd x\right|_0^\ell \right|
+\frac{N^2}{\theta}\left|\int_0^{\ell} \int_\bR (\cH(x, t)- \cA(x, t))\del_t g_t(x)\rd x\rd t\right|\\
&\lesssim \varepsilon \|g\|_\infty \ell N^2+\varepsilon \|\del_t g\|_\infty (\ell N)^2.
\end{split}\end{align}
Thanks to \Cref{l:gtestimate}, we have $\|g\|_\infty\leq\ln(1/\zeta)+\ln(\ell/\delta)+C$ and $|\del_t g_t(x)|\leq C/\delta$, and the second statement of \eqref{e:expt1} follows from \eqref{e:replace}.

\end{proof}

In the following, we prove \Cref{l:DLE}. For simplicity of notations, we will omit the dependence on time $t$. 
We can rewrite the lefthand side of \eqref{e:DLE0} as
\begin{align}\begin{split}\label{e:rewrite_as_integral}
  \ln   \bE[e^{\sum_i    e_i  g (   x_i/N)}|  \bmx]
    &= \int_0^1\del_\tau\ln   \bE[e^{\sum_i    e_i \tau  g (   x_i/N)}|  \bmx]\rd \tau\\
    &=\int_0^1\frac{  \bE[\sum_i    e_i  g (   x_i/N) e^{\sum_i    e_i \tau  g (   x_i/N)}|  \bmx]}{  \bE[e^{\sum_i    e_i \tau  g (   x_i/N)}|  \bmx]}\rd \tau,
\end{split}\end{align}
which is the expectation of  $\sum_i e_i   g (   x_i/N)$ under the following deformed measure
\begin{align}\label{e:Etheta}
    \frac{1}{Z_\tau}\frac{V(  \bmx+  \theta \bme/N)}{V(  \bmx)}e^{\sum_i    e_i \tau  g (   x_i/N)},
\end{align}
where $Z_\tau$ is the normalization factor. We denote $  \bE_\tau[\cdot]$ the expectation with respect to the measure \eqref{e:Etheta}.

We will use the dynamical loop equation \Cref{t:loopeq}  and \Cref{t:loopstudy} to compute the expectation
\begin{align*}
 \frac{  \bE[\sum_i    e_i  g (   x_i/N) e^{\sum_i    e_i \tau  g (   x_i/N)}|  \bmx]}{  \bE[e^{\sum_i    e_i \tau  g (   x_i/N)}|  \bmx]}= \bE_\tau\left[\sum_i e_i  g (   x_i/N)\right].
\end{align*}

Take a large constant $8C/(\delta \zeta)\leq K\leq \zeta/(2\varepsilon)$, where the constant $C>0$ is from \Cref{l:gtestimate}. This is possible thanks to our choice of parameters from \eqref{e:introepsilon}.
We denote the annulus region
\begin{align}\label{e:defLambda}
    \Lambda=\Lambda_{K}=\{z\in \bC: (1/K)\leq \dist(z, [-3\ell, 3\ell])\leq 2/K\}.
\end{align}
We need to verify the conditon that on $\Lambda$, the following are well defined, and lower and upper bounded 
\begin{align*}
\ln   (1+ f^{(\tau)}(z)), \quad  f^{(\tau)}(z):= e^{  m(z)+\tau  g (z)}.
\end{align*}
\begin{lemma}\label{l:Bzcondition}
Take $8C/(\delta\zeta)\leq K\leq \zeta/(2\varepsilon)$, then on $\Lambda$ from \eqref{e:defLambda}, we have 
\begin{align}\label{e:ftaubound}
    \sin\left(\frac{\zeta\pi}{4}\right)\leq |  1+ f^{(\tau)}(z)|\lesssim \frac{\ell^2 K}{\delta\zeta}, \quad \frac{| f^{(\tau)}(z)|}{|1+ f^{(\tau)}(z)|}\lesssim \frac{1}{\zeta},
\end{align}
and $\ln (1+ f^{(\tau)}(z))$ is well defined.
\end{lemma}
\begin{proof}
We first show that for any $z\in \Lambda$ 
\begin{align}\label{e:Immbound}
  -\pi(1-\zeta/2)\leq \Im[ m(z)]\leq \pi(1-\zeta/2).
\end{align}
Since $\Im[ m(z)]=-\Im[ m(\overline z)]$, we will only prove \eqref{e:Immbound} for $z\in \Lambda\cap \overline\bH$. Thanks to \Cref{c:nabHbound}, $\cH(\pm 3\ell, t)=\cA(\pm 3\ell, t)$. Therefore, for any $z\in\Lambda$, we find
\begin{align}\begin{split}\label{e:mm*diff}
|m(z)-m^*(z)|
&=\left|\int_\bR\frac{(\del_x \cH(x,t)-\del_x \cA(x,t))\rd x}{x-z}\right|
\leq \int_{-3\ell}^{3\ell}\frac{| \cH(x,t)- \cA(x,t)|\rd x}{|x-z|^2}\\
&\leq \varepsilon \int_{-\infty}^\infty \frac{\rd x}{x^2+(1/K)^2}=\pi K\varepsilon\leq \zeta\pi/2,
\end{split}\end{align}
where in the last inequality we used that $K\leq\zeta/(2\varepsilon)$. 
We can also compute $m^*(z)$ and $\Im[ m^*(z)]$ explicitly
\begin{align}\label{e:Imtm}
    m^*(z)
    =\int_\bR\frac{ \del_x \cA\rd x}{z-x}
    =\int_{tv\leq x\leq \ell+tv}\frac{ \varrho\rd x}{z-x} 
    =-\varrho\ln\left(1-\frac{\ell}{z-tv}\right).
\end{align}
and 
\begin{align}\label{e:Imtm2}
    -\Im[ m^*(z)]
    =\eta\int_\bR\frac{ \del_x \cA\rd x}{|z-x|^2}
    =\eta\int_{tv\leq x\leq 1+tv}\frac{ \varrho\rd x}{|z-x|^2}\in [0,\pi \varrho]\subset [0, (1-\zeta)\pi].
\end{align}
The claim \eqref{e:Immbound} follows from combining \eqref{e:mm*diff} and \eqref{e:Imtm2}.
From \Cref{l:gtestimate}, $|\Im g(z)|\leq C\Im[z]/\delta \leq 2C/(\delta K)$. 
Using \eqref{e:Immbound}, for any $z\in \Lambda$, we have  $|\Im g(z)|\leq 2C/(\delta K)\leq \zeta\pi/4$, it follows that
\begin{align*}
-(1-\zeta/4)\pi\leq \Im[ m(z)+\tau g(z)]\leq (1-\zeta/4)\pi.
\end{align*}
It follows that
$
    \arg e^{ m(z)+\tau g(z)}\in (-(1-\zeta/4)\pi, (1-\zeta/4)\pi).
$
Thus $1+e^{ m(z)+ \tau g(z)}$ takes value in the sector region given by $\{1+re^{\ri u}: u\in (-(1-\zeta/4)\pi, (1-\zeta/4)\pi)\}$, which does not include the negative axis. Therefore, we can take $\ln(\cdot)$ the branch defined on $\bC\setminus \bR_{\leq 0}$ and $\ln(1+e^{ m(z)+ \tau g(z)})$ is well defined. Moreover 
\begin{align*}
    \min_{u\in (-(1-\zeta/4)\pi, (1-\zeta/4)\pi)}|1+re^{\ri u}|
    \geq \sin \left(\frac{\zeta \pi}{4}\right),
\end{align*}
we conclude that for $z\in \Lambda$, it holds $|1+e^{ m(z)+ \tau g(z)}|\geq \sin (\zeta \pi/4)\gtrsim \zeta$. The last statement of \eqref{e:ftaubound} also follows.
On $\Lambda$, since $\dist(z, [-3\ell, 3\ell])\geq 1/K$, we have
\begin{align*}
    | e^{ m(z)+ \tau g(z)}|\leq e^{\Re m^*(z)+|m^*(z)-m(z)| +\|g\|_\infty}\leq \left|1-\frac{\ell}{z-tv}\right|e^{\pi K \varepsilon+\|g\|_\infty}\leq (1+\ell K)e^{\pi K \varepsilon+\|g\|_\infty},
\end{align*}
where we used \eqref{e:Imtm} to bound $|\Re m^*(z)|$ and \eqref{e:mm*diff} to bound $|m(z)-m^{*}(z)|$. Using \eqref{e:Immbound}, for any $z\in \Lambda$, we have  $|g(z)|\leq \ln(1/\zeta)+\ln(\ell/\delta)+C$
thus $| 1+e^{ m(z)+ \tau g(z)}|\lesssim \ell^2 K/(\zeta\delta)$. This finishes the proof of \eqref{e:ftaubound}.
\end{proof}

\begin{proof}[Proof of \Cref{l:DLE}]

\Cref{l:Bzcondition} with $K=8C/\delta\zeta$ verifies that the measure \eqref{e:Etheta} satisfies the assumption in \Cref{t:loopstudy} (by taking $\varepsilon, n$ there to be $1/N,  n$ here). Let $\bmy=\bmx+\bme/N$. We conclude that
 \begin{align}\label{e:dmg2}
\frac{N}{\theta}\int_{\bR} \frac{(\rho(s;\bmy)-\rho(s;\bmx))}{z-s}\rd s
=\Delta\cM(z)+\frac{1}{2\pi \ri}\oint_{\cin}\frac{\ln (1+e^{ m(w)+ \tau g(w)})\rd w}{(w-z)^2}
+ \OO\left(\frac{1}{N} \right),
\end{align}
where the contour $\cin\subset \Lambda$ encloses $[-3\ell, 3\ell]$, but not $z$, the implicit constant in the error term depends on $K, \delta, \ell$. Moreover, $\Delta\cM(z)$ is a random variable with mean zero under $\bE_{\tau}$.
By a contour
 integral with respect to $G(z):=\int_0^{z}g(w)\rd w$ on both sides and taking expectation we get
\begin{align}\label{e:rhodiff}
    \frac{N}{\theta}\bE_\tau\int_{\bR} (\rho(s;\bmy)-\rho(s;\bmx))G(s)\rd s
    =\frac{1}{2\pi \ri}\oint_{\omega}\ln(1+e^{ m(w)+ \tau g(w)}) g(w)\rd w
+ \OO\left(\frac{1}{N} \right),
\end{align}
where $\omega$ is a contour inside $\Lambda$. For the lefthand side of \eqref{e:rhodiff}, we have
\begin{align*}
    \frac{N}{\theta}\bE_\tau \int_{\bR} (\rho(s;\bmy)-\rho(s;\bmx))G(s)\rd s
    &=\bE_\tau\sum_i   \frac{e_i}{\theta} \int_{  x_i/N}^{  (x_i+\theta)/N}(G(s+1/N)-G(s))\rd s\\
    &=\frac{1}{N}\bE_\tau\left[\sum_i   e_i g(  x_i/N)\right]+\OO\left(\frac{1}{N}\right).
\end{align*}
It follows that  
\begin{align}\label{e:linearst}
    \frac{1}{N}  \bE_\tau\left[\sum_i    e_i  g (   x_i/N)\right]
    =\frac{1}{2\pi\ri\theta}\oint_\omega \ln (1+e^{ \theta  m(z)+\tau  g (z)}) g (z)\rd z+\OO\left(\frac{1}{N}\right).
\end{align}
This gives the first statement in \eqref{e:DLE0} by integrating from $\tau=0$ to $\tau=1$ and using \eqref{e:rewrite_as_integral}.

Next we replace $m(z)$ in the integral \eqref{e:linearst} by $m^*(z)$.
Thanks to \eqref{e:mm*diff}, on the contour we have $|m(z)-m^*(z)|\leq \varepsilon K\pi\leq 1/2$, thus 
\begin{align}\begin{split}\label{e:logdiffhi}
&\phantom{{}={}}|\ln (1+e^{m^*(z)+\tau g(z)})-\ln (1+e^{m(z)+\tau g(z)}))|
\leq 
\left|\ln \left(1-\frac{(e^{m^{*}(z)-m(z)}-1)  f^{(\tau)}(z)}{1+e^{m^*(z)+\tau g(z)}}\right) \right|\\
&\leq \left|\frac{(e^{m^{*}(z)-m(z)}-1)f^{(\tau)}(z)}{1+e^{m^*(z)+\tau g(z)}}\right|
\lesssim \frac{\varepsilon K}{\zeta} \lesssim \frac{\varepsilon }{\delta \zeta^2 },
\end{split}\end{align}
where we used \Cref{l:Bzcondition} and take $K=8C/\delta\zeta$ in the last inequality. The second statement in  \eqref{e:DLE0} follows by plugging \eqref{e:logdiffhi} into \eqref{e:linearst}, and noticing that the contour integral over $\omega$ gives a factor $\OO(\ell)$.

\end{proof}

\subsection{Rate Function and Proof of \Cref{p:entropy}}\label{s:ratef}
We recall the rate function $S(\cA;g)$  from \eqref{e:rateS}
\begin{align*}\begin{split}
    S(\cA;g)&=
     -\int_0^{\ell} \int_\bR \del_t \cA(x,t) g_t(x)\rd x \rd t-\frac{1}{2\pi\ri}\int_0^{\ell}\int_0^1\oint_\omega\ln (1+e^{m^*_t(z)+\tau g_t(z)})g_t(z)\rd z\rd \tau\rd t,
\end{split}\end{align*}
where the contour $\omega\in \Lambda$ (from \eqref{e:defLambda}) encloses $[-3\ell,3\ell]$.
In this section, we simplify $S(\cA;g)$ and prove \Cref{p:entropy}.

%
%

%

\begin{lemma}\label{e:tmm*diff}
If $\dist(z, \{tv, \ell+tv\})\gtrsim \delta$, then 
\begin{align*}
|\widetilde m_t(z)-m_t^*(z)| \lesssim \frac{\delta \varrho}{\dist(z, \{tv, \ell+tv\})},
\end{align*}
\end{lemma}

\begin{lemma}\label{l:angleest}
We introduce a new complex slope as $f^*_t(z)=e^{m^*_t(z)+ g_t (z)}\in \bH^-$. Then for $x\in [tv, tv+\ell]$, 
\begin{align}
\arg f^*_t(x)=-\varrho\pi,\quad
|\arg (1+f^*_t(x))+\varrho v\pi|\lesssim \frac{\delta}{\delta + \dist(x, \{tv, \ell+tv\})}. 
\end{align}
\end{lemma}
\begin{proof}[Proof of \Cref{e:tmm*diff}]
The expressions of $m_t^*(x)$ and $\wt m_t(x)$ are explicit,
\begin{align}
|\widetilde m_t(z)-m_t^*(z)| 
&=\varrho \left|\ln\left(1+ \frac{\ell}{tv-z}\right)-\ln\left(1+ \frac{\ell}{tv-z-\ri \delta}\right)\right|\nonumber\\
&=\varrho \left|\ln\left(1- \frac{\ri \delta}{tv-z}\right)-\ln\left(1- \frac{\ri\delta}{\ell+tv-z}\right)\right|,\label{mm}
\end{align}
from which the result follows from $|\ln(1+x)|\lesssim |x|$.
\end{proof}

\begin{proof}[Proof of \Cref{l:angleest}]
We recall $\kappa_t(x)$ from \eqref{e:defkappat},  $g_t(x)$ from \eqref{e:gtx}, and  $\widetilde f_t(x)=e^{\widetilde m_t(x)+g_t(x)}$ from \eqref{e:ftx}.
The three vertices $\{0,-1,\wt f_t\}$ form a triangle, with three angles give by $\varrho\kappa_t(x)\pi, \varrho v\kappa_t(x)\pi, \varrho(1-v)\kappa_t(x)\pi$. From \Cref{c:nabHbound}, for $x\in [tv, \ell+tv]$, 
\begin{align}\label{e:kappabound}
\kappa_t(x)\asymp 1,\quad 1-\kappa_t(x) \asymp \frac{C\delta}{\delta+\dist(x, \{tv, \ell+tv\})},
\end{align}
It follows that $|1+\wt f_t(x)|\gtrsim \varrho$. 
We can rewrite $f^*_t(x)$ in terms of $\wt f_t(x)$
\begin{align*}
f^*_t(x)=\wt f_t(x) e^{m_t^*(x)-\wt m_t(x)}.
\end{align*}
$m_t^*(x)$ and $\wt m_t(x)$ are explicit, satisfy \eqref{mm} and  \Cref{e:tmm*diff}. For $\dist(x, \{tv, \ell+tv\})\gtrsim \delta$ we have
\begin{align*}
\arg (1+f^*_t(x))
&=\arg (1+ \wt f_t(x) e^{m_t^*(x)-\wt m_t(x)})\\
&=\arg ((1+ \wt f_t(x)) e^{m_t^*(x)-\wt m_t(x)} -(e^{m_t^*(x)-\wt m_t(x)}-1))\\
&=\arg (1+ \wt f_t(x)) +\arg \left(1- \frac{(e^{m_t^*(x)-\wt m_t(x)}-1)}{(1+ \wt f_t(x)) }\right)\\
&=-\varrho v \kappa_t(x)\pi+\OO\left(\frac{|m_t^*(x)-\wt m_t(x)|}{|1+ \wt f_t(x)|}\right)
=-\varrho v\pi +\OO\left(\frac{\delta}{\dist(x, \{tv, \ell+tv\})}\right),
\end{align*}
where in the last equality we used \eqref{e:kappabound} and \Cref{e:tmm*diff}. 
\end{proof}

\begin{proof}[Proof of \Cref{p:entropy}]
We examine the second term on the righthand side of \eqref{e:rateS}. We recall the dilogarithm function (see \cite{zagier2007dilogarithm}), 
\begin{align*}
    \Li_2(z)=-\int_0^z \ln(1-u)\frac{\rd u}{u}, \quad \Li_2'(x)=-\frac{\ln(1-x)}{x}
\end{align*}
which can be analytically extended to the cut plane $\bC\setminus [1,\infty]$.
By a change of variable,
\begin{eqnarray*}
    \Li_2(-e^w)
    &=&-\int_{0}^{-e^w} \ln(1-u)\frac{\rd u}{u}\\
    &=&-\int_{-\infty}^{w} \ln(1+e^x)\rd x,
\end{eqnarray*}
which is analytic on the strip $\{w\in \bC: -\pi<\Im[w]<\pi\}$.
Then, we can rewrite the last term in \eqref{e:rateS} by noticing that
\begin{eqnarray*}
    -\int_0^1\ln (1+e^{m^*_t(z)+\tau  g_t (z)}) g_t (z)\rd \tau
   &=&\int_{m^*_t(z)}^{m^*_t(z)+ g_t (z)}-\ln(1+e^x)\rd x\\
    &=&\Li_2(-e^{m^*_t(z)+ g_t (z)})-\Li_2(-e^{m^*_t(z)}).
\end{eqnarray*}
For the last term, since $m^*_t(z)=\ell \varrho/z+\OO(1/z^2)$ as $z\rightarrow \infty$. Thus $-e^{m^*_t(z)}=-1-\ell \varrho/z +\OO(1/z^2)$ when $z\rightarrow \infty$. Thus the contour integral is given by
\begin{align*}
    -\frac{1}{2\pi \ri}\oint \Li_2(-e^{m^*_t(z)})\rd z
    &=
    -\frac{1}{2\pi \ri}\oint \Li_2\left(-1-\frac{\ell \varrho}{z}+\OO\left(\frac{1}{z^2}\right)\right)\rd z\\
    &=-\frac{1}{2\pi \ri}\oint \Li_2\left(-1-\ell \varrho w+\OO\left(w^2\right)\right)\frac{1}{w^2}\rd w
    =\ell \varrho \Li_2'(-1)=\ell \varrho \ln(2).
\end{align*}
For the first term  in \eqref{e:rateS} we need to evaluate the integral
\begin{align}\label{e:Literm}
\frac{1}{2\pi\ri}\oint_\omega \Li_2(-e^{m^*_t(z)+ g_t (z)})\rd z,
\end{align}
where the contour $\omega\in \Lambda$ (from \eqref{e:defLambda}) encloses $[tv,tv+\ell]$.

We recall from \cite{zagier2007dilogarithm}, the dilogarithmic function function $\Li_2(z)$ jumps by $2\pi \ri \log |z|$, as $z$ crosses the cut. And the function $\Li_2(z) + \ri \arg(1-z)\log |z|$, where $\arg$
denotes the branch of the argument lying between $-\pi$ and $\pi$, is continuous. Its imaginary part gives the Bloch-Wigner function $D(z)$
\begin{align}\label{e:defDz}
  D(z)=\Im[\Li_2(z)]+\arg(1-z)\log|z|.  
\end{align}
Bloch-Wigner function $D(z)$ can be expressed as a single real variable
\begin{align*}
    D(z)=\frac{1}{2}\left[D\left(\frac{z}{\overline{z}}\right)+D\left(\frac{1-1/z}{1-1/\overline{z}}\right)+D\left(\frac{1-\overline z}{1-{z}}\right)\right],
\end{align*}
and 
\begin{align*}
    D(e^{\ri\theta})=\Im[\Li_2(e^{\ri\theta})]=2L(\theta/2),
\end{align*}
where $L$ is the Lobachevsky function from \eqref{sigmal}. We recall  the complex slope $f^*_t(z)=e^{m^*_t(z)+ g_t (z)}\in \bH^-$ from \Cref{l:angleest}, and denote the three angle of the triangle $\{0,-1,f^*_t(x)\}$ as $\pi p_1, \pi p_2, \pi p_3$. Then \Cref{l:angleest} implies that 
\begin{align}\label{e:p123}
p_1=\varrho, \quad p_2=\varrho v+\OO\left(\frac{\delta}{\delta+ \dist(x, \{tv, \ell+tv\})} \right)
,\quad p_3=\varrho(1-v)+\OO\left(\frac{\delta}{ \delta+\dist(x, \{tv, \ell+tv\})} \right),
\end{align}
and it follows
\begin{align}\begin{split}\label{e:Lobachevsky}
    D(-f)
    &=\frac{1}{2}\left[D\left(\frac{f}{\overline{f}}\right)+D\left(\frac{(f+1)/f}{(\overline f+1)/\overline{f}}\right)+D\left(\frac{1+\overline f}{1+{f}}\right)\right]\\
 &=L(\pi p_1)+L(\pi p_1)+L(\pi p_3)
 =\pi \sigma(\varrho, -\varrho v)+\OO\left(\frac{\delta \log(1/\zeta)}{\delta+\dist(x, \{tv, \ell+tv\})} \right),
\end{split}\end{align}
where in the last equality, we also used that $\zeta\leq \varrho, \varrho v, \varrho(1-v)\leq 1-\zeta$, \eqref{e:Lobdiff} and \eqref{e:sigmasmall}.

With the notation above, we can rewrite \eqref{e:Literm} as 
\begin{align}\label{e:Literm2}
    \frac{1}{2\pi\ri}\oint_\omega \Li_2(-f^*_t(z))\rd z,
\end{align}
where the contour $\omega\in \Lambda$ (from \eqref{e:defLambda}) encloses $[tv,\ell+tv]$. We can deform $\omega$ to the real axis,
\begin{align}\begin{split}\label{e:Literm2}
    &\phantom{{}={}}\frac{1}{2\pi\ri}\oint_\omega \Li_2(-f^*_t(z))\rd z=-\frac{1}{\pi}\int_{\bR} \Im \Li_2(-f^*_t(x))\rd x\\
    &=-\frac{1}{\pi}\int_{\bR} (D(-f^*_t(x))-\arg(1+f^*_t(x))\log|f^*_t(x)|)\rd x,
\end{split}\end{align}
where we used \eqref{e:defDz}, $f^*_t(\overline z)=\overline{f^*_t(z)}$, and $\Li_2(\overline w)=\overline{\Li_2( w)}$.

Thanks to \eqref{e:Literm2}, we can rewrite the rate $S(\cA;g)$ from \eqref{e:rateS} as
%
%
%
\begin{align}\begin{split}\label{e:SHg}
 S(\cA;g)
&=\int_0^{\ell}\int_{tv}^{tv+\ell}\pi^{-1} \arg(1+f_t^*)\log|f_t^*|-\del_t \cA(x,t) g_t(x)\rd x\rd t\\
&-\frac{1}{\pi}\int_0^{\ell}\int_{tv}^{tv+\ell}\int_{\bR} D(-f_t^*(x))\rd x \rd t+\ell^2 \varrho\ln(2).
\end{split}\end{align}
Recall that $\log |f_t^*|=\Re m_t^*(x)+g_t(x)$, $\int_{tv}^{tv+\ell} \Re m_t^*(x)\rd x=0$, and $|\Re m_t^*(x)|\lesssim 1+\ln(\ell/\dist(x,\{tv, \ell+tv\}))$. 
We can rewrite the first term on the righthand side of \eqref{e:SHg} as
\begin{align}\begin{split}\label{e:cancel}
  &\phantom{{}={}}\int_0^{\ell}\int_{tv}^{tv+\ell} (\pi^{-1}\arg(1+f_t^*)-\del_t \cA(x,t))(\Re m_t^*(x)+g_t(x))\rd x\rd t\\
&\lesssim 
 \int_0^{\ell}\int_{tv}^{tv+\ell} \frac{\delta}{\delta + \dist(x, \{tv, \ell+tv\})} (\ln(\ell/\dist(x, \{tv, tv+\ell\}))+\ln(\ell/\delta)+\ln(1/\zeta))\rd x\rd t.
\end{split}\end{align}
where we used \Cref{l:angleest} and \Cref{l:gtestimate} .
We divide the integral above into two cases: i) $\dist(x, \{tv, \ell+tv\})\geq \delta$ and ii) $\dist(x, \{tv, \ell+tv\})\leq \delta$:
\begin{align*}
\int_0^\ell\int_{\dist(x, \{tv, \ell+tv\})\geq \delta} \frac{\delta \ln(\ell/\delta)}{ \dist(x, \{tv, \ell+tv\})} \rd x\lesssim \delta \ln(\ell/\delta)^2\ell ,
\end{align*}
and 
\begin{align*}
 \int_0^{\ell}\int_{\dist(x, \{tv, \ell+tv\})\leq \delta} \frac{\delta \ln(\ell/\dist(x, \{tv, tv+\ell\}))}{\delta} \rd x\rd t\lesssim \delta \ln(\ell/\delta)\ell. 
\end{align*}

Combining the discussions above, we get that the rate function is given by
\begin{align*}
    S(\cA;g)
    &=-\frac{1}{\pi}\int_0^{\ell}\int_{\bR} D(-f_t^*(x))\rd x+\ell^2\varrho\ln(2) +\OO( \delta \ln(\ell/\delta)^2\ell )\\
   &= -\int_0^{\ell}\int_{tv}^{tv+\ell} \left(\sigma(\varrho, -\varrho v)+\OO\left(\frac{\delta \log(1/\zeta)}{\delta+\dist(x, \{tv, \ell+tv\})} \right)\right)\rd x \rd t +\ell^2\varrho\ln(2)+\OO( \delta \ln(\ell/\delta)^2\ell )\\
   &=-\ell^2 \sigma(\varrho, -\varrho v)+\ell^2 \varrho\ln(2)+\OO( \delta \ln(\ell/\delta)^2\ell ),
\end{align*}
where in the second equality we used
\eqref{e:Lobachevsky}. This finishes the proof of \Cref{p:entropy}.

\end{proof}

\begin{proof}[Proof of \Cref{l:offdiagonal}]
For any $\beta,\al<\al'$, denote $I=I(\alpha, \beta)$ and $J=I(\alpha', \beta)$. Next we compute the interaction between different blocks. Thanks to \eqref{e:xi-xj}, for $\al\sfL\leq \sft<(\al+1)\sfL$,
\begin{align}\begin{split}\label{e:sumblock}
\sum_{i\in I(\alpha, \beta), j\in I(\alpha', \beta)}\ln\left(1+\frac{\theta(e_i(\sft)-e_j(\sft))}{\sfx_i(\sft)-\sfx_j(\sft)}\right)
&=\sum_{i\in I(\alpha, \beta), j\in I(\alpha', \beta)}\theta\ln\left(\frac{\sfx_i(\sft+1)-\sfx_j(\sft+1)}{\sfx_i(\sft)-\sfx_j(\sft)}\right)\\
&+\OO\left(\sum_{i\in I(\alpha, \beta), j\in I(\alpha', \beta)}\frac{1}{(\sfx_i(\sft)-\sfx_j(\sft))^2} \right).
\end{split}\end{align}
We can bound the total error as
\begin{align}\label{e:totalsum}
\sum_{\beta}\sum_{i\in I(\alpha, \beta), j\notin I(\alpha, \beta)}\frac{1}{(\sfx_i(\sft)-\sfx_j(\sft))^2}
\lesssim \sum_{\beta}\sum_{i\in I(\alpha, \beta), j\notin I(\alpha, \beta)}\frac{1}{(i-j)^2},
\end{align}
where we used that $|\sfx_i(\sft)-\sfx_j(\sft)|\geq \theta|i-j|$.

The upper bound on the righthand side of \eqref{e:totalsum} is achieved when particles in $I(\alpha,\beta)$ are tightly packed at $\{-|I(\alpha, \beta)|,\cdots, -1\}$, and the other particles are also tightly packed on $\{0,1,2,\cdots, N-|I(\alpha,\beta)|-1\}$. Therefore we can bound it as
\begin{align}\begin{split}\label{e:tsum}
 \sum_{\beta}\sum_{i\in I(\alpha, \beta), j\notin I(\alpha, \beta)}\frac{1}{(i-j)^2}
&\lesssim \sum_\beta \sum_{1\leq i\leq |I(\alpha,\beta)|}
\left(\frac{1}{i^2}+\frac{1}{(i+1)^2}+\cdots +\frac{1}{N^2}\right)\\
&\lesssim\sum_\beta \ln |I(\alpha,\beta)|\lesssim \frac{1}{\ell} \ln(\ell N),
\end{split}\end{align}
where in the last inequality, we used that $\sum_\beta |I(\alpha,\beta)|=N$ and Jessen's inequality.

Thanks to \eqref{e:sumblock}, \eqref{e:totalsum} and \eqref{e:tsum}, we can sum over $\beta \sfL\leq \sft<(\beta+1)\sfL$, and notice the cancellation of the first term on the righthand side of \eqref{e:sumblock} when summing over $\sft$ 
\begin{align}\begin{split}\label{e:sumbb}
&\phantom{{}={}}\sum_{\al<\al'}\sum_{\beta \sfL\leq \sft<(\beta+1)\sfL}\sum_{i\in I(\alpha, \beta), j\in I(\alpha', \beta)}\ln\left(1+\frac{\theta(e_i(\sft)-e_j(\sft))}{\sfx_i(\sft)-\sfx_j(\sft)}\right)\\
&=\sum_{\al<\al'}\sum_{i\in I(\alpha, \beta), j\in I(\alpha', \beta)}\theta\ln\left(\frac{\sfx_i((\beta+1)\sfL)-\sfx_j((\beta+1)\sfL)}{\sfx_i(\beta \sfL)-\sfx_j(\beta\sfL)}\right)+\OO\left(N\ln (\ell N) \right)\\
&=\left(\sum_{i< j}-\sum_{\beta}\sum_{i<j\in I(\alpha, \beta)}\right)\theta\ln\left(\frac{\sfx_i((\beta+1)\sfL)-\sfx_j((\beta+1)\sfL)}{\sfx_i(\beta \sfL)-\sfx_j(\beta\sfL)}\right)
+\OO\left(N\ln (\ell N) \right).
\end{split}\end{align}
To bound the summation over $i< j\in I(\al,\beta)$ on the righthand side of \eqref{e:sumbb}, we recall from \Cref{l:freeentropy}
\begin{align}\begin{split}\label{e:entropybound1}
&\phantom{{}={}}\frac{\theta^2}{N^2}\sum_{i< j\in I(\alpha, \beta)}\ln\left(\sfx_i(\beta\sfL)-\sfx_j(\beta\sfL)\right)\\
&=\frac{1}{2}\iint \ln|x-y| \rho(x;\bmx_{I(\al,\beta)}(\beta\ell))\rho(y;\bmx_{I(\al,\beta)}(\beta\ell))\rd x\rd y +\OO\left(\frac{\ln N}{\ell N}\right).
\end{split}\end{align}
There are two cases. If on $\fR(\al, \beta)$, $H^*$ has a linear approximation, then the height functions corresponding to $\rho(x;\bmx_{I(\al,\beta)}(\beta\ell))$ and $\rho(x;\bmx_{I(\al,\beta+1)}((\beta+1)\ell)$ differs by at most $\OO(\varepsilon)$. Thus \Cref{l:freeentropy} and \eqref{e:entropybound1} together imply
\begin{align}\label{e:badab}
\sum_{i< j\in I(\alpha, \beta)}\theta\ln\left(\frac{\sfx_i((\beta+1)\sfL)-\sfx_j((\beta+1)\sfL)}{\sfx_i(\beta \sfL)-\sfx_j(\beta\sfL)}\right)=\OO(\varepsilon (\ell N)^2\ln(1/\ell))=\OO(\varepsilon_0 (\ell N)^2),
\end{align}
if on $\fR(\al, \beta)$, $H^*$ does not have a linear approximation, then we simply bound 
\begin{align}\label{e:goodab}
\sum_{i< j\in I(\alpha, \beta)}\theta\ln\left(\frac{\sfx_i((\beta+1)\sfL)-\sfx_j((\beta+1)\sfL)}{\sfx_i(\beta \sfL)-\sfx_j(\beta\sfL)}\right)=\OO((\ell N)^2).
\end{align}
Recall from \Cref{l:Lipschitz}, the total number of such pair $(\al,\beta)$ such that on $\fR(\al, \beta)$, $H^*$ does not have a linear approximation, is at most $\OO(\varepsilon_0/\ell^2)$. By plugging \eqref{e:badab}, \eqref{e:goodab} into \eqref{e:sumbb}, and summing over $\al,\beta$, we get
\begin{align*}
\sum_{\beta} \sum_{\al<\al'}\sum_{\beta \sfL\leq \sft<(\beta+1)\sfL}
\sum_{i\in I(\alpha, \beta), j\in I(\alpha', \beta)}\ln\left(1+\frac{\theta(e_i(\sft)-e_j(\sft))}{\sfx_i(\sft)-\sfx_j(\sft)}\right)
=\sum_{i< j}\theta\ln\left(\frac{\sfx_i( \sfL )-\sfx_j( \sfL )}{\sfx_i(0)-\sfx_j(0)}\right)+\OO(\varepsilon_0 N^2).
\end{align*}
This finishes the proof of \Cref{l:offdiagonal}.
\end{proof}

\section{Large Deviation Lower Bound: Constant Slope Case}
\label{s:lowB}
We recall the parallelogram shaped region $\widetilde \fP$ from \eqref{e:deffP}:
\begin{align}\label{defp2}
\widetilde \fP=\{x,t\in \bR^2: 0\leq t\leq \ell, a(t)\leq x\leq b(t)\}.
\end{align}
where  $\{a(t), b(t)\}_{0\leq t\leq \ell}$ are two Bernoulli walk paths such that $a(t)<b(t)$ for all $0\leq t\leq \ell$ (more specifically
$a(t)= y_{i}(t+\beta\ell)-\al \ell, b(t)= y_{j+1}(t+\beta\ell)-\theta-\al\ell$. From the construction in \eqref{s:lowbb}, the region $\widetilde \fP$ satisfies the following assumption

\begin{assumption}\label{a:Hset}
Fix any  $(\varrho, -\varrho v)\in \cT_\zeta$. Denote the height function $\cA(x,t)$ with constant $(\varrho, -\varrho v)$ (see \Cref{d:constantslope}). Take large integer $n\geq 1$ such that
\begin{align*}
\left|\frac{n\theta}{N}-\varrho\right|\leq \varepsilon,
\end{align*}
and an $n$-particle Bernoulli random walk $\{\bmy(t)\}_{0\leq t\leq \ell}$ such that $\bmy(t)\in [a(t), b(t)-\theta]$ and its height function $\cH(x,t)$ satisfying 
\begin{align}
\label{e:cHcAdiff}
\|\cH-\cA\|_\infty\leq \varepsilon.
\end{align}

 We denote $\Adm(\bmy(0),\bmy(\ell), a,b)$ the set of $n$-particle nonintersecting Bernoulli random walks $\sfp=\{\bmx(t)\}_{0\leq t\leq \ell}$ such that
\begin{align*}
\bmx(0)=\bmy(0),\quad \bmx(\ell)=\bmy(\ell),\quad
a(t)\leq x_n(t)\leq\cdots\leq x_1\leq b(t)-\theta,\quad 0\leq t\leq \ell.
\end{align*}
\end{assumption}

\begin{proposition}\label{p:LDPlower}
Adopt \Cref{a:Hset}. We recall the rate function $\sigma$ from \eqref{sigmal} and  parameters $\zeta, \delta, \varepsilon, \ell$ from \eqref{e:introepsilon}. There exists an $n$-particle configuration $\bmx$, such that the Markov process \eqref{e:Mkcopy} starting from $\bmx$ with height function $\cH$ satisfies
\begin{align*}
\frac{1}{(\ell N)^2}\ln \bP(\{\cH:\|\cH-\cA\|_\infty\leq C\delta \ln (\ell/\delta)\} )\geq \frac{1}{\theta}\sigma(\varrho, -\varrho v)-\frac{n}{\ell N}\ln(2)+\OO\left(\frac{\varepsilon \ln(\ell/\delta) }{\delta\zeta^2} +\frac{\delta \ln(\ell/\delta)^2}{\ell}\right),
\end{align*}
provided $N$ is large enough.
\end{proposition}
The nonintersecting Bernoulli walk ensembles starting from $\bmx$ in \Cref{p:LDPlower} may not belong to $\Adm(\bmy(0),\bmy(\ell), a,b)$, as needed to prove  \Cref{p:LDPlow}. The following proposition states that we can slightly modify them such that they belong to $\Adm(\bmy(0),\bmy(\ell), a,b)$.

\begin{proposition}\label{p:changeboundary}
Adopt \Cref{a:Hset}. We recall the parameters $\zeta, \delta,\xi, \ell$ from \eqref{e:introepsilon}. Given a height function $\cH$ with $\|\cH- \cA\|_\infty\leq C\delta \ln (\ell/\delta)=:\varepsilon'$ corresponding to the nonintersecting Bernoulli walk $\sfp=\{\bmx(t)\}_{0\leq t\leq \ell}$, and initial configuration $\bmx$ as in \Cref{p:LDPlower}.  There exists a modified nonintersecting Bernoulli walk $\widehat\sfp=\{\widehat\bmx(t)\}_{0\leq t\leq \ell}\in \Adm(\bmy(0),\bmy(\ell), a,b)$ such that 
its height function $\widehat \cH$ satisfies
\begin{align}\label{e:heightdiffhatH}
\|\widehat \cH-\cA\|_\infty\leq C' \xi.
\end{align}
Moreover,
the map from $\sfp$ to $\widehat \sfp$ is at most $e^{\OO(\xi \ell N^2)}$ to one, and
\begin{align*}
\ln\bP(\sfp)=\ln\bP(\widehat\sfp)+\OO(\xi \ell\ln (\ell/\xi)N^2).
\end{align*}
\end{proposition}

\begin{proof}[Proof of \Cref{p:LDPlow}]
The first statement \eqref{e:triviallow} in \Cref{p:LDPlow} follows \Cref{l:VVbb}.
For the second statement \eqref{e:linearaplow},we notice that the sum is taken over configurations satisfying \eqref{e:heightclose2}, namely which are restricted to live in $\widetilde \fP$ defined in
\eqref{defp2}, or in other words which belong to  $ \Adm(\bmy(0),\bmy(\ell), a,b)$. 
 Thanks to \Cref{p:changeboundary}, we have
\begin{align*}
&\phantom{{}={}}\frac{1}{(\ell N)^2}\ln \bP(\widehat \cH\in \Adm(\bmy(0),\bmy(\ell), a,b): \|\widehat \cH-\cA\|_\infty\leq C'\xi)=\frac{1}{(\ell N)^2}\sum_{\widehat\sfp: \widehat \cH\in \Adm(\bmy(0),\bmy(\ell), a,b)\atop \|\widehat \cH-\cA\|_\infty\leq C'\xi}\ln \bP(\widehat\sfp)\\
&\geq \frac{1}{(\ell N)^2}\sum_{\sfp:  \| \cH-\cA\|_\infty\leq  \varepsilon'}\ln \bP(\sfp)+\OO((\xi/ \ell)\ln(\ell/\xi))=\frac{1}{(\ell N)^2}\ln \bP(\| \cH-\cA\|_\infty\leq\varepsilon')+\OO((\xi/ \ell)\ln(\ell/\xi))\\
&\geq\frac{1}{\theta}\sigma(\varrho, -\varrho v)-\frac{n}{\ell N}\ln(2)+\OO\left(\frac{\varepsilon \ln(\ell/\delta) }{\delta\zeta^2} +\frac{\delta \ln(\ell/\delta)^2}{\ell}+\frac{\xi\ln(\ell/\xi)}{ \ell}\right)\\
&=\frac{1}{\theta}\iint_{[\al\ell,  (\al+1)\ell]\times [\beta\ell, (\beta+1)\ell]}\sigma(\nabla H^*)\rd x\rd t-\frac{n}{\ell N}\ln(2)+\OO\left(\varepsilon_0+\frac{\varepsilon \ln(\ell/\delta) }{\delta\zeta^2} +\frac{\delta \ln(\ell/\delta)^2}{\ell}+\frac{\xi\ln(\ell/\xi)}{ \ell}\right),
\end{align*}
where in the third line we used \Cref{p:LDPlower}
and in the last equality we used the second statement in \Cref{l:Lipschitz}.

\end{proof}

\subsection{Proof of \Cref{p:changeboundary}}

We recall from \eqref{e:introepsilon} that $\varepsilon'=\delta \ln(\ell/\delta)\ll \xi\zeta$, and introduce the index sets 
\begin{align}\label{e:defI012b}
I_0=\qq{1, \xi N}, \quad I_1=\qq{\xi N, n-\xi N},\quad I_2=\qq{n-\xi N,n}.
\end{align}
We set
\begin{align}\label{e:constructhatp}
\widehat x_i(t)=x_i(t),\quad \xi \leq t\leq \ell-\xi, \quad \xi N\leq i\leq n-\xi N.
\end{align}
To construct $\widehat\sfp$ given \eqref{e:constructhatp}, we need to construct the Bernoulli walk paths $\{\widehat\sfx_i(\sft)\}_{i\in I_0}, \{\widehat\sfx_i(\sft)\}_{i\in I_2}$, and also $\{\widehat\sfx_i(\sft)\}_{i\in I_1, \sft\leq \xi N}, \{\widehat\sfx_i(\sft)\}_{i\in I_1, \sft\geq (\ell-\xi) N}$.

For any $1\leq i\leq \varrho \ell N/\theta$, we denote the $i$-th level line of $\cA$ as
\begin{align}\label{e:defgamma}
\gamma_i(t)=\inf \{x: \cA(x, t)>\theta(i-1)/N\}=\gamma_i(0)+tv=\frac{\theta(i-1)}{\varrho N}+tv,
\end{align}
which are straight lines. 

By our assumption $\|\cH-\cA\|_\infty\leq \varepsilon'$, we have for $\lceil\varepsilon' N/\theta\rceil<i<n-\lceil\varepsilon' N/\theta\rceil$,
\begin{align}\label{e:Hbound}
\cA(\gamma_{i+\lceil\varepsilon' N/\theta\rceil}(t),t)\leq \frac{\theta(i-1)}{N}-\varepsilon'  \leq \cA(x_i(t),t)\leq \frac{\theta(i-1)}{N}+\varepsilon'  \leq \cA(\gamma_{i-\lceil\varepsilon' N/\theta\rceil}(t),t).
\end{align}
It follows from \eqref{e:defgamma} and \eqref{e:Hbound}, and noticing $\xi\geq 2\varepsilon'/\theta$, that for $i\in I_1$.
\begin{align}\label{e:xloc}
\gamma_{i-\lceil\varepsilon' N/\theta\rceil}(t)
\leq x_i(t)\leq \gamma_{i-\lceil\varepsilon' N/\theta\rceil}(t),\quad \Rightarrow \quad
|x_i(t)-\gamma_i(t)|\leq 2\varepsilon'/\varrho.
\end{align}

By the same argument, our assumption \eqref{e:cHcAdiff} also implies that for any $i\in I_1$,
\begin{align}\label{e:xloc2}
|y_i(t)-\gamma_i(t)|\leq 2\varepsilon'/\varrho,\quad 0\leq t\leq \ell.
\end{align}
Moreover, since $\bmy(t)\in [a(t), b(t)-\theta]$, \eqref{e:cHcAdiff} also implies that either $a(t)\leq tv$, or $a(t)\geq tv$ and $\cA(a(t),t)\leq \cH(a(t),t)+\varepsilon'=\varepsilon'$. In both cases we have
\begin{align}\label{e:gammaiat}
\gamma_i(t)-a(t)\geq \frac{1}{\varrho}\left(\frac{(i-1)\theta}{N}-\varepsilon'\right).
\end{align}

\begin{figure}
	\begin{center}
	 \includegraphics[scale=0.8]{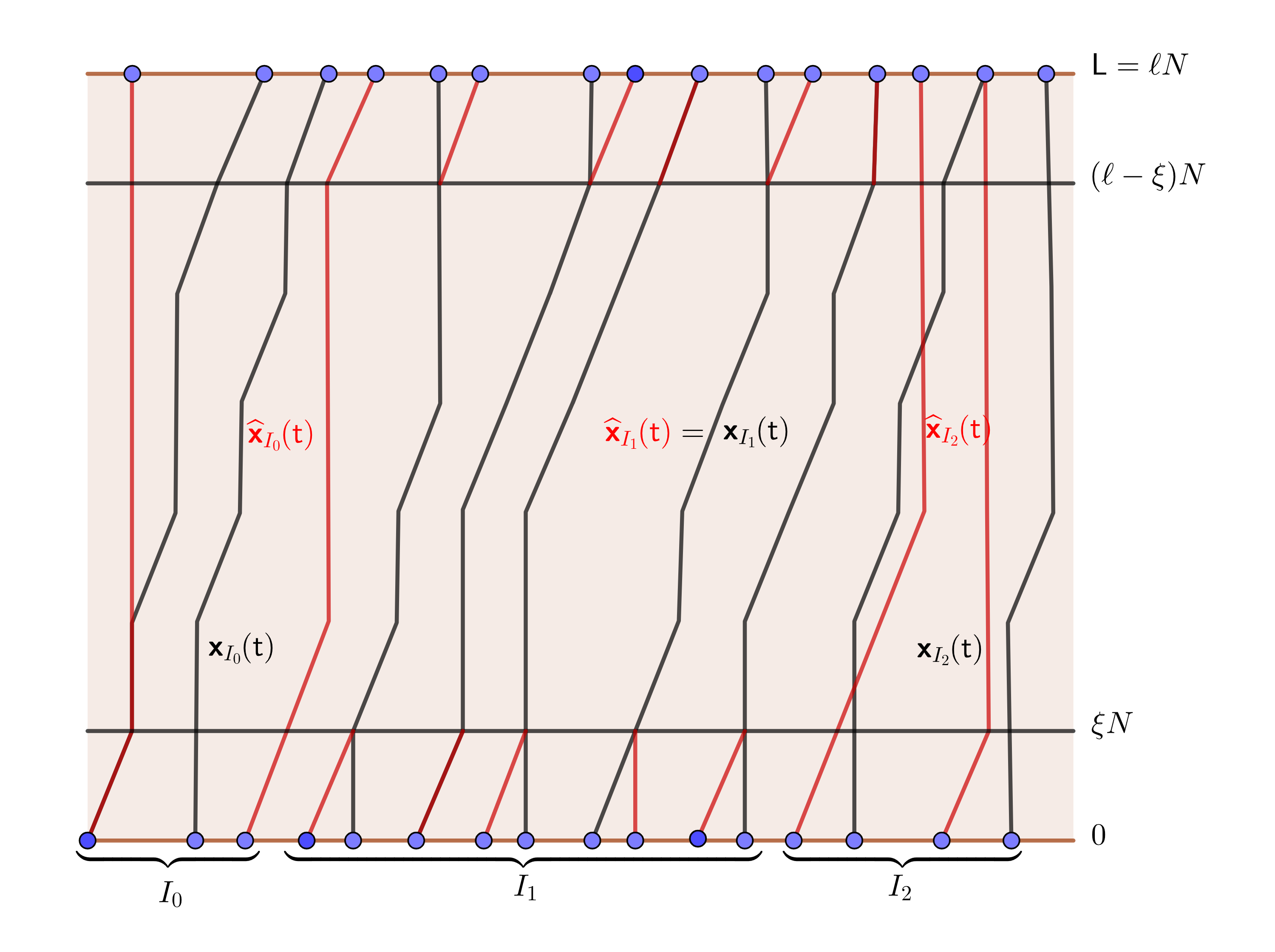}
	 \caption{Shown above is construction of the Bernoulli walk paths $\widehat{\bm\sfx}(t)$. }
	 \label{f:path}
	 \end{center}
	 \end{figure}

We will construct the Bernoulli walk paths $\{{\widehat \sfx}_i(\sft)\}_{i\in I_0},\{{\widehat \sfx}_i(\sft)\}_{i\in I_1, \sft\leq \xi N}$; and  $\{{\widehat \sfx}_i(\sft)\}_{i\in I_2},\{{\widehat \sfx}_i(\sft)\}_{i\in I_1, \sft\geq (\ell-\xi) N}$ in the following way; see \Cref{f:path}.
For $1\leq i\leq \xi N, 0\leq \sft\leq \sfL$, 
\begin{align}\begin{split}\label{e:extremes}
&{\widehat \sfx}_i(\sft)=\max\{Na(\sft)+(i-1)\theta, \sfy_i(0), \sfy_i(\sfL)-(\sfL-\sft)\}\\
&{\widehat \sfx}_{n-i+1}(\sft)=\min\{Nb(\sft)-i\theta, \sfy_{n-i+1}(0)+\sft, \sfy_{n-i+1}(\sfL)\}.
\end{split}\end{align}
And for $i \in I_1, 0\leq \sft\leq \xi N$, 
\begin{align}\begin{split}\label{e:extremet}
&{\widehat \sfx}_i(\sft)=\max\{\sfx:\sfx- \sfy_i(0)\in\bZ,\sfx\leq (1-t/\xi)\sfy_i(0) +(t/\xi)\sfx_i(\xi N)\},\\
&{\widehat \sfx}_i(\sfL-\sft)=\max\{ \sfx: \sfy_i(\sfL)-\sfx\in \bZ, \sfx\leq (1-t/\xi)\sfx_i(\sfL-\xi N)+(t/\xi)\sfy_i(\sfL)\}.
\end{split}\end{align}

The maximum or minimum over Bernoulli paths is still a Bernoulli path. So $\{{\widehat \sfx}_i(\sft)\}_{0\leq \sft\leq\sfL}$ is a Bernoulli path. The map from $\sfp$ to $\widehat \sfp$ is at most $e^{\OO(\xi \ell N^2)}$ to one. 

To show that $\{{\widehat \sfx}_i(\sft)\}_{1\leq i\leq n, 0\leq \sft\leq \sfL}\in \Adm(\bmy(0),\bmy(\ell), a,b)$, we need to check the following conditions
\begin{enumerate}
\item \label{e:it1} $\{{\widehat \sfx}_i(\sft)\}_{1\leq i\leq \xi N, 0\leq \sft\leq \sfL}$ are non-interesecting Bernoulli paths form $\sfy_i(0)$ to $\sfy_i(\sfL)$; $\{{\widehat \sfx}_{n-i+1}(\sft)\}_{1\leq i\leq \xi N, 0\leq \sft\leq \sfL}$ are non-intersecting Bernoulli paths form $\sfx_{n-i+1}(0)$ to $\sfy_{n-i+1}(\sfL)$.
\item \label{e:it2} $\{{\widehat \sfx}_i(\sft)\}_{i\in I_1, 0\leq \sft\leq \xi N}$ are non-interesecting Bernoulli paths form $\sfy_i(0)$ to $\sfx_i(\xi N)$; $\{{\widehat \sfx}_{i}(\sfL-\sft)\}_{i\in I_1, 0\leq \sft\leq \sfL}$ are non-intersecting Bernoulli paths form $\sfx_{i}(\sfL-\sft)$ to $\sfy_{i}(\sfL)$.
\item \label{e:it3}${\widehat \sfx}_{\xi N}(\sft)<{\widehat \sfx}_{\xi N+1}(\sft)$ and ${\widehat \sfx}_{n-\xi N}(\sft)<{\widehat \sfx}_{n-\xi N+1}(\sft)$.
\end{enumerate} 

From \Cref{a:Hset}, $\Adm(\bmy(0),\bmy(\ell),a,b)\neq\emptyset$, so $\sfy_i(0)\geq Na(0)+(i-1)\theta$, $\sfy_{n-i+1}(\sfL)\leq Nb(\sfL)-i\theta$ and ${\sfy}_i(0)\geq {\sfy}_i(\sfL)-\sfL$. Thus at time $\sft=0$, ${\widehat \sfx}_i(0)=\sfy_i(0)$; at time $\sft=\sfL$, ${\widehat \sfx}_i(\sfL)=\sfy_i(\sfL)$. In particular ${\widehat \sfx}_i(\sft)$ is a Bernoulli path from $\sfy_i(0)$ to $\sfy_i(\sfL)$. Also from the construction \eqref{e:extremet}, we have ${\widehat \sfx}_i(\sfL)< {\widehat \sfx}_j(\sfL)$ for $1\leq i<j\leq \xi N$. The same argument also implies that $\{{\widehat \sfx}_{n-i+1}(\sft)\}_{1\leq i\leq \xi N, 0\leq \sft\leq \sfL}$ are non-intersecting Bernoulli paths form $\sfy_{n-i+1}(0)$ to $\sfy_{n-i+1}(\sfL)$. This proves \Cref{e:it1}.

For \Cref{e:it2}, by our construction \eqref{e:extremet}, we have that ${\widehat \sfx}_i(0)=\sfy_i(0)$ and ${\widehat \sfx}_i(\xi N)=\sfx_i(\xi N)$. It follows from \eqref{e:xloc} and \eqref{e:xloc2}, that $\sfx_i(\xi N)-\sfy_i(0)=N(\gamma_i(\xi N)-\gamma_i(0))+\OO(\varepsilon' N/\varrho)=N\xi v+\OO(\varepsilon' N/\varrho)$, and 
\begin{align*}
\left(1-\frac{\sft+1}{N\xi}\right)\sfy_i(0) +\left(\frac{\sft+1}{N\xi}\right)\sfx_i(\xi N)-\left(1-\frac{\sft}{N\xi}\right)\sfy_i(0) -\left(\frac{\sft}{N\xi}\right)\sfx_i(\xi N)=v+\OO(\varepsilon'/(\xi\varrho))\in(0,1),
\end{align*}
provided $\varrho v\geq \zeta\gg \varepsilon'/\xi$.

Thus the construction \eqref{e:extremet} implies that ${\widehat \sfx}_i(\sft+1)-{\widehat \sfx}_i(\sft)\in\{0,1\}$. We conclude that $\{{\widehat \sfx}_i(\sft)\}_{i\in I_1, 0\leq \sft\leq \xi N}$ are non-interesecting Bernoulli paths form $\sfy_i(0)$ to $\sfx_i(\xi N)$. The same argument also implies that $\{{\widehat \sfx}_{i}(\sfL-\sft)\}_{i\in I_1, 0\leq \sft\leq \sfL}$ are non-intersecting Bernoulli paths form $\sfx_{i}(\sfL-\sft)$ to $\sfy_{i}(\sfL)$.

For \Cref{e:it3}, by symmetry, we only need to show that 
\begin{align}\label{e:xibound}
{\widehat \sfx}_i(\sft)<{\widehat \sfx}_{\xi N+1}(\sft), \quad i\in I_0,\quad 0\leq \sft\leq \sfL-\xi N.
\end{align}
We will check the three relations: $\widehat\sfx_{\xi N+1}(\sft)\geq Na(\sft)+(\xi N+1)\theta$, $\widehat\sfx_{\xi N+1}(\sft)\geq \sfy_i(0)$ and $\widehat\sfx_{\xi N+1}(\sft)\geq \sfy_i(\sfL)-(\sfL-\sft)$. For $\xi N\leq \sft\leq \sfL-\xi N$, the first relation follows from \eqref{e:xloc} and \eqref{e:gammaiat}, $\sfx_{\xi N+1}(\sft)=N\gamma_{\xi N+1}(t)+\OO(\varepsilon' N/\varrho)\geq Na(t)+\theta (\xi N)/\varrho+\OO(\varepsilon' N/\varrho)>Na(\sft)+\xi \theta N$. The second statement follows from $\sfx_{\xi N+1}(\sft)=N\gamma_{\xi N+1}(t)+\OO(\varepsilon' N/\varrho)=N\gamma_{\xi N+1}(0)+Ntv+\OO(\varepsilon' N/\varrho)=\sfy_i(0)+Ntv+\OO(\varepsilon' N/\varrho)>\sfy_i(0)$, where we used that $\gamma_{\xi N}(t)$ is linear in $t$, \eqref{e:xloc2} and $t\geq \xi$. The last statement follows from similar argument $\sfx_{\xi N+1}(\sft)=N\gamma_{\xi N+1}(t)+\OO(\varepsilon' N/\varrho)=N\gamma_{\xi N+1}(\ell)+Nv(\ell-t)+\OO(\varepsilon' N/\varrho)=\sfy_i(\sfL)+Nv(\ell-t)+\OO(\varepsilon' N/\varrho)>\sfy_i(\sfL)$, where we used that $\gamma_{\xi N}(t)$ is linear in $t$, \eqref{e:xloc2} and  $t\leq \ell-\xi$. This finishes the proof of \eqref{e:xibound} for $\xi N\leq \sft\leq \sfL-\xi N$. 

For $0\leq \sft\leq \xi N$, using \eqref{e:xloc} the first statment follows from: 
\begin{align*}
{\widehat \sfx}_{\xi N+1}(\sft)
&\geq  (1-t/\xi)\sfy_{\xi N+1}(0) +(t/\xi)\sfx_{\xi N+1}(\xi N)-1
=N\gamma_{\xi N+1}(t)-\OO(\varepsilon' N/\varrho)\\
&=Na(t)+\xi\theta N/\varrho-\OO(\varepsilon' N/\varrho)
>Na(t)+\xi \theta N,
\end{align*}
where we used \eqref{e:xloc} and \eqref{e:xloc2}.
The second statement that $ {\widehat \sfx}_{\xi N+1}(\sft)\geq \sfy_i(0)$ holds trivially; and  the third statement follows from \eqref{e:xloc2} that $ {\widehat \sfx}_{\xi N+1}(\sft)
\geq N\gamma_{\xi N+1}(t)-\OO(\varepsilon' N/\varrho)=N\gamma_{\xi N+1}(\sfL)-(\sfL-\sft)v-\OO(\varepsilon' N/\varrho)\geq \sfy_i(\sfL)-(\sfL-\sft)$.
This finishes the proof of \eqref{e:xibound} for $0\leq \sft\leq \xi N$. And we conclude that $\widehat\sfp$ constructed above belongs to $\Adm(\bmy(0),\bmy(\ell), a,b)$.

Next we show for any $i\in I_1$, we have
\begin{align}\label{e:levelinebound}
|\widehat x_i(t)-\gamma_i(t)|\lesssim \varepsilon'/\zeta\ll \xi.
\end{align}
and the claim \eqref{e:heightdiffhatH} follows.

For $\xi\leq t\leq \ell-\xi$, \eqref{e:levelinebound} follows from \eqref{e:xloc}. For $0\leq t\leq \xi$, from the construction \eqref{e:extremet} and \eqref{e:xloc} and \eqref{e:xloc2}
\begin{align*}
|\widehat x_i(t)-\gamma_i(t)|
&\leq |(1-t/\xi)y_i(0)+(t/\xi)x_i(\xi N)-\gamma_i(t)|+1/N\\
&\leq |(1-t/\xi)\gamma_i(0)+(t/\xi)\gamma_i(\xi N)-\gamma_i(t)|+(\varepsilon'+1/N)
=\varepsilon'+1/N.
\end{align*}
The case that $\ell-\xi\leq t\leq \ell$ follows from the same argument.

We can decompose the weight $\bP(\sfp)$ 
\begin{align}\begin{split}\label{e:Wpdecompose}
\ln\bP(\sfp)
&=\sum_{\sft\leq \xi N \text{ or } \sft\geq (\ell-\xi) N}\ln\frac{V(\bm\sfx(\sft)+\theta\bme(\sft))}{V(\bm\sfx(\sft))}+\sum_{ \xi N \leq \sft\leq (\ell-\xi) N}
\left(\ln\frac{V(\bm\sfx_{I_{0}\cup I_2}(\sft)+\bme_{I_{0}\cup I_2}(\sft))}{V(\bm\sfx_{I_{0}\cup I_2}(\sft))}\right.\\
    &\left.+\ln\frac{V(\bm\sfx_{I_{1}}(\sft)+\bme_{I_1}(\sft))}{V(\bm\sfx_{I_1}(\sft))}
    +\sum_{i\in I_0\cup I_2, j\in I_1}
    \ln\left(1+\frac{\theta(e_i(\sft)-e_j(\sft))}{(\sfx_i(\sft)-\sfx_j(\sft))}\right)\right)-  \ell N n.
\end{split}\end{align}
Thanks to \Cref{l:VVbb}, we can bound terms in \eqref{e:Wpdecompose} as
\begin{align}\begin{split}\label{e:Vterm}
&\sum_{\sft\leq \xi N \text{ or } \sft\geq (\ell-\xi) N}\ln\frac{V(\bm\sfx(\sft)+\theta\bme(\sft))}{V(\bm\sfx(\sft))}=\OO(\xi n N)=\OO(\xi \ell N^2),\\
&\sum_{ \xi N \leq \sft\leq (\ell-\xi) N}
\ln\frac{V(\bm\sfx_{I_{0}\cup I_2}(\sft)+\bme_{I_{0}\cup I_2}(\sft))}{V(\bm\sfx_{I_{0}\cup I_2}(\sft))}=\OO((|I_1|+|I_2|)\ell N)=\OO(\xi \ell N^2).
\end{split}\end{align}
And for the last term in \eqref{e:Wpdecompose} 
\begin{align}\begin{split}\label{e:IJdiff}
    &\phantom{{}={}}\sum_{i\in I_0\cup I_2, j\in I_1} \ln\left(1+\frac{\theta(e_i(\sft)-e_j(\sft))}{(\sfx_i(\sft)-\sfx_j(\sft))} \right)
    \leq 2\theta \sum_{i\in I_0\cup I_2, j\in I_1} \frac{1}{|\sfx_i(\sft)-\sfx_j(\sft)|}
    \leq 2\sum_{i\in I_0\cup I_2, j\in I_1} \frac{1}{|i-j|},
\end{split}\end{align}
where in the last inequality we used $|\sfx_{i}(\sft)-\sfx_{j}(\sft)|\geq \theta|i-j|$. Recall the sets $I_0, I_1, I_2$ from \eqref{e:defI012}, we can further bound \eqref{e:IJdiff} as
\begin{align}\label{e:sumb}
    \sum_{i\in I_0\cup I_2, j\in I_1} \frac{1}{|i-j|}
    \leq \sum_{i=1}^{\xi N}\sum_{j=i}^{\ell N}\frac{1}{j}
    \leq \xi N+\xi N\sum_{\xi N\leq j\leq \ell N}\frac{1}{j}
    \leq \xi N +\xi\ln (\ell/\xi) N =\OO(\xi \ln (\ell/\xi)N).
\end{align}
Plugging \eqref{e:Vterm}, \eqref{e:IJdiff} and \eqref{e:sumb} into \eqref{e:Wpdecompose}, we get
\begin{align}\begin{split}\label{e:Wpdecompose2}
\ln\bP(\sfp)
&=\sum_{ \xi N \leq \sft\leq (\ell-\xi) N}
    \ln\frac{V(\bm\sfx_{I_{1}}(\sft)+\bme_{I_1}(\sft))}{V(\bm\sfx_{I_1}(\sft))}-  \ell N n+\OO(\xi \ell\ln (\ell/\xi)N^2).
\end{split}\end{align}
The same statement \eqref{e:Wpdecompose2} holds for $\widehat \sfp$. Moreover, from our construction \eqref{e:constructhatp}, we conclude from \eqref{e:Wpdecompose2}
\begin{align*}
\ln\bP(\sfp)=\ln\bP(\widehat\sfp)+\OO(\xi \ell\ln (\ell/\xi)N^2).
\end{align*}

\subsection{Proof of \Cref{p:LDPlower}}
 We recall the smoothed height function $\wt H$ from \eqref{e:deftH}, the associated complex slope $\wt f_t$ from \eqref{fh}, and the drift $g_t(z)$ from \eqref{e:gtx}. 
We consider Markov process $\{\bmx(t)\}_{0\leq t\leq \ell}$ with height function $\cH(x,t)$, such that the initial data $\cH(x,0)$ is close to $\widetilde H(x,0)$. From the construction of $\widetilde H$ in \eqref{e:deftH}, thanks to \Cref{c:nabHbound}, we have
\begin{align*}
|\cA(x,t)-\widetilde H(x,t)|\lesssim\delta+\int_{-3\ell}^{x}\frac{\delta \varrho \rd y}{\delta+\dist(y, \{tv, \ell+tv\})}\lesssim \delta \ln (\ell/\delta) ,
\end{align*}
and it follows there exists some large $C>1$
\begin{align}\label{e:tHcenter}
\frac{1}{(\ell N)^2}\ln \bP(\|\cH(x,t)-\cA(x,t)\|_\infty\leq C\delta \ln (\ell/\delta))
\geq \frac{1}{(\ell N)^2}\ln \bP(\|\cH(x,t)-\wt H(x,t)\|_\infty\leq \varepsilon).
\end{align}
To lower bound \eqref{e:tHcenter}, by the same argument as in proof of the large deviation upper bound, we tilt the Markov chain by the exponential Martingale \eqref{e:Mt}. The following lemma collects some estimates for the numerator and denominator of the exponential Martingale on the event $\|\cH(x,t)-\wt H(x,t)\|_\infty\leq \varepsilon$,

\begin{lemma}\label{l:Nterm2}
Adopt \Cref{a:Hset}. 
Let $\cH(x,t)$ be the height function associated with the particle configuration $\{\bmx(t)\}_{0\leq t\leq \ell}$, with $\supp(\bmx(t))\in [-3\ell, 3\ell]$ and $\|\cH(x,t)-\wt H(x,t)\|_\infty\leq \varepsilon$ then 
\begin{align}\label{e:exp}
&\phantom{{}={}}\sum_{\sft=0}^{\sfL-1}{\sum_{i=1}^n e_i(\sft) g_t(x_i(\sft/N))}
=-\frac{N^2}{\theta}\int_0^{\ell} \int_\bR \del_t \cA(x, t)g_t(x)\rd x\rd t +\OO\left(\left(\frac{\varepsilon}{\delta} +\frac{\delta \ln(\ell/\delta)^2}{\ell} \right) (\ell N)^2
\right).
\end{align}
Dynamical loop equation gives
\begin{align}\label{e:DLE}
    \frac{1}{N}\ln   \bE[e^{\sum_i    e_i(\sft) g_t(   x_i(t)/N)}|  \bmx(t)]
   =\frac{1}{2\pi\ri\theta}\int_0^1\oint_\omega\ln (1+e^{   m_t^*(z)+\tau  g_t (z)}) g_t (z)\rd z\rd \tau+\OO\left(\frac{\varepsilon \ell \ln(\ell/\delta)}{\delta\zeta^2}+\delta \ln(\ell/\delta)^2\right),
\end{align}
where the contour $\omega\in \Lambda$ (from \eqref{e:defLambda}) encloses $[-3\ell,3\ell]$, and $m_t^*(z)$ is the Stieltjes transform of the empirical measure of $\del_x\cA(x,t)$
\begin{align*}
m_t^*(z)=\int_{tv}^{tv+\ell} \frac{\del_x \cA(x,t)\rd x}{x-z}.
\end{align*}
\end{lemma}

With \Cref{l:Nterm2}, we can lower bound the righthand side of \eqref{e:tHcenter} as
\begin{align}\begin{split}\label{e:tilt}
&\phantom{{}={}}\frac{1}{(\ell N)^2}\ln \bP(\|\cH(x,t)-\wt H(x,t)\|_\infty\leq \varepsilon)\geq -\frac{S(\cA;g)}{\theta \ell^2}+\OO\left(\frac{\varepsilon}{\delta} +\frac{\delta \ln(\ell/\delta)^2}{\ell} +\frac{\varepsilon \ell \ln(\ell/\delta)}{\delta\zeta^2}\right)\\
&+  \frac{1}{(\ell N)^2}\ln\bE\left[\prod_{t\in\qq{0, \sfL-1}/N} \frac{e^{\sum_{i=1}^N e_i(Nt) g_t(x_i(t))}}{\bE[e^{\sum_{i=1}^N e_i(Nt) g_t(x_i(t))}|\bmx(t)]}\bm1(\|\cH(x,t)-\wt H(x,t)\|_\infty\leq \varepsilon)\right],
\end{split}\end{align}
where $ S(\cA;g)$ is from \eqref{e:rateS} and  \Cref{p:entropy} gives
\begin{align}\begin{split}\label{e:Sexp}
    S(\cA;g)&=
     -\int_0^{\ell} \int_\bR \del_t \cA(x,t) g_t(x)\rd x \rd t-\frac{1}{2\pi\ri}\int_0^{\ell}\oint\int_0^1\ln (1+e^{m^*_t(z)+\tau g_t(z)})g_t(z)\rd z\rd \tau\rd t\\
     &=-\ell^2(\sigma(\varrho, -\varrho v)-\varrho \ln(2))+\OO( \delta \ln(\ell/\delta)^2\ell ).
\end{split}\end{align}

To lower bound the second term on the righthand side of \eqref{e:tilt}, we introduce the tilded measure
 $\bP^g(\cdot)$ as defined below,
\begin{align}\label{e:Pg}
\bP^g(\cdot)=\bE\left[\prod_{t\in\qq{0, \sfL-1}/N} \frac{e^{\sum_{i=1}^N e_i(Nt) g_t(x_i(t))}}{\bE[e^{\sum_{i=1}^N e_i(Nt) g_t(x_i(t))}|\bmx(t)]}(\cdot )\right]. 
\end{align}
Under $\bP^g(\cdot)$,
 the Markov process \eqref{e:Mkcopy} becomes
\begin{align}\label{e:Mk}
    \bP^g(\bm\sfx(\sft+1)=\bm\sfx+\bme|\bm\sfx(\sft)=\bm\sfx)
    \propto 
    \frac{V(\bm\sfx+\theta \bme)}{V(\bm\sfx)} \prod_{i=1}^N e^{e_i g_t(\sfx_i/N)},\quad t=\sft/N.
\end{align}

The large deviation lower bound follows from the following limit shape result 
\begin{proposition}\label{p:limitshape}
Adopt \Cref{a:Hset}. We recall the parameter $\varepsilon$ from \eqref{e:introepsilon}. 
For any $\varepsilon>0$, there exists a small $c(\varepsilon)>0$, such that if the height function $\cH(x,0)$ of the initial data $\bmx(0)$ satisfies $\|\cH(x,0)-\wt H(x,0)\|_\infty\leq c(\varepsilon)$, then  
\begin{align}\label{e:Pg2b}
     \bP^g(\|\cH(x,t)-\wt H(x,t)\|_\infty\leq \varepsilon)\geq 1-\varepsilon,
\end{align}
provided $N$ is large enough.
\end{proposition}

\begin{proof}[Proof of \Cref{p:LDPlower}]
Take $\bmx$ with height function $\cH(x,0)$ as in \Cref{p:limitshape}. Proof of \Cref{p:LDPlower} follows from plugging \eqref{e:Sexp} and \Cref{p:limitshape} into \eqref{e:tilt}.
\end{proof}

\subsection{Proof of \Cref{l:Nterm2}}
\begin{proof}
From the construction of $\widetilde H$ in \eqref{e:deftH}, and thanks to \eqref{e:ktxbound} from \Cref{c:nabHbound}
\begin{align}\label{e:Htboundhi}
\widetilde H(-3\ell, t)=\int_\infty^{-3\ell}\kappa_t(y)\rd y\lesssim \delta,\quad \theta-\widetilde H(3\ell, t)\lesssim \delta,
\end{align}
it follows that 
\begin{align*}
H(-3\ell, t)\lesssim \delta+\varepsilon\lesssim \delta,\quad \theta-\widetilde H(3\ell, t)\lesssim \delta+\varepsilon\lesssim \delta.
\end{align*}
Using the above estimate we can rewrite the lefthand side of \eqref{e:exp} as
\begin{align}\begin{split}\label{e:replaceH2}
\sum_{\sft=0}^{\sfL-1}{\sum_{i=1}^N e_i(\sft) g_t(x_i(t))}=
\sum_{\sft=0}^{\sfL-1}{\sum_{i: |x_i(t)|\leq 3\ell} e_i(\sft) g_t(x_i(t))}
 +\OO(\delta\ell N^2  \|g_t\|_\infty).
\end{split}\end{align}
By the same argument as in \eqref{e:replaceH}, \eqref{e:replaceH2} leads to
\begin{align}\begin{split}\label{e:hi0}
\sum_{\sft=0}^{\sfL-1}{\sum_{i=1}^N e_i(\sft) g_t(x_i(\sft)/N)}=-\frac{N^2}{\theta}\int_0^{\ell} \int_{-3\ell}^{3\ell} \del_t \cH(x, t)g_t(x)\rd x\rd t +\OO(\delta\ell N^2 \|g_t\|_\infty).
\end{split}\end{align}

Next we replace $\cH(x,t)$ by $\wt H(x,t)$. Using the second and third statement of \Cref{l:gtestimate}, and integrating by part
\begin{align}\begin{split}\label{e:hi1}
&\phantom{{}={}}\left|\frac{N^2}{\theta}\int_0^{\ell} \int_{-3\ell}^{3\ell} (\del_t \cH(x, t)-\del_t \wt H(x, t))g_t(x)\rd x\rd t\right|\\
&\leq \frac{N^2}{\theta}\left|\left. \int_{-3\ell}^{3\ell} (\cH(x,t)-\wt H(x,t))g_t(z)\rd x\right|_{t=0}^{t=\ell} \right|
+\frac{N^2}{\theta}\left|\int_0^{\ell} \int_{-3\ell}^{3\ell} (\cH(x, t)- \wt H(x, t))\del_t g_t(x)\rd x\rd t\right|\\
&\lesssim \varepsilon \|g\|_\infty \ell N^2+\varepsilon \|\del_t g\|_\infty (\ell N)^2\lesssim \varepsilon \log(\ell/\delta) \ell N^2+(\varepsilon/\delta)(\ell N)^2\lesssim (\varepsilon/\delta)(\ell N)^2.
\end{split}\end{align}
Finally, we replace $\widetilde H(x,t)$ by $\cA(x,t)$
\begin{align}\begin{split}\label{e:hi2}
&\phantom{{}={}}\left|\frac{N^2}{\theta}\int_0^{\ell} \int_{-3\ell}^{3\ell} (\del_t \cA(x, t)-\del_t \wt H(x, t))g_t(x)\rd x\rd t\right|\\
&\leq \frac{N^2\|g\|_\infty}{\theta}\int_0^\ell \int_{-3\ell}^{3\ell} |\del_t \cA(x, t)-\del_t \wt H(x, t)|\rd x\rd t\\
&\lesssim \frac{N^2\|g\|_\infty}{\theta}\int_0^\ell \int_{-3\ell}^{3\ell} \frac{\delta}{\delta+\dist(x, \{tv, \ell+tv\})+\dist(x, \{tv, \ell+tv\})^2/\ell}\rd x\rd t\\
&\lesssim \delta \ln(\ell/\delta)\ell \|g\|_\infty N^2,
\end{split}\end{align}
where we used \Cref{c:nabHbound} for the second inequality. The claim \eqref{e:exp} follows from combining \eqref{e:replaceH2}, \eqref{e:hi0}, \eqref{e:hi1} and \eqref{e:hi2}.

The same as in \Cref{l:Nterm2}, by using the dynamical loop equation, we have
\begin{align}\begin{split}\label{e:DLE2}
    \ln   \bE[e^{\sum_i    e_i(\sft) g_t(   x_i(t)/N)}|  \bmx(t)]
    &=\frac{N}{2\pi\ri\theta}\int_0^1\oint_\omega\ln (1+e^{   m_t(z)+\tau  g_t (z)}) g_t (z)\rd z\rd \tau+\OO\left(1\right),
\end{split}\end{align}
where $   m_t(z)$ is the Stieltjes transform of the empirical measure of $   \bmx(t)$, and the implicit constant in the error term $\OO(1)$ depends on $\delta,\zeta$. We can deform the contour $\omega$ such that it consists of two parallel lines $\{z: \Im[z]=\pm 1/K\}$, with $K\asymp 1/\delta\zeta$.

Recall from \Cref{l:gtestimate} and use \eqref{e:largezbound}, if $\dist(z, \{tv, \ell +tv\})\geq \ell$, we have 
\begin{align}\label{e:gtdiff}
\left|g_t(z)-\ln\frac{\sin( v)}{ \sin(1-v)}\right|=\OO\left(\frac{\delta}{\dist(z, \{tv, \ell+tv\})}\right)=\OO\left(\frac{\delta}{|z|}\right).
\end{align}
We can replace $g_t(z)$ in \eqref{e:DLE2} by $g_t(z)-\ln(\sin(v)/\sin(1-v))$ and contour integral does not change. To replace $m_t(z)$ in \eqref{e:DLE2} by $m_t^*(z)$, the error is given by
\begin{align}\begin{split}\label{e:replacemm*2}
&\phantom{{}={}}\left|\int_0^1\oint_\omega\ln (1+e^{   m_t(z)+\tau  g_t (z)}) g_t (z)\rd z\rd \tau-\oint\int_0^1\ln (1+e^{   m^*_t(z)+\tau  g_t (z)}) g_t (z)\rd z\rd \tau\right|\\
&\leq \int_0^1\oint_\omega |\ln (1+e^{   m_t^*(z)+\tau  g_t (z)})-\ln (1+e^{   m_t(z)+\tau  g_t (z)})| \left|g_t(z)-\ln\frac{\sin( v)}{ \sin(1-v)}\right|\rd z\rd \tau.
\end{split}\end{align}

For the difference $|\ln (1+e^{   m_t^*(z)+\tau  g_t (z)})-\ln (1+e^{   m_t(z)+\tau  g_t (z)})|$, we need to upper bound $|m_t^*(z)-m_t(z)|$.
Recall from \Cref{e:tmm*diff}, for $\dist(z, \{tv, \ell+tv\})\geq\delta$
\begin{align}\label{e:difln2}
|m_t^*(z)-\widetilde m_t(z)|\lesssim \frac{\delta}{\delta+\dist(z, \{tv, \ell+tv\})}.
\end{align}
Next we bound $|m_t(z)-\widetilde m_t(z)|$. Since $\Im[z]\geq 1/K$, the same as in  \eqref{e:mm*diff}, we have $|m_t(z)-\wt m_t(z)|\leq \varepsilon \pi K$. They together give a simple bound $|m_t(z)-m_t^*(z)|\leq \delta/\dist(z,\{tv, \ell+tv\})+\varepsilon K\pi\ll 1$ for $\dist(z, \{tv, \ell+tv\})\gg \delta$. It follows that
\begin{align}\begin{split}\label{e:difln}
&\phantom{{}={}}|\ln (1+e^{m^*_t(z)+\tau g_t(z)})-\ln (1+e^{m_t(z)+\tau g_t(z)}))|
= 
\left|\ln \left(1+\frac{(e^{m_t(z)-m_t^*(z)}-1) e^{m^*_t(z)+\tau g_t(z)}}{1+e^{m^*_t(z)+\tau g_t(z)}}\right) \right|\\
&\leq \left|\frac{(e^{m_t(z)-m_t^*(z)}-1) e^{m^*_t(z)+\tau g_t(z)}}{1+e^{m^*_t(z)+\tau g_t(z)}}\right|
\lesssim \frac{|m_t(z)-m_t^*(z)|}{\zeta}\lesssim\frac{|m_t(z)-\widetilde m_t(z)|+|\widetilde m_t(z)- m^*_t(z)|}{\zeta},
\end{split}\end{align}
where in the second to last inequality we used \eqref{e:ftaubound}.

Next we will prove an improved estimate of the difference $|m_t(z)-\widetilde m_t(z)|$ when $z$ is far away from $0$. For $\dist(z, \{tv, \ell+tv\})\geq 3\ell$,
\begin{align}\label{e:mtdiff}
|m_t(z)-\widetilde m_t(z)|\lesssim \frac{\delta}{|z|}+\int_{|x|\geq \Re[z]/2}\frac{\del_x \cH(x,t) \rd x}{|z-x|}.
\end{align}
By symmetry, we will only prove \eqref{e:mtdiff} for $\Re[z]\leq -\ell$, then $\dist(z, \{tv,\ell+tv\})\asymp \Re[z]\geq \ell$. 
\begin{align}\begin{split}\label{e:diffmtm}
|m_t(z)-\widetilde m_t(z)|
&=\left|\int_\bR \frac{\del_x \cH(x,t)-\del_x \widetilde H(x,t)}{z-x}\rd x\right|\\
&\leq\left|\int_{x\leq \Re[z]/2} \frac{\del_x \cH(x,t)-\del_x \widetilde H(x,t)}{z-x} \rd x\right|
+\left|\int_{x\geq \Re[z]/2} \frac{\del_x \cH(x,t)-\del_x \widetilde H(x,t)}{z-x}\rd x\right|\\
&=\left|\int_{x\leq \Re[z]/2} \frac{\del_x \cH(x,t)-\del_x \widetilde H(x,t)}{z-x}\rd x\right|
+\OO\left(\frac{\varepsilon}{|z|}\right),
\end{split}\end{align}
where the last term is bounded by an integration by part and $|z-x|\gtrsim |z|$. For the first term on the righthand side of \eqref{e:diffmtm}
and we have
\begin{align}\label{e:dtH}
\left|\int_{x\leq \Re[z]/2} \frac{\del_x \widetilde H(x,t)}{z-x}\rd x\right|
\leq \left|\int_{x\leq 2\Re[z]} \frac{\varrho \kappa_t(x)}{z-x}\rd x\right|+\left|\int_{2\Re[z]\leq x\leq \Re[z]/2} \frac{\varrho \kappa_t(x)}{z-x}\rd x\right|
\lesssim \frac{\delta \ln(|z|K)}{|z|},
\end{align}
where we used \Cref{c:nabHbound} to bound $\kappa_t$ in the last term. 
The claim \eqref{e:mtdiff} follows from combining \eqref{e:diffmtm} and \eqref{e:dtH}.
%
%
%
%
%

We decompose the contour integral $\omega=\{z: \Im[z]=\pm 1/K\}$ on the righthand side of \eqref{e:replacemm*2} into the following three parts
\begin{enumerate}
\item For $\dist(z, \{tv, \ell+tv\})\leq \delta$, using \Cref{l:Bzcondition} and \eqref{e:introepsilon} that $|\ln (1+e^{   m_t^*(z)+\tau  g (z)})-\ln (1+e^{   m_t(z)+\tau  g (z)})|\lesssim \ln(\ell^2 K/\delta \zeta)\lesssim \ln(\ell/\zeta\delta)$, we can bound the integral as 
\begin{align}\label{e:case1}
\oint_{\dist(z, \{tv, \ell+tv\})\leq \delta} \ln (\ell/\zeta\delta) (\ln(1/\zeta)+\ln(\ell/\delta)+C)|\rd z|
\lesssim \delta\ln^2(\ell /\delta).
\end{align}
\item For $\delta\leq \dist(z, \{tv, \ell+tv\})\leq 10\ell$, using \eqref{e:difln} and \eqref{e:difln2}, we can bound the integral as
\begin{align}\begin{split}\label{e:case2}
&\phantom{{}={}}\oint_{\delta\leq \dist(z, \{tv, \ell+tv\})\leq 10\ell}\frac{\varepsilon K+\delta/\dist(z, \{tv, \ell+tv\})}{\zeta} (\ln(1/\zeta)+\ln(\ell/\delta)+C)|\rd z|
\\
&\lesssim \varepsilon \ell \ln(\ell/\delta)/(\delta\zeta^2)
+\delta\ln^2(\ell/\delta)/\zeta.
\end{split}\end{align}
\item For $\dist(z, \{tv, \ell+tv\})\geq 10\ell$, using  \eqref{e:difln}, \eqref{e:difln2}, \eqref{e:mtdiff} and \eqref{e:gtdiff}, we can bound the integral as
\begin{align}\begin{split}\label{dztv}
&\phantom{{}={}}\oint_{\dist(z, \{tv, \ell+tv\})\geq 10\ell}\frac{1}{\zeta}\left(\frac{\delta \ln(|z| K)}{|z|}+\int_{|x|\geq |z|/2}\frac{\del_x \cH(x,t) \rd x}{|z-x|}\right)\frac{\delta}{|z|}|\rd z|\\
&\lesssim
\frac{\delta^2 \ln(\ell/\zeta\delta)}{\zeta\ell}+\int_{|x|\geq 3\ell} \del_x \cH(x,t) \int_{6\ell\leq |z|\leq 2|x|}  \frac{\delta}{\zeta |z-x||z|}|\rd z| \rd x,
\end{split}\end{align}
where we used that $K\asymp 1/\zeta\delta$, the inner integral is integrable 
\begin{align}\label{e:inner}
\int_{6\ell\leq |z|\leq 2|x|}  \frac{\delta}{\zeta |z-x||z|}|\rd z|\lesssim \frac{\delta \ln(\ell K)}{\ell\zeta},
\end{align}
and it follows from combining \eqref{dztv} and \eqref{e:inner}, and using \eqref{e:Htboundhi}
\begin{align}\label{e:case3}
\oint_{\dist(z, \{tv, \ell+tv\})\geq 3\ell}\frac{1}{\zeta}\left(\frac{\delta \ln(|z| K) }{|z|}+\int_{|x|\geq \Re[z]/2}\frac{\del_x \cH(x,t) \rd x}{|z-x|}\right)\frac{\delta}{|z|}|\rd z|\lesssim \frac{\delta^2\ln(\ell/\zeta\delta)}{\ell\zeta}.
\end{align}
\end{enumerate}
By plugging \eqref{e:case1}, \eqref{e:case2} and \eqref{e:case3} into \eqref{e:replacemm*2}, we get
\begin{align}\begin{split}\label{e:replacemm*3}
&\phantom{{}={}}\left|\int_0^1\oint_\omega\ln (1+e^{   m_t(z)+\tau  g (z)}) g (z)\rd z\rd \tau-\oint\int_0^1\ln (1+e^{   m^*_t(z)+\tau  g (z)}) g (z)\rd z\rd \tau\right|\\
&\lesssim \frac{\varepsilon  \ell \ln(\ell/\delta)}{\delta\zeta^2}+\frac{\delta^2\ln(\ell/\zeta\delta)}{\ell\zeta}+\delta \ln^2(\ell/\delta)\lesssim \frac{\varepsilon  \ell \ln(\ell/\delta)}{\delta\zeta^2} +\delta \ln^2(\ell/\delta)
\end{split}\end{align}
The claim \eqref{e:DLE} follows from combining \eqref{e:replacemm*3} and \eqref{e:DLE2}, and $\delta\ll \zeta\ell$ from \eqref{e:introepsilon}.

%
%
\end{proof}

\subsection{Proof of \Cref{p:limitshape}}
In this section we prove \Cref{p:limitshape}. We recall the smoothed height function $\wt H$ from \eqref{e:deftH}, the associated complex slope $\wt f_t(z)$ from \eqref{fh}, and the drift $g_t(z)$ from \eqref{e:gtx}. The following lemma collects some properties of $\wt f_t(z)$ and $g_t(z)$ from \Cref{s:heightslope}

\begin{lemma}\label{a:Mk}
The Markov process \eqref{e:Mk} satisfies
\begin{enumerate}
\item There exists a constant $A=3\ell>0$, such that $\supp(\bmx(t))\subset [-A,A]$.
\item The height function $\wt H(x,t)\in \Adm(\mathfrak R, h)$ such that for $x\in \bR$, and a complex slope $f_t(x)\in \bH^-$ which satisfies
\begin{align*}\begin{split}
&-\arg\wt f_t(x)=\pi \partial_x\wt H(x,t),\\
&\arg(1+\wt  f_t(x))=\pi \del_t \wt H(x,t).
\end{split}\end{align*}
\item There exists a neighborhood $\Lambda$ of $[-A,A]$, such that $\wt f_t(x)$ extends continuously to $\Lambda\cap \overline \bH$ as 
$\wt f_t(z)=e^{\wt m_t(z)+g_t(z)}$, where $\wt m_t(z)$ is the Stieltjes transform of $\wt \varrho_t(x)=\partial_x \wt H(x,t)$,
\begin{align}
\wt m_t(z)=\int_\bR \frac{\wt \varrho_t(x)\rd x}{x-z},
\end{align}
and $g_t(z)$ is analytic on $\Lambda$, and $g_t(\overline z)=\overline{g_t(z)}$.
\item We assume that $\ln(1+\wt f_t(z))$ is well defined on $\Lambda\setminus [-A,A]$ and is uniformly bounded on any compact subset of $\Lambda\setminus [-A,A]$.
\end{enumerate}
\end{lemma}
\begin{proof}
The second statement is from the definition \eqref{fh}; the third and fourth statement follows from \Cref{l:gtestimate}.
\end{proof}


From the second the third statement in \Cref{a:Mk}, we have
\begin{align}\label{e:Burger1}
    \del_t \wt m_t(w)
    &=\del_t \int_\bR \frac{\wt \varrho_t(x)\rd x}{ w - x }
    =\frac{1}{2\pi\ri}\oint_{\omega_-} \frac{\ln(1+\wt f_t(z))}{(w-z)^2}\rd z,
\end{align}
where the contour $\omega_-\in \Lambda$ encloses $[-A,A]$, but not $w$. We can deform the contour $\omega_-$ to enclose $w$, and 
\begin{align}\begin{split}\label{e:Burger2}
    \del_t \wt m_t(w)=\del_t \int_\bR \frac{ \wt \varrho_t(s)\rd s}{ w - s }
    &=-\del_w \ln(1+\wt f_t(w))+\frac{1}{2\pi\ri}\oint_{\omega} \frac{\ln(1+\wt f_t(z))}{(w-z)^2}\rd z\\
    &=-\frac{(\del_w \wt m_t(w)+\del_w g_t(w))\wt f_t(w)}{1+\wt f_t(w)}+\frac{1}{2\pi\ri}\oint_{\omega} \frac{\ln(1+\wt f_t(z))}{(w-z)^2}\rd z\\
    &=-\frac{\del_w \wt m_t(w)\wt f_t(w)}{1+\wt f_t(w)}+\frac{1}{2\pi\ri}\oint_{\omega} \frac{\ln(1+\wt f_t(z))}{(w-z)^2}\rd z.
\end{split}\end{align}
where the contour $\omega\in \Lambda$ encloses $[-A,A]$ and $w$.

In the rest of this section we prove the uniqueness of the solutions of the equation \eqref{e:Burger1}. Assume that there is another solution $\widehat f_t(z)=e^{\widehat m_t(z)+g_t(z)}$  satisfying \eqref{e:Burger1}
\begin{align}\label{e:Burger2}
    \del_t \widehat m_t(w)
    &=\del_t \int_\bR \frac{ \widehat\varrho_t(s)\rd s}{ w - s }
    =\frac{1}{2\pi\ri}\oint_{\omega_-} \frac{\ln(1+\widehat f_t(z))}{(w-z)^2}\rd z,
\end{align}
where $\widehat m_t(z)$ is the Stieltjes transform of $\widehat \varrho_t(s)$ supported inside $[-A,A]$. 
\begin{proposition}\label{p:unique}
If $\widehat m_0(z)=\wt m_0(z)$ for any $z\in \bC\setminus[-A,A]$, then $\widehat m_t(z)=\wt m_t(z)$ for any $t\geq 0$ and $z\in \bC\setminus[-A,A]$.
\end{proposition}
\begin{proof}[Proof of \Cref{p:unique}]
Local uniqueness implies the global uniqueness. We only need to show that the statement of \Cref{p:unique} holds for $t\in [0,\delta]$ for some small $\delta>0$. For the equation \eqref{e:Burger1}, we define the characteristic flow
\begin{align}\label{e:ccff}
 \del_t z_t=\frac{\wt f_t(z_t)}{\wt f_t(z_t)+1}.
\end{align}
By plugging the characteristic flow \eqref{e:ccff} into \eqref{e:Burger2}, we have 
\begin{align}\label{e:mtzt}
\del_t  \wt m_t(z_t)=-\frac{\del_z g_t(z_t) \wt f_t(z_t)}{1+ \wt f_t(z_t)}+\frac{1}{2\pi\ri}\oint_{\omega} \frac{\ln(1+ \wt f_t(w))}{(z_t-w)^2}\rd w.
\end{align}
We take a contour $\cC_0\in \Lambda$ surrounding the interval $[-A,A]$, such that the characteristic flow \eqref{e:ccff} maps $z_0\in \cC_0$ to $z_{t}\in \cC_t$ for $t\geq 0$.  For $t\leq \delta$, by the continuity of the characteristic flow, we have that $\cC_t$ is inside an annulars region $S$, see \Cref{f:annulars}

For $\delta>0$ sufficiently small, there exists some constant $0<c<1$, the following holds uniformly for $z\in S$ and $0\leq t\leq\delta$,
\begin{align}\label{e:fbound}
  |\del_z \widehat m_t(z)|\leq 1/c,\quad  |\del_z g_t(z)|\leq 1/c,\quad |\wt f_t(z)|\leq 1/c, \quad |1+\wt f_t(z)|\geq c.
\end{align}
We denote the control parameter 
\begin{align*}
\Gamma_t:=\sup_{z\in \cC_t}|\widehat m_t(z)-\wt m_t(z)|.
\end{align*}
Again for $\delta>0$ sufficiently small, we can assume $\Gamma_t<c^2/4$. 

\begin{figure}
	\begin{center}
	 \includegraphics[scale=2]{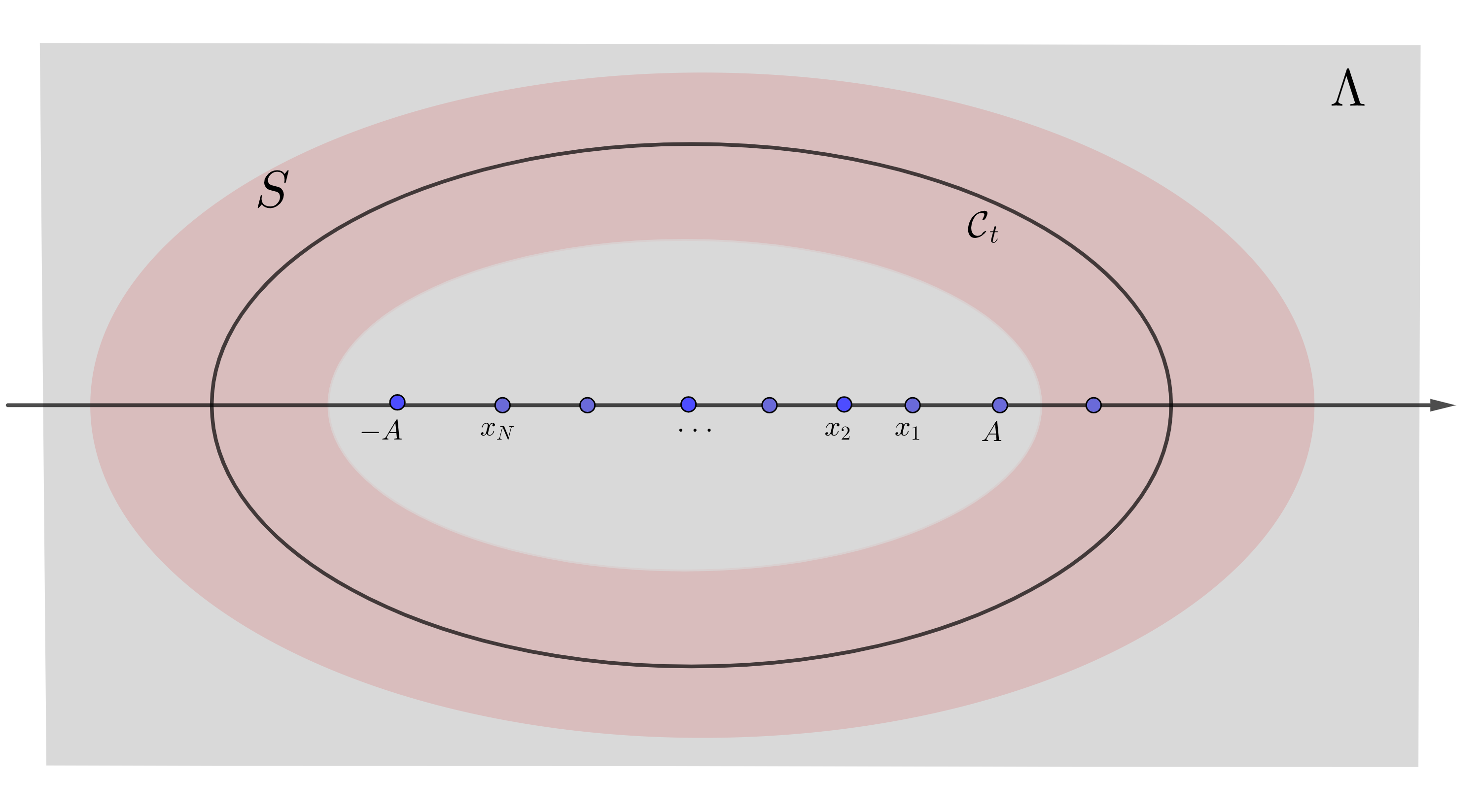}
	 \caption{Shown above is the contour $\cC_t$ inside the annulars region $S$. }
	 \label{f:annulars}
	 \end{center}
	 \end{figure}

We notice that $\lim_{z\rightarrow \infty}|\widehat m_t(z)- \wt m_t(z)|=0$. By maximal principle, $|\widehat m_t(z)- \wt m_t(z)|\leq \Gamma_t$ for any $z$ outside the contour $\cC_t$. 
\begin{align*}
|\widehat f_t(z)- \wt f_t(z)|= |(e^{\widehat m_t(z)-\wt m_t(z)}- 1){\wt f_t(z)}|
\leq (e^{\Gamma_t }- 1)|\wt f_t(z)|\leq 2\Gamma_t/c\leq c/2,
\end{align*}
where in the last two inequalities we used \eqref{e:fbound}. It follows that $|1+\widehat f_t(z)|\geq |1+\wt f_t(z)|-|\wt f_t(z)-\widehat f_t(z)|\geq c/2$, 
and 
\begin{align}\label{e:fdiff2}
\left|\frac{\wt f_t(z)}{1+\wt f_t(z)}-\frac{\widehat f_t(z)}{1+\widehat f_t(z)}\right|
=\frac{|\wt f_t(z)-\widehat f_t(z)|}{|1+\wt f_t(z))(1+\widehat f_t(z))|}\leq \frac{2\Gamma_t/c}{(c/2) c}=\frac{4\Gamma_t}{c^3}.
\end{align}
And finally, we also have that
\begin{align}\label{e:fdiff3}
|\ln (1+\widehat f_t(z))-\ln (1+\wt  f_t(z))|=\left|1+\frac{\widehat f_t(z)-\wt f_t(z)}{1+\wt f_t(z)}\right|\leq \frac{|\widehat f_t(z)-\wt f_t(z)|}{|1+\wt f_t(z)|}\leq \frac{2\Gamma_t/c}{c}=\frac{2\Gamma_t}{c^2}.
\end{align}
By plugging the characteristic flow \eqref{e:ccff} into \eqref{e:Burger2}, we have 
\begin{align}\label{e:tmtzt}
\del_t \widehat m_t(z_t)=-\del_z \widehat m_t(z_t) \left(\frac{\widehat f_t(z_t)}{1+\widehat f_t(z_t)}-\frac{\wt f_t(z_t)}{1+\wt f_t(z_t)}\right)-\frac{\del_z g_t(z_t)\widehat f_t(z_t)}{1+\widehat f_t(z_t)}+\frac{1}{2\pi\ri}\oint_{\omega} \frac{\ln(1+\widehat f_t(w))}{( z_t-w)^2}\rd w,
\end{align}
where the the contour is bounded away from $z_t$, namely $|w-z_t|\geq c$. Moreover, the length of the contour is bounded by $10A$. 

By taking the difference of \eqref{e:mtzt} and \eqref{e:tmtzt}, we have
\begin{align}\begin{split}\label{e:mdiff1}
\del_t (\widehat m_t(z_t)-\wt m_t(z_t))
&=-(\del_z \widehat m_t(z_t)+\del_z g_t(z_t)) \left(\frac{\widehat f_t(z_t)}{1+\widehat f_t(z_t)}-\frac{\wt f_t(z_t)}{1+\wt f_t(z_t)}\right)\\
&\phantom{{}={}}+\frac{1}{2\pi\ri}\oint_{\omega} \frac{\ln(1+\widehat f_t(w))-\ln(1+ \wt f_t(w))}{( z_t-w)^2}\rd w.
\end{split}\end{align}
Thanks to \eqref{e:fbound}, \eqref{e:fdiff2} and \eqref{e:fdiff3}, the righthand side of \eqref{e:mdiff1} is bounded by
\begin{align*}
\left|(\del_z \widehat m_t(z_t)+\del_z g_t(z_t)) \left(\frac{\widehat f_t(z_t)}{1+\widehat f_t(z_t)}-\frac{\wt f_t(z_t)}{1+\wt f_t(z_t)}\right)\right|\leq \frac{2}{c}\frac{2\Gamma_t}{c^2}=\frac{4\Gamma_t}{c^3}. 
\end{align*}
and
\begin{align*}
\left|\frac{1}{2\pi\ri}\oint_{\omega} \frac{\ln(1+\widehat f_t(w))-\ln(1+ \wt f_t(w))}{( z_t-w)^2}\rd w\right|\leq \frac{2\Gamma_t}{c^2} \frac{1}{2\pi}\oint_\omega \frac{|\rd w|}{|z_t-w|^2}
\leq \frac{2\Gamma_t}{c^2}\frac{1}{2\pi} \frac{10A}{c^2}
=\frac{10A\Gamma_t}{\pi c^4}.
\end{align*}
It then follows that
\begin{align*}
\del_t \Gamma_t\leq \frac{4\Gamma_t}{c^3}+\frac{10A\Gamma_t}{\pi c^4}=\left(\frac{4}{c^3}+\frac{10A}{\pi c^4}\right)\Gamma_t.
\end{align*}
Since $\Gamma_0=0$, Gr{\"o}nwell inequality implies that $\Gamma_t=0$. Since $\widehat m_t(z)-\wt m_t(z)$ is analytic outside $[-A,A]$, it follows that $\widehat m_t(z)=\wt m_t(z)$ and the solution is unique.
\end{proof}

\begin{proof}[Proof of \Cref{p:limitshape}]

We denote the empirical particle density $ \rho_t(x)=\rho(x;\bmx(t))$ and its Stieltjes transform $ m_t(z)$, and the complex slope
\begin{align}\label{e:rhomt}
 \rho_t(x)=\sum_{i=1}^n \bm1([x_i(t), x_i(t)+\theta/N])\rd x, 
\quad
 m_t(z)=\int_\bR \frac{   \rho_t(x)}{  z  - x  }\rd x,  \quad
 f_t(z)=e^{ m_t(z)+g_t(z)}\quad 0\leq t\leq \ell.
\end{align}
\Cref{a:Mk} verifies that the transition probability \eqref{e:Mk} with parameters satisfying Assumptions \Cref{a:asymp} and \Cref{a:stable} (with $(\bml, \bmr)=(-A,A)$) for all large enough $N$. Thanks to \Cref{t:loopstudy},  the time difference of the Stieltjes transform for the empirical particle density \eqref{e:rhomt} satisfy for  for any $z\in \Lambda\setminus [-A,A]$, and $\sft\in [0,\sfL]$, we have as $\eps\to 0$:
 \begin{align}\label{e:dmg2}
N\left( m_{t+1/N}(z)- m_t(z))\right)
=\Delta\cM_t(z)+\frac{1}{2\pi \ri\theta}\oint_{\cin}\frac{\ln(1+ f_t(z))    \rd w}{( w -  z  )^2}
+ \OO\left(\frac{1}{N} \right),\quad t=\sft/N.
\end{align}
Moreover, $\Delta\cM(z)$ are mean $0$ random variables such that $\{\varepsilon^{-1/2}\Delta \cM(z)\}_{z\in \Lambda\setminus[-A,A]}$ are asymptotically Gaussian with covariance given by
\begin{align}\begin{split}\label{e:covT}
&\phantom{{}={}}N\bE\left[ \Delta \cM(z_1),  \Delta \cM(z_2)\right]=\frac{1}{2\pi \ri \theta}\oint_{\cin} \frac{ f_t(z)}{1+ f_t(z)}\frac{ 1}{( w -z_1)^2} \frac{ 1}{( w -z_2)^2} \rd w +\oo(1),
\end{split}\end{align}
where the contour $\cin\subset \Lambda$ encloses $[-A, A]$, but not $z_1, z_2$. We also have that the higher order joint moments of $\{\sqrt N\Delta \cM(z)\}_{z\in \Lambda\setminus[-A,A]}$ converge as $\eps\to 0$ to the Gaussian joint moments.


Let $C([0,\ell], \mathscr M_1(\bR)])$ denote the space of continuous functions from $[0,\ell]$ to the set of probability measures on $\mathbb R$ equipped with the weak topology of measures. Then $\rho(\cdot; \bmx(t))$ represents a random element of this space. In fact, there is a compact set $\mathcal K\subset C([0,\ell], \mathscr M_1(\bR)])$, such that the distribution of $\rho(\cdot; \bmx(t))_{0\leq t \leq \ell}$ is supported on $\mathcal K$ --- this is because measures $\rho(\cdot; \bmx(t))$ are supported inside $[-A, A]$ and the dependence on $t$ is Lipschitz (each element of $\bmx(t)$ jumps at most by $1/N$ when $t$ grows by $1/N$), cf.\ \cite[Lemma 4.3.13]{AGZ}.

Since the space of probability measures on a compact set is compact, we conclude that the stochastic processes $\rho(\cdot; \bmx(t))_{0\leq t\leq \ell}$ have subsequential limits in distribution as $N\to \infty$. We let $(\rho^*_t)_{0\leq t\leq \ell}$ be one of the limiting points. Our task is to show that $\rho^*_t$ are as described in \Cref{p:limitshape} (Implying, in particular, that all the limiting points are the same.) Note that for each $0\leq t\leq \ell$, $\rho^*_t$ is an absolutely continuous measures of density at most $1$, because so were the prelimit measures.

For any $t\in \qq{0,\sfL}/N$, using \eqref{e:dmg2}, we have
\begin{align}\label{e:martingale_equation}
 m_t(z)=  m_0(z)+\frac{1}{N}  \cM_t(z)+
\sum_{\tau \in N^{-1}\qq{0, \sft-1}} \frac{1}{2\pi \ri\theta N}\oint_{\cin}\frac{\ln(1+ f_t(z))    \rd w}{( w -  z  )^2}
+\OO\left(\frac{1}{N} \right),
\end{align}
where $\cM_t(z)$ is a martingale given by
\begin{align}\label{e:deftMt}
\cM_t(z)\deq \sum_{\tau \in N^{-1}\qq{0, \sft-1}}  \Delta \cM_{\tau}(z).
\end{align}
We can estimate the $\cM_t(z)$ term by Doob/Kolmogorov's inequality for martingales:
$$
 {\rm Prob} \left(\sup_{t\in N^{-1}\qq{0,\sfL}} | \cM_t(z)/N|>\lambda \right)\leq \frac{1}{\lambda^2N^2}\sum_{\tau \in N^{-1}\qq{0,\sfT-1}}  \bE [\Delta \cM_{\tau}(z)]^2, \qquad \lambda>0.
$$
Using \eqref{e:covT}, each term in the last sum is $\OO(1)$. There are $\OO(N)$ terms and therefore the probability decays as $\OO(1/N \lambda^2)$. We conclude that the martingale part in \eqref{e:martingale_equation} goes to $0$ in probability as $N\rightarrow \infty$. Hence, $N\rightarrow \infty$ limit of \eqref{e:martingale_equation} gives an integral equation for $\rho^*_t$:
\begin{align}\begin{split}\label{e:LLN_integral_equation}
m^*_t(z)= \wt m_0(z)+
\int_0^t\frac{1}{2\pi \ri}\oint_{\cin}\frac{\ln(1+ f^*_s(z))    \rd w}{( w -  z  )^2}\rd s,
\end{split}\end{align}
where $f^*_{s}(z)=e^{m^*_s(z)+g_s(z)}$ is the limit of $ f_{s}(z)$. We arrive at the same partial differential equation as \eqref{e:Burger1}, and \Cref{p:unique} implies that $m^*_s(z)=\widetilde m_t(z)$. \Cref{p:limitshape} follows.

\end{proof}

\section{Asymptotics for Jack Polynomials}\label{s:Jprocess}

In this section we explain the correspondence between  non-intersecting $\theta$-Bernoulli walk ensembles and certain Jack ascending process \eqref{e:Jprocess}. This correspondence will be used to derive large deviation asymptotics for (skew) Jack polynomials. We collect some basic properties of Jack symmetric functions  in \Cref{s:YoungD} and \Cref{s:Jack}. We recall  the Jack ascending process in \Cref{s:Jackprocess}. Our main references are \cite{MR3443860} and \cite{MR1014073}. In \Cref{s:proofmain2}, we prove our main result \Cref{t:main2}.

\subsection{Young Diagrams and Symmetric Functions}\label{s:YoungD}
Given a Young diagram $\bmla$, a box $\Box\in \bmla$ is a pair of integers, 
\begin{align*}
\Box=(i,j)\in {\bm\lambda}, \text{ if and only if } 1\leq i\leq \ell(\bm\lambda), 1\leq j\leq \lambda_i.
\end{align*} 
We denote $\bmla'$ the transposed diagram  of $\bm\lambda$, defined by 
\begin{align*}
\lambda_j'=|\{i: 1\leq j\leq \lambda_i\}|, \quad 1\leq j\leq \la_1.
\end{align*} 
For a box $\Box=(i,j)\in \bmla$, its arm $a_\Box$ and leg $l_\Box$ are\begin{align*}
a_\Box=\lambda_i-j,\quad l_\Box=\lambda_j'-i.
\end{align*}

Let $\bm x=(x_1, x_2,x_3,\cdots)$ be an infinite set of indeterminates, and $\Sym$ the graded algebra of symmetric functions in infinitely many variables over $\bR$. We use the following notations for certain symmetric functions indexed by partitions $\bmla$:
\begin{itemize}
\item We denote the monomial symmetric functions $m_{\bmla}(\bmx)$, which is defined as the sum of all monomials $\bm{x}^{{\bmla_\sigma}}$ where ${\bmla_\sigma}$ ranges over all distinct permutations of $\bmla$.
\item We denote the power sum symmetric functions $p_{\bmla}(\bmx)$, as
\begin{align*}
p_{\bmla}(\bmx)=\prod_{1\leq i\leq \ell({\bmla})}p_{\la_i}(\bmx),\quad p_n(\bmx)=x_1^n+x_2^n+x_3^n+\cdots.
\end{align*}
\end{itemize}
Both the set of monomial symmetric functions $\{m_\bmla(\bmx): \bmla\in \bY\}$, and the set of power sum symmetric functions $\{p_\bmla(\bmx): \bmla\in \bY\}$ form homogeneous bases of $\Sym$. 

\subsection{Jack Symmetric Polynomials}\label{s:Jack}
Let $\theta>0$ be a parameter (indeterminante), and let $\bQ(\theta)$ denote the field of all rational functions of $\theta$ with rational coefficients. Define a scalar product $\langle\cdot, \cdot\rangle$ on the vector space $\Sym \otimes \bQ(\theta)$ of all symmetric functions over the field $\bQ(\theta)$ by the condition
\begin{align}\label{e:scalarP}
\langle p_{\bmla}(\bmx), p_{\bmmu}(\bmx)\rangle =\delta_{\bmla\bmmu}z_{\bmla} \theta^{-\ell(\bmla)}.
\end{align}

We recall that a partial order on Young diagrams is obtained by declaring $\bm\lambda \preceq \bmmu$ if $|\bm\lambda| = |\bmmu|$ and $\la_1+\la_2+\cdots+\la_i\leq \mu_1+\mu_2+\cdots+\mu_i$ for all $i$. The following fundamental result is due to Macdonald, which characterizes Jack symmetric functions.

Let $\theta>0$ be a parameter (indeterminante). Jack symmetric functions $J_\bmla(\bmx;\theta)$ are elements of the algebra $\Sym$ of the symmetric functions in infinitely many variables $(x_i)_{i=1}^\infty$ uniquely determined by the following two properties:
\begin{enumerate}
\item $J_\bmla$, $|\bm\lambda|=m$, can be expressed in terms of the monomial symmetric functions via a strictly upper unitriangular transition matrix:
$$
J_\bmla=m_\bmla+\sum_{\bmmu<\bmla\in \bY_m}R_{\bmla\bmmu}m_{\bmmu},
$$
where $R_{\bmla \bmmu}$ are functions of $\theta$ and $\bmmu<\bmla$ is comparison in the dominance order on the set $\bY_m$ of all partitions of $m$ (equivalently, Young diagrams with $m$ boxes).
\item They are pairwise orthogonal with respect to the scalar product defined on the power sums via
\begin{align}\label{e:scalarP}
\langle p_\bmla,p_\bmmu\rangle_{\theta}=\delta_{\bmla\bmmu}z_\bmla \theta^{-\ell(\bm\la)}, \qquad z_{\bm\lambda}=\prod_{1\leq i\leq \la_1} i^{m_i(\bmla)}\prod_{1\leq i\leq \la_1}m_i(\bmla)!,
\end{align}
where $\bmla=1^{m_1}2^{m_2}\dots$, i.e.\ $m_i$ is the multiplicity of $i$ in $\bmla$, $\ell(\bmla)$ is the number of rows, and $p_k=(x_1)^k+(x_2)^k+\dots$, $k\geq 1$.
\end{enumerate}

We recall the scalar product on $\Sym$ as defined in \eqref{e:scalarP}, and introduce the following notation,
\begin{align}\label{e:jl}
j_{\bmla}(\theta)=\langle J_{\bmla}(\bmx;\theta), J_{\bmla}(\bmx;\theta)\rangle.
\end{align}
The number $j_{\bmla}(\theta)$ is explicitly given by
\begin{align}\label{def:jl2}
j_{\bmla}=\prod_{\Box\in \bmla}\frac{a_{\Box}+\theta l_{\Box}+1} {a_{\Box}+\theta l_{\Box}+\theta}.
\end{align}
See \cite[Theorem 5.8]{MR1014073} for a proof. 
We denote the dual polynomial $\widetilde J_{\bmla}(\bmx;\theta)=J_\bmla(\bmx;\theta)/ j_{\bm\lambda}(\theta)$.
Finally, the skew Jack polynomials $J_{\bmla/\bmmu}, \widetilde J_{\bmla/\bmmu}$ are defined through the expansions:
\begin{align*}
J_\bmla(x_1,x_2,\dots, y_1,y_2,\dots;\theta)&=\sum_{\bmmu} J_{\bmla/\bmmu}(x_1,x_1,\dots;\theta) J_\bmmu(y_1,y_2,\dots; \theta),\\
 \widetilde J_\bmla(x_1,x_2,\dots, y_1,y_2,\dots; \theta)&=\sum_{\bmmu}  \widetilde J_{\bmla/\bmmu}(x_1,x_1,\dots;\theta)  \widetilde J_\bmmu(y_1,y_2,\dots; \theta).
\end{align*}

%

\begin{theorem}\label{t:J1N}
We have 
\begin{align}\label{e:Jlambda}
J_{\bmla}(1^N;\theta)=\prod_{\Box\in \bmla} \frac{N\theta+(j-1)-\theta (i-1)}{a_\Box+\theta l_\Box+\theta}.
\end{align}
\end{theorem}
\begin{proof}
See \cite[Proposition 2.3]{MR1014073}.
\end{proof}

The following lemma gives the asymptotics of the Jack symmetric polynomials at the principal specialization
\begin{lemma}
Given a Young diagram $\bmla\in \bY_N$, we identify it as a particle configuration
\begin{equation}\label{e:defxfromla}
\bm\sfx=(\sfx_1,\sfx_2,\dots, \sfx_{N})\in \bW_\theta^{N},\quad \sfx_i=\la_i-\theta(i-1) \quad1\leq i\leq N,
\end{equation}
then 
\begin{align}\label{e:Jlambda2}
J_{\bmla}(1^N;\theta)=\prod_{i<j}\frac{\Gamma(\sfx_i-\sfx_j+\theta)}{\Gamma(\sfx_i-\sfx_j)}\prod_{i=1}^N\frac{\Gamma(\theta)}{\Gamma(i\theta)}.
\end{align}
Given  a sequence of Young diagrams
\begin{align*}
\bm\la^{(N)}=(\la_1^{(N)}\geq \la_2^{(N)}\geq\cdots\geq \la_N^{(N)})\in \bY_N, \quad N\geq 1
 \end{align*} 
 such that 
\begin{enumerate}
\item
There exists a constant $C>0$, $\la_1^{(N)}\leq CN$
\item  There exists a $1$-Lipschitz nondecreasing function $h:\bR\mapsto [0,\theta]$, and when $N\rightarrow \infty$
\begin{align*}
\frac{\theta}{N}\sum_{i=1}^N\delta\left(\frac{\lambda_i^{(N)}-(i-1)\theta}{N}\right)\rightarrow \del_x h(x),
\end{align*}
in distribution. Then
\end{enumerate}
\begin{align}\label{e:NlimitJ1}
\lim_{N\rightarrow\infty}\frac{1}{N^2}\ln J_{\bmla^{(N)}}(1^N;\theta)=\frac{1}{2\theta}\int_{\bR^2}\ln|x-y|\rd h(x)\rd h(y) -\frac{\theta\ln(\theta)}{2}+\frac{3\theta}{4}. 
\end{align}
\end{lemma}
\begin{proof}
We can reorganize the numerator and denominator in \eqref{e:Jlambda}, and rewrite it in terms of the particle configuration $\bm\sfx$ from \eqref{e:defxfromla}.
\begin{align*}
\prod_{\Box\in \bmla} N\theta+(j-1)-\theta (i-1)
=\prod_{i=1}^N\frac{\Gamma(N\theta+\sfx_i)}{\Gamma((N-i+1)\theta)}
=\prod_{i=1}^N\frac{\Gamma(N\theta+\sfx_i)}{\Gamma(i\theta)},
\end{align*}
and
\begin{align*}
\prod_{\Box\in \bmla} \frac{1}{a_\Box+\theta l_\Box+\theta}
=\prod_{i\leq j}\frac{\Gamma(\lambda_i-\lambda_j+\theta(j-i)+\theta)}{\Gamma(\lambda_i-\lambda_{j+1}+\theta(j-i)+\theta)}
=\Gamma(\theta)^{N}\prod_i \frac{1}{\Gamma(\sfx_i+N\theta)}\prod_{i<j}\frac{\Gamma(\sfx_i-\sfx_j+\theta)}{\Gamma(\sfx_i-\sfx_j)}.
\end{align*}
The claim \eqref{e:Jlambda2} follows from combining the above two expressions.

For the asymptotics \eqref{e:NlimitJ1}, simply write 
$\lambda_i^{(N)}-(i-1)\theta=\sfx_i$.
We notice that there exists $C=C(\theta)>0$, Stirling's formula implies that the log Gamma function satisfies $(z-1/2)\ln z-z-C\leq \ln\Gamma(z)\leq (z-1/2)\ln z-z+C$ for $z\geq \theta$, and $|\del_z \ln \Gamma(z)-\ln z|\leq C/z$ for $z\geq \theta$.

So we have
\begin{align}\begin{split}\label{e:Jlimitt1} 
\frac{1}{N^2}\sum_{i=1}^N(\ln \Gamma(\theta)-\ln\Gamma(i\theta))
&=-\frac{1}{N^2}\sum_{i=1}^N \left(i\theta \ln(i\theta/e)+\OO(N)\right)\\
&=-\frac{\theta\ln N}{2}-\theta \int_0^1 x\ln x\rd x-\frac{\theta}{2}\ln(\theta/e)+\OO(\ln N/N)\\
&=-\frac{\theta\ln N}{2}-\frac{\theta\ln(\theta)}{2}+\frac{3\theta}{4}+\OO(\ln N/N),
\end{split}\end{align}
and
\begin{align}\begin{split}\label{e:Jlimitt2}
&\phantom{{}={}}\sum_{i<j}\ln \frac{\Gamma(\sfx_i-\sfx_j+\theta)}{\Gamma(\sfx_i-\sfx_j)}
=\sum_{i<j}\theta\ln (\sfx_i-\sfx_j)+\OO\left(\sum_{i<j}\frac{1}{\sfx_i-\sfx_j}\right)\\
&=\sum_{i<j}\theta\ln (\sfx_i-\sfx_j)+\OO\left(N\ln N\right)
=N^2\left(\frac{\theta \ln N}{2}+\frac{1}{2\theta}\iint_{\bR^2} \ln|x-y|\rd h(x)\rd h(y)\right)+\OO(N\ln N).
\end{split}\end{align}
where we used that $\sfx_i-\sfx_j\geq \theta(i-j)$ for the second inequality and  in the last equality we used \Cref{l:freeentropy}. The claim  \eqref{e:NlimitJ1} follows from plugging \eqref{e:Jlimitt1} and \eqref{e:Jlimitt2} into \eqref{e:Jlambda2}.
\end{proof}

\subsection{Jack Process}\label{s:Jackprocess}
Computations in the algebra of symmetric functions $\Sym$ can be converted into numeric identities by means of \emph{specializations}, which are algebra homomorphism from $\Sym$ to the set of complex numbers. A specialization $\rho$ is uniquely determined by its values on any set of algebraic generators of $\Sym$ and we use $(p_k)_{k=1}^{\infty}$ as such generators. The value of $\rho$ on a symmetric function $f$ is denoted $f(\rho)$.
Given two specializations $\rho,\rho'$, we define their union $(\rho,\rho')$ through the formula:
$$
p_k(\rho,\rho')=p_k(\rho)+p_k(\rho'),\quad k\geq 1.
$$

A specialization $\rho$ is an algebra homomorphism from $\Sym$ to the set of complex numbers. A specialization $\rho$ is called \emph{Jack-positive} if its values on all (skew) Jack symmetric functions with a fixed $\theta$ are real and non-negative, i.e.,
\begin{align*}
J_\bmla(\rho;\theta)=\rho(J_\bmla(\bmx;\theta))\geq 0, 
\end{align*}
for all $\bmla\in \bY$. The set of Jack positive specializations are classified by 
Kerov, Okounkov and Olshanski in
\cite[Theorem A]{MR1609628} (We slightly modify the statement (our $\beta_i$ is $\theta \beta_i$ in \cite[Theorem A]{MR1609628}) such that it matches with \Cref{Theorem_Macdonald_positive}).
\begin{theorem}\label{t:Jackspecialization}
For any fixed $\theta>0$, Jack positive specializations can be parameterized by triplets $(\bm\alpha, \bm\beta,\gamma)$, where $\bm\alpha, \bm\beta$ are sequences of real numbers with 
\begin{align*}
\alpha_1\geq \alpha_2\geq \cdots \geq 0, \quad
\beta_1\geq \beta_2\geq \cdots\geq 0,\quad
\sum_i \alpha_i+\frac{\beta_i}{\theta}<\infty,
\end{align*}
and $\gamma$ is a non-negative real number. The specialization corresponding to a triplet $(\bm\alpha, \bm \beta, \gamma)$ is given by its values on Newton power sums $p_k$,
$$
 p_1(\rho)=\sum_{i=1}^\infty \alpha_i+ \frac{1}{\theta}\left(\gamma+ \sum_{i=1}^{\infty}\beta_i\right) , \qquad p_k(\rho)=\sum_{i=1}^{\infty} (\alpha_i)^k + \frac{(-1)^{k-1}}{\theta} \sum_{i=1}^{\infty} (\beta_i)^k.
$$
\end{theorem}

Following \cite[Section 2.3]{gorin2015multilevel}, we create Markov chains out of the Jack-positive specializations:
\begin{definition} \label{Definition_Jack_ascending}
 Given two specializations $\rho$ and $\rho'$ we define the ascending transition through
\begin{equation}
\label{eq_Jascending_transition}
 \bP(\bmla\mid \bmmu)= \frac{1}{H_\theta(\rho;\rho')} \frac{J_{\bmla}(\rho; \theta)}{ J_{\bmmu}(\rho; \theta)}\widetilde J_{\bmla/\bmmu}(\rho';\theta)
,
\end{equation}
where $\bmla$ and $\bmmu$ are partitions with $\bmmu\subset\bmla$ and the normalization constant $H_\theta(\rho;\rho')$ is given by
$$
H_\theta(\rho, \rho')=\exp\left\{\sum_{k\geq 1}\frac{\theta}{k} p_k(\rho) p_k(\rho')\right\}.
$$
\end{definition}

We further use the $w_{\beta}$ automorphism of the algebra of symmetric functions $\Sym$ defined on the power sums by:
$$
 w_{\beta}(p_k)=(-1)^{k-1} \beta p_k.
$$
As shown in \cite[Chapter VI, Section 10, (10.6)]{macdonald1998symmetric},
$$
 w_{\theta^{-1}}\bigl(J_{\bmla'/\bmmu'}(\cdot; \theta)\bigr)= \widetilde J_{\bmla/\bmmu}(\cdot ; \theta^{-1}),
$$
where $\bmla'$ and $\bmmu'$ are transposed Young diagrams $\bmla$ and $\bmmu$, respectively. Hence, with the specialization $\bm\beta=(b_0, b_1, \cdots, b_{\sfT-1})$, as in \eqref{eq_Jascending_transition}, we have
$$
\widetilde J_{\bmla/\bmmu}(\bm\beta=(b_0, b_1, \cdots, b_{\sfT-1});\theta)= J_{\bmla'/\bmmu'}(\bm\alpha=(b_0, b_1, \cdots, b_{\sfT-1});\theta^{-1}).
$$

To obtain the asymptotics of skew Jack polynomials, we study the following special Jack ascending process. We fix $N, \sfT\geq 1$, and a sequence of positive numbers $b_0, b_1, b_2, \cdots, b_{\sfT-1}$. The transition probability at time $0\leq \sft\leq \sfT-1$ is given in the notations of Theorem \ref{t:Jackspecialization} as
\begin{equation} \label{eq_principal_ascending}
 \rho:\, \alpha_i=1,\, 1\leq i \leq N;\qquad \rho': \beta_1=b_\sft,
\end{equation}
with all other parameters set to $0$.  The transition probability at time $0\leq \sft\leq \sfT-1$ is given by 
\begin{align}\label{e:Jprocess}
    \bP(\bmla(\sft+1)=\bmla|\bmla(\sft)=\bmmu)
    =\frac{1}{H_{\theta}(1^N, \beta_1=b_\sft)}\frac{J_{\bmla}(1^N; \theta)}{J_{\bmmu}(1^N; \theta)}\widetilde J_{\bmla/\bmmu}(\beta_1=b_{\sft};\theta).
\end{align}
Then the transition probability gives the skew Jack polynomials
\begin{align}\label{e:jackdynamics}
    \bP(\bmla(\sfT)=\bmla|\bmla(0)=\bmmu)
    &=\prod_{0\leq \sft\leq \sfT-1}\frac{1}{(1+b_\sft)^N}\frac{J_{\bmla}(1^N;\theta)}{J_{\bmmu}(1^N;\theta)}\widetilde J_{\bmla/\bmmu}(\bm\beta=(b_0, \cdots, b_{\sfT-1});\theta).
\end{align}

The following claim states that we can encode the Jack ascending process \ref{e:Jprocess} of Young diagrams as an $N$-particle non-intersecting $\theta$-Bernoulli random walk.

\begin{claim}\label{c:Jackdensity}
The transition probability of the Jack process \ref{e:Jprocess} is non-degenerate only for partitions $\bmla,\bmmu$ with at most $N$ parts, i.e.\ $\bmla=(\lambda_1\geq \lambda_2\geq\dots\geq\lambda_N)$ and $\bmmu=(\mu_1\geq \mu_2\geq\dots\geq\mu_N)$. Further, if we identify
\begin{equation}\label{e:x1}
\bm\sfx=(\sfx_1,\sfx_2,\dots, \sfx_{N})\in \bW_\theta^{N},\quad \sfx_i=\mu_i-\theta(i-1),\quad \sfx_i+e_i=\lambda_i-\theta(i-1), \quad1\leq i\leq N,
\end{equation}
then $e_i\in\{0,1\}$ and the transition probability is 
\begin{align}\label{e:drift_dynamics}
\bP^{\sfb}(\bm\sfx+\bme|\bm\sfx)=
\frac{1}{(1+b_\sft)^N}
\prod_{1\leq i<j\leq N}\frac{(\sfx_i+\theta e_i)-(\sfx_j+\theta e_j)}{{\sfx_i}-{\sfx_j}} b_\sft^{\sum_{i=1}^N e_i}.
\end{align}
\end{claim}

\subsection{Proof of \Cref{t:main2}} \label{s:proofmain2}
Thanks to the relation \eqref{e:jackdynamics}, the large deviation asymptotics of the skew Jack polynomials can be obtained from the large deviation principle of the dynamics \eqref{e:drift_dynamics},
\begin{align}\begin{split}\label{e:jackdynamics2}
    &\phantom{{}={}}\frac{1}{N^2}\ln \widetilde J_{\bmla/\bmmu}(\bm\beta=(b_0, \cdots, b_{\sfT-1});\theta)
    =\frac{1}{N^2}\ln \left(\bP^{\sfb}(\bmla(\sfT)= \bmla|\bmla(0)=\bmmu)H(1^N, \bm\beta=(b_0, \cdots, b_{\sfT-1}))\right)\\
    &+\frac{1}{2\theta}\left(\int_{\bR^2}\ln|x-y|\rd h(x, 0)\rd h(y,0)-\int_{\bR^2}\ln|x-y|\rd h(x, T)\rd h(y,T) \right)+\oo(1),
\end{split}\end{align}
where we used \Cref{e:defxfromla}.
	The dynamics \eqref{e:drift_dynamics} and the non-intersecting $\theta$-Bernoulli random walks \eqref{e:mdensity2} differ by the following drift
\begin{align*}
{D(\{{\bm\sfx}(\sft)\}_{0\leq \sft\leq \sfT})}:=\prod_{0\le {\mathsf t}\le \mathsf{T-1}}2^N b_{\mathsf {t}}^{\sum_{i=1}^{N}e_{i}(\mathsf t)},\quad 
H(1^N, \bm\beta=(b_0, \cdots, b_{\sfT-1}))\rd \bP^{\sfb}=D(\{\sf\bmx(\sft)\}_{0\leq \sft\leq \sfT})\rd \bP.
\end{align*}
\Cref{t:main2} follows from plugging the following lemma to \eqref{e:jackdynamics2}.
\begin{lemma}\label{l:drift} 
Adopt the notations and assumptions of \Cref{t:main2}, and denote the height function of $\{\bm\sfx(\sft)\}_{0\leq \sft \leq \sfT}$ as $H(x,t)$, then  the following holds
$$\frac{1}{N^{2}}\ln {D(\{\bm\sfx(\sft)\}_{0\leq \sft\leq \sfT})}=\frac{1}{\theta} \mathcal F^{f}(H)
+T\ln 2+\OO(1/N),$$
with 
\begin{align}\label{defd}
\mathcal F^{f}(H)
:=
-\int_{0}^{T} \int_\bR f(s) \del_s H(y,s)\rd y \rd s.
\end{align}
Moreover, $\cF^{f}$ is continuous on the space of  Lipschitz functions with uniform norm.
\end{lemma} 
\begin{proof} 
We  notice that since $e_{i}(\mathsf t)=\mathsf{x}_{i}(\mathsf t+1)-\mathsf{x}_{i}(\mathsf t)$ and $b_\sft=e^{ f(\sft/N)}$, we have
\begin{align*}\begin{split}
\ln {D(\{\bm\sfx(\sft)\}_{0\leq \sft\leq \sfT})}&= N\sfT\ln 2+ \sum_{0\le {\mathsf t}\le  \mathsf{T}-1}  f\left({\frac{\mathsf{t}}{N}}\right)\sum_{i=1}^{N}(\mathsf{x}_{i}(\mathsf t+1)-\mathsf{x}_{i}(\mathsf t))\\
&= N\sfT\ln 2+ \sum_{1\le {\mathsf t}\le  \mathsf{T}}  f\left({\frac{\mathsf{t-1}}{N}}\right)\sum_{i=1}^{N}\mathsf{x}_{i}(\mathsf{t})- \sum_{0\le {\mathsf t}\le \sfT-1}  f\left({\frac{\mathsf{t}}{N}}\right)\sum_{i=1}^{N} \mathsf{x}_{i}(\mathsf{t})\\
&=N\sfT\ln 2+ \sum_{1\le {\mathsf t}\le  \mathsf{T}-1}
\left( f\left(\frac{\mathsf{t-1}}{N}\right)-f\left({\frac{\mathsf{t}}{N}}\right)\right)
\sum_{i=1}^{N}\mathsf{x}_{i}(\mathsf{t})+f\left({\frac{\mathsf{T}}{N}}\right)\sum_{i=1}^{N}\mathsf{x}_{i}(\mathsf{T})-f(0)\sum_{i=1}^{N}\mathsf{x}_{i}(0)\\
&=N\sfT\ln 2+ \int_{0}^{T}\partial_s f(s)\sum_{i=1}^{N}\mathsf{x}_{i}(\lfloor\mathsf{ N}s\rfloor) \rd s +f(T)\sum_{i=1}^{N}\mathsf{x}_{i}(\mathsf{T})-f(0)\sum_{i=1}^{N}\mathsf{x}_{i}(0).
\end{split}\end{align*}
On the other hand, for $t\in [ \sft/N, (\sft+1)/N]$, we have set $x(t)=\mathsf{x}( \mathsf {t})/N+(Nt- \mathsf {t})e( \mathsf {t})/N$, so that for all integer number $\mathsf{t}$, 
$$\sup_{t\in  [ \sft/N, (\sft+1)/N] }| \sum_{i=1}^{N}\mathsf{x}_{i}(\lfloor\mathsf{ N}t\rfloor) -N\sum_{i=1}^{N} x_{i}(t)|\leq N.$$ 
Therefore, we deduce that
\begin{align*}
\ln {D(\{\bm\sfx(\sft)\}_{0\leq \sft\leq \sfT})}=N\sfT\ln 2 - N\int_{0}^{T}\partial_s f(s)\sum_{i=1}^{N}x_i(s) \rd s +f(T)\sum_{i=1}^{N}\mathsf{x}_{i}(\mathsf{T})-f(0)\sum_{i=1}^{N}\mathsf{x}_{i}(0)+\OO(N),
\end{align*}
where the $\OO(N)$ error is given as $ N\theta^{-1}\int_{0}^{T}|\partial_s f(s)|\rd s$. 
Now it is easy to check  that
\begin{eqnarray*}
\int_\bR y\rho(y;\bmx(t))\rd y&=& \int_\bR y \sum_{i=1}^{N} {1}_{[x_{i}(t), x_{i}(t)+\theta/N]}(y)\rd y\\
&
=&\frac{1}{2} \sum_{i=1}^{N}\left(\left( x_{i}(t)+\frac{\theta}{N}\right)^{2}-x_{i}(t)^{2}\right)
\\
&=& \frac{\theta}{N} \sum_{i=1}^{N}x_{i}(t)+\OO(1/N).
\end{eqnarray*}
By integration by parts we also have since $H(x;\bmx(t))=\int_{-\infty}^{x} \rho(y;\bmx(t)) \rd y$, 
$$\frac{\theta}{N} \sum_{i=1}^{N}x_{i}(t)+\OO(1/N)=\int y\rho(y,\bmx(t))\rd y= [y H(y,t)]_{-K}^{K} -\int_{-K}^{K} H (y,t)\rd y,$$
where we notice that we can restrict the integral on the left to a compact set where the particles lives, say $[-K,K]$ that can be chosen independent of the time (provided the particles at time $0$ are bounded and $T$ is also bounded). Then 
$ [y H(y;\bmx(t))]_{-K}^{K}=\theta K$ is independent of $t$.

As a conclusion, we get: 
\begin{align*}
\frac{1}{N^{2}}\ln{D(\{\bm\sfx(\sft)\}_{0\leq \sft\leq \sfT})}&=\frac{1}{\theta}\left(\int_0^T\partial_s f(s)\int_{-K}^{K} H(y,s)\rd y \rd s -f(T)\int_{-K}^{K} H(y,T)\rd y+f(0)\int_{-K}^{K} H(y,0)\rd y\right)\\
&+T\ln 2+ \OO(1/N).
\end{align*}
Noticing that $H(y,s)$ is Lipschitz, thus differentiable. 
The claim \eqref{defd} follows from an integration by part. Finally, it is clear that if $\partial_s f$ is bounded, $\mathcal F^{f}$ is continuous for the uniform norm. 
\end{proof}

\appendix
\section{Dynamical Loop Equation}
\label{s:DynamicalLoopE}

We define the particle configuration:
\begin{align}\label{e:defWtheta}
 \bW^n\deq \bigl\{(x_1,\dots,x_n)\in\mathbb R^n\mid  x_{i}>x_{i+1}, \quad i=1,2,\dots,n-1\bigr\}.
\end{align}

Given a particle configuration $\bmx=(x_1, x_2,\dots,x_n)\in  \bW^n$, an interaction function $b(z)$, and weight functions $\phi^+(z),\phi^-(z)$. Our central object is the following transition probability
\begin{align}\label{e:m1ccopy}
\bP(\bmx+\bme|\bmx)=\frac{1}{Z(\bmx)} \prod_{1\leq i<j\leq n}\frac{b(x_i+\theta e_i)-b(x_j+\theta e_j)}{b(x_i)-b(x_j)}\prod_{i=1}^n \phi^+(x_i)^{e_i} \phi^-(x_i)^{1-e_i},
\end{align}
where $\bme=(e_1,e_2,\dots, e_n)\in\{0,1\}^n$ and $Z(\bmx)$ is a normalization constant.

\begin{assumption}\label{a:asymp}
Fix constants $\sfN>0$ and $\bl, \br$.
We introduce a small parameter $\varepsilon=\sfN/n\ll 1$, and assume the particle configuration $\bmx=(x_1,x_2,\dots, x_n)\in \widetilde \bW_\theta^n$ satisfies $\bl\leq \varepsilon x_{n}\leq \varepsilon x_1\leq \br-\varepsilon$. We further assume
\begin{enumerate}
\item There exists an open complex neighborhood $\Lambda$ of $[\bl,\br]$, such that for $z=\varepsilon \xi$,
\begin{align}\label{e:asymp}
b(\xi)=  z  \quad \text{ and }\quad
\phi^\pm(\xi)={\varphi}^{\pm}(z), \quad z\in \Lambda.
\end{align}
where $  z  $ and $\varphi^{\pm}(z)$ are holomorphic functions on $\Lambda$. In addition, we require $  z  $ to be conformal (injective and biholomorphic) and $\sfb(\bar{z})=\overline{  z  }$.
\item The functions $  z  $, $[\partial_z   z  ]^{-1}$, and $\varphi^\pm(z)$  are uniformly bounded, namely there exists a universal\footnote{By ``universal'' we mean not depending on $n$, $\eps$, or $(x_1,\dots,x_n)$.} constant $C>0$ such that
\begin{align}\label{e:derbba}
 |   z  |\leq C,
 \quad
 |\del_z   z  |\geq 1/C,\
\quad
|\varphi^\pm(z)|\leq C,\quad z\in  \Lambda.
\end{align}
\end{enumerate}
\end{assumption}

\begin{remark}
In \cite{gorin2022dynamical}, the particle configurations are assumed to be inside the lattice:
\begin{align*}
\bW_\theta^n\deq \bigl\{(x_1,\dots,x_n)\in\mathbb R^n\mid x_1\in \mathbb Z,\quad x_{i}-x_{i+1}\in \theta+\mathbb Z_{\geq 0}, \quad i=1,2,\dots,n-1\bigr\}.
\end{align*}
Here we relaxed the particle configuration to \eqref{e:defWtheta}.
\end{remark}

The dynamical loop equation states that the following observable is a holomorphic function. 
\begin{theorem}[Dynamical Loop Equation]\label{t:loopeq}
Choose an open set $U\subset \mathbb C$, a particle configuration ${\bmx=(x_1>x_2>\dots>x_n)\in \bR^n}$ such that the interval $[x_n,x_1]\subset U$, a parameter $\theta>0$, two holomorphic functions $\phi^+(z)$, $\phi^-(z)$ on $U$ and a conformal (i.e., holomorphic and injective) function $b(z)$ on $U$. Assume that the random $n$--tuple $\bme=(e_1,e_2,\dots, e_n)\in\{0,1\}^n$  is distributed according to the transition probability \eqref{e:m1ccopy}. Then the following observable is a holomorphic function of $z\in U$:
\begin{align}\label{e:sum1}
\bE\left[\phi^+(z)\prod_{j=1}^n\frac{b(z+\theta)-b(x_j+\theta e_j)}{b(z)-b(x_j)}+\phi^-(z)\prod_{j=1}^n\frac{b(z)-b(x_j+\theta e_j)}{b(z)-b(x_j)}\right].
\end{align}
\end{theorem}
\begin{proof}
See \cite[Proof of Theorem 1.1]{gorin2022dynamical}
\end{proof}

Our last assumption for the asymptotic analysis  involves the following functions defined in terms of the empirical density \eqref{eq_rho_x_def} and functions $\phi^\pm(z)$ of Assumption \ref{a:asymp}:
\begin{equation}
\label{e:B_function}
\cB(z)=\cG(z)\varphi^+(z)+\varphi^-(z), \quad \quad  \cG(z)=\exp\left[\theta\int_{\bl}^{\br}\frac{  \rho(s;\bmx)}{  z  - s }\rd s\right].
\end{equation}
Note that $\cB(z)$ is a holomorphic function for $z\in \Lambda\setminus  [\bl,\br]$.
\begin{assumption}\label{a:stable}
There exists annular set $S\subset \bigl(\Lambda\setminus [\bl,\br]\bigr)$ (containing a contour surrounding $[\bl,\br]$), a small universal constant $c>0$ such that  for $z\in S$ we have $c<|\cB(z)|<c^{-1}$. Moreover, for any closed contour $\omega\subset  \bigl(\Lambda\setminus [\bl,\br]\bigr)$, we have
\begin{align}\label{e:aB0}
\frac{1}{2\pi \ri}\oint_{\omega}\frac{\del_z\cB(z)}{\cB(z)}\rd z=0,
\end{align}
which implies that there exists a well-defined single-valued branch of the function $\ln \cB(z)$ in $S$.
\end{assumption}

\begin{theorem}\label{t:loopstudy}
Consider transition probability \eqref{e:m1ccopy} with parameters satisfying Assumptions \ref{a:asymp} and \ref{a:stable} for all small enough $\eps$. Then for any $z\in \Lambda\setminus [\bl,\br]$ we have as $\eps\to 0$:
 \begin{align*}
\frac{1}{\varepsilon}\int_{\bl}^{\br} \frac{  (\rho(s;\bmy)-\rho(s;\bmx))}{  z  - s }\rd s
=\Delta\cM(z)+\frac{1}{2\pi \ri\theta}\oint_{\cin}\frac{\ln \cB(w)    \rd w}{( w -  z  )^2}
+ \OO\left(\varepsilon \right),
\end{align*}

Moreover, $\Delta\cM(z)$ are mean $0$ random variables such that $\{\varepsilon^{-1/2}\Delta \cM(z)\}_{z\in \Lambda\setminus[\bl,\br]}$ are asymptotically Gaussian with covariance given by
\begin{align*}\begin{split}
&\phantom{{}={}}\lim_{\eps\to 0}\bE\left[\varepsilon^{-1/2} \Delta \cM(z_1), \varepsilon^{-1/2} \Delta \cM(z_2)\right]\\
&=\frac{1}{2\pi \ri \theta}\oint_{\cin} \frac{\cG(w)\varphi^+(w)}{\cB(w)}\frac{ \sfb'(z_1)}{( w -\sfb(z_1))^2} \frac{ \sfb'(z_2)}{( w -\sfb(z_2))^2} \rd w,
\end{split}\end{align*}
where the contour $\cin\subset \Lambda$ encloses $[\bl, \br]$, but not $z_1, z_2$. The higher order joint moments of $\{\varepsilon^{-1/2}\Delta \cM(z)\}_{z\in \Lambda\setminus[\bl,\br]}$ converge as $\eps\to 0$ to the Gaussian joint moments.

 The implicit constants in the $\OO(\eps)$ are uniform in all the involved parameters, as long as the constants $C$ and $c$ of Assumptions  \ref{a:asymp} and \ref{a:stable} are fixed and $z$ belongs to a compact subset of $\Lambda\setminus [\bl,\br]$.
\end{theorem}

\section{Proofs for Results from \Cref{s:setup0} and \Cref{s:heightslope}}\label{s:setupproof}

\begin{proof}[Proof of \Cref{e:constructH}]
In this proof, for simplicity of notations, we omit the dependence on $N$ and simply write $\bm\sfy=\bm\sfy^{(N)}$ and $\bm\sfz= \bm\sfz^{(N)}$.
By our assumption $\cP(\bm\sfy, \bm\sfz;\sfT)\neq\emptyset$, there exists a non-intersecting Bernoulli paths $\{\bmy(t)\}_{0\leq t\leq T}$ from $\bmy(0)=\bm\sfy/N$ to $\bmy(T)=\bm\sfz/N$.

Next we construct a non-intersecting Bernoulli paths $\{\bmz(t)\}_{0\leq t\leq T}$, such that its height function approximates $H^*$. First we can construct the level line of $H^*$ as
\begin{align*}
 \gamma(y,t)=\inf \{x: H^*(x, t)>y\}, \quad 0\leq y<\theta.
\end{align*}
Since $H^*$ is $2$-Lipschitz we have $H^*(\gamma(y,t),t)=y$ for any $0\leq y< \theta$. 
Since $\nabla H^*\in \overline{\cT}$, for any $s\geq t$, 
\begin{align*}
H^*(\gamma(y,t),s)\leq H^*(\gamma(y,t),t)\leq H^*(\gamma(y,t)+(s-t),s).
\end{align*}
We conclude that $\gamma(y,t)\leq \gamma(y,s)\leq \gamma(y,t)+(s-t)$. It follows that $\gamma(y,t)$ is $1$-Lipschitz in $t$, and $\del_t\gamma(y,t)\in [0,1]$.

We take $m=\lceil\varepsilon N/(2\theta)\rceil$, and construct the nonintersecting Bernoulli paths $\{\bmx(t)\}_{0\leq t\leq T}$ in the following way: the path $\{x_i(t)\}_{0\leq t\leq T}$ is obtained from  $\{y_i(t)\}_{0\leq t\leq T}$ by truncating from left using $\gamma(\theta(i-m)/N,t)$, and from right using  $\gamma(\theta(i+m)/N,t)$. Here we make the convention $\gamma(y,s)=-\infty$ if $y<0$, and $\gamma(y,s)=+\infty$ if $y>1$. Formally, $\{\bmx(t)\}_{0\leq t\leq T}$ is given as for any $1\leq i\leq N$ and $ t\in N^{-1}\qq{0,\sfT}$
\begin{align}\label{e:constructx}
x_i(t)=\argmin \{|x-y_i(t)|: x-y_i(t)\in N^{-1}\bZ, \gamma(\theta (i-m)/N,t)\leq x \leq \gamma(\theta (i+m)/N,t)\}.
\end{align}
We denote the height function of $\{\bmx(t)\}_{0\leq t\leq T}$ as $\cH$.

Next we check that $\{\bmx(t)\}_{0\leq t\leq T}$ forms a non-intersecting $\theta$-Bernoulli walk from $\sfy^{(N)}/N$ to $\sfz^{(N)}/N$. Under assumption \eqref{e:rhoconverge}, for $N$ large enough 
\begin{align*}
\sup_{x\in \bR}|\cH(x,0)-h(x,0)|, \sup_{x\in \bR}|\cH(x,T)-h(x,T)|\leq \varepsilon/(4\theta).
\end{align*}
It follows that for $t\in \{0,T\}$, 
\begin{align*}
 &h(y_i(t),t)\geq \cH(y_i(t),t)-\varepsilon/4=\frac{\theta(i-1)}{N}-\varepsilon/4\geq h(\gamma(\theta(i-m)/N,t),\\
  &h(y_i(t),t)\leq \cH(y_i(t),t)+\varepsilon/4=\frac{\theta(i-1)}{N}+\varepsilon/4\leq h(\gamma(\theta(i+m)/N,t),
\end{align*}
And thus $\gamma(\theta(i-m)/N,t)\leq y_i(t)=x_i(t)\leq \gamma(\theta(i+m)/N,t)$ for $t\in \{0, T\}$.

 We need to show $x_i(t+1/N)-x_i(\sft/N)\in \{0,1\}$ and $x_{i+1}(t)-x_i(t)\geq \theta/N$ for any $t=\sft/N\in N^{-1}\qq{0, \sfT}$. By symmetry, we assume $y_i(t)\geq \gamma_i(t)$, so it is far from $\gamma_{i-m}(t)$. In fact we have 
\begin{align}\label{e:asump}
y_i(t), y_{i+1}(t), y_i(t+1/N), y_{i+1}(t+1/N)\geq \gamma_{i-m}(t)+\varepsilon/2.
\end{align}

Notice that under \eqref{e:asump} we only need to deal with the possible truncation caused by $\gamma_{i+m}(t)$. There are still several cases
\begin{enumerate}
\item If $\gamma(\theta (i-m)/N,t)\leq y_i(t) \leq \gamma(\theta (i+m)/N,t)$, then $x_{i}(t)=y_i(t)$. If also $x_i(t+1/N)=y_i(t+1/N)$, then $x_i(t+1/N)-x_i(t)=y_i(t+1/N)-y_i(t)\in \{0,1/N\}$; otherwise $y_i(t+1/N)>\gamma(\theta(i+m)/N,t+1/N)$. Recall that $\del_s \gamma(y,s)\in [0,1]$, we must have 
\begin{align*}
y_i(t+1/N)>\gamma(\theta(i+m)/N,t+1/N)\geq \gamma(\theta(i+m)/N,t)\geq y_i(t)\Rightarrow y_i(t+1/N)=y_i(t)+1/N.
\end{align*}
In this case $x_i(t+1/N)=y_i(t)$, and $x_i(t+1/N)-x_i(t)=0$. 

If $x_{i+1}(t)=y_{i+1}(t)$, then $x_{i+1}(t)-x_i(t)=y_{i+1}(t)-y_i(t)\geq \theta/N$; otherwise $y_{i+1}(t)>\gamma(\theta(i+m+1)/N,t)$. Recall that $ \gamma(y+\theta/N,s)-\gamma(y,s)\geq \theta/N$, we must have 
\begin{align*}
y_{i+1}(t)>\gamma(\theta(1+i+m)/N,t+1/N)\geq \gamma(\theta(i+m)/N,t)+\theta/N\geq y_i(t)+\theta/N.
\end{align*}
In this case $x_{i+1}(t)\geq y_i(t)+\theta/N=x_i(t)+\theta/N$.

\item If $y_i(t)>\gamma(\theta (i+m)/N,t)$, then $x_i(t)= y_i(t)-\bZ_{\geq 0}$. For $y_i(t+1/N)$ there are two cases: If $y_i(t+1/N)\leq \gamma(\theta (i+m)/N,t+1/N)$, recalling that $\del_s \gamma(y,s)\in [0,1]$, then we must have
\begin{align*}
y_i(t+1/N)\leq \gamma(\theta (i+m)/N,t+1/N)\leq \gamma(\theta (i+m)/N,t)+1/N<y_i(t)+1/N.
\end{align*}
In this case $y_i(t)=y_i(t+1/N)=x_i(t+1/N)$, and it follows 
\begin{align*}
y_i(t)-1/N=y_i(t+1/N)-1/N\leq \gamma(\theta (i+m)/N,t+1/N)-1/N\leq \gamma(\theta (i+m)/N,t).
\end{align*}
So $x_i(t)=y_i(t)-1/N$ and $x_i(t+1/N)-x_i(t)=1/N$.

Otherwise if $y_i(t+1/N)\geq \gamma(\theta (i+m)/N,t+1/N)$, recall $\del_t\gamma(y,t)\in [0,1]$, we have
\begin{align*}
x_i(t)\leq x_i(t+1/N)\leq \gamma(\theta (i+m)/N,t+1/N)\leq \gamma(\theta (i+m)/N,t)+1/N<x_i(t)+1/N+1/N.
\end{align*}
In this case we still have $x_i(t+1/N)-x_i(t)\in \{0,1/N\}$.

For $y_{i+1}(t)$ there are two cases: If $y_{i+1}(t)\leq \gamma(\theta (1+i+m)/N,t+1/N)$, then $x_{i+1}(t)=y_{i+1}(t)$ and $x_{i+1}(t)-x_i(t)\geq y_{i+1}(t)-y_i(t)\geq \theta/N$.
Otherwise if $y_{i+1}(t)> \gamma(\theta (1+i+m)/N,t+1/N)$, recall that $ \gamma(y+\theta/N,s)-\gamma(y,s)\geq \theta/N$, so 
\begin{align*}
x_i(t)+\theta/N\leq \gamma(\theta(i+m)/N,t)+\theta/N\leq \gamma(\theta(1+i+m)/N,t),
\end{align*}
and $x_{i+1}(t)\geq x_i(t)+\theta/N$.

\end{enumerate}

Finally we check that the height function $\cH$ of $\{\bmx(t)\}_{0\leq t\leq T}$ as constructed in \eqref{e:constructx} satisfies \eqref{e:heightcloseH}. For any $x_i(t)\leq x<x_{i+1}(t)$, then $\gamma(\theta(i-m)/N)\leq x< \gamma(\theta(1+i+m)/N)$
\begin{align}\label{e:Hbound1}
\left|\cH(x,t)-\frac{\theta i}{N}\right|\leq \frac{\theta}{N}.
\end{align}
and 
\begin{align}\label{e:Hbound2}
  \max\left\{0, \frac{\theta(i-m)}{N}\right\}\leq H^*(x,t)\leq \min\left\{\theta, \frac{\theta(1+i+m)}{N}\right\},
\end{align}
The claim \eqref{e:heightcloseH} follows from \eqref{e:Hbound1} and \eqref{e:Hbound2}.
\end{proof}

\begin{proof}[Proof of \Cref{c:nabHbound}]
The statement \eqref{e:dtH} follows from \eqref{e:deftH} and  $\nabla \cA(x,t)=(\varrho, -\varrho v)\bm1(tv\leq x\leq 1+tv)$. 
On the interval $[tv, \ell+tv]$, we have
\begin{align*}
\kappa_t(x)
=  \int_{tv}^{\ell+tv} \frac{\delta  \rd y}{(y-x)^2+\delta^2}\asymp 1,
\end{align*}
and 
\begin{align*}
1-\kappa_t(x)
= \int_{[tv, \ell+tv]^\complement}\frac{ \delta \rd y}{(y-x)^2+\delta^2}
\asymp \frac{\delta}{\delta+\dist(x, \{tv, \ell+tv\})}.
\end{align*}

Outside the interval $[tv, \ell+tv]$, we have
\begin{align*}
\kappa_t(x)
=  \int_{tv}^{\ell+tv} \frac{\delta  \rd y}{(y-x)^2+\delta^2}\lesssim \frac{\delta}{\delta+\dist(x, [tv, \ell+tv])+\dist(x, [tv, \ell+tv])^2/\ell},
\end{align*}
and 
\begin{align*}
1-\kappa_t(x)
= \int_{[tv, \ell+tv]^\complement}\frac{ \delta \rd y}{(y-x)^2+\delta^2}
\asymp 1.
\end{align*}
\end{proof}

%
%

\begin{proof}[Proof of \Cref{l:kappaHib}]
From the expression  \eqref{e:defkappat}, 
\begin{align}\begin{split}\label{e:kappat}
\kappa_t(x) =-\Im \int_{tv}^{\ell+tv} \frac{\rd y}{(x+\ri\delta)-y}
&=\frac{\ri}{2}\left( \int_{tv}^{\ell+tv} \frac{\rd y}{(x+\ri\delta)-y}-\int_{tv}^{\ell+tv} \frac{\rd y}{(x-\ri\delta)-y}\right)\\
&=\frac{\ri}{2}\left(\ln \frac{(x+\ri \delta)-(\ell+tv)}{(x+\ri \delta)-tv}-\ln\frac{(x-\ri \delta)-(\ell+tv)}{(x-\ri \delta)-tv}\right),
\end{split}\end{align}
where $\ln(\cdot)$ is the branch on $\bC\setminus[-\infty, 0]$.
And 
\begin{align}\begin{split}\label{e:Hibkappa}
\Hib(\kappa_t)(x)=\Re \int_{tv}^{\ell+tv} \frac{\rd y}{(x+\ri\delta)-y}
&=\frac{1}{2}\left( \int_{tv}^{\ell+tv} \frac{\rd y}{(x+\ri\delta)-y}+\int_{tv}^{\ell+tv} \frac{\rd y}{(x-\ri\delta)-y}\right)\\
&=\frac{1}{2}\left(\ln \frac{(x+\ri \delta)-(\ell+tv)}{(x+\ri \delta)-tv}+\ln\frac{(x-\ri \delta)-(\ell+tv)}{(x-\ri \delta)-tv}\right).
\end{split}\end{align}
The expression \eqref{e:kappat} is well defined for $|\Im[z]|<\delta$. Thus $\kappa_t(x)$ extends analytically to $\kappa_t(z)$ for $|\Im[z]|<\delta$. Moreover, 
\begin{align}\begin{split}\label{e:dzkappa}
\del_z \kappa_t(z)
&=\frac{\ri}{2}\del_z \left(\ln \frac{(z+\ri \delta)-(\ell+tv)}{(z+\ri \delta)-tv}-\ln\frac{(z-\ri \delta)-(\ell+tv)}{(z-\ri \delta)-tv}\right)\\
&=\frac{\ri}{2} \left(\frac{1}{(z+\ri \delta)-(\ell+tv)}- \frac{1}{(z+\ri \delta)-tv}-\frac{1}{(z-\ri \delta)-(\ell+tv)}+\frac{1}{(z-\ri \delta)-tv}\right)\\
&=\frac{\delta}{(z-\ell-tv)^2+\delta^2}-\frac{\delta}{(z-tv)^2+\delta^2}.
\end{split}\end{align}
If $\dist(z, \{tv, \ell+tv\})$, the first statement in the claim \eqref{e:largezbound} follows from \eqref{e:dzkappa}.

For $z=x+\ri\eta$ with $|\Im[z]|\leq \delta/2$
\begin{align}\label{e:term1}
\left|\frac{\delta}{(z-tv)^2+\delta^2}\right|\leq \frac{\delta}{(x-tv)^2+\delta^2-\Im[z]^2}\lesssim \frac{\delta}{\delta^2+(x-tv)^2},
\end{align}
and similarly
\begin{align}\label{e:term2}
\left|\frac{\delta}{(z-\ell-tv)^2+\delta^2}\right|\lesssim \frac{\delta}{\delta^2+(x-\ell-tv)^2},
\end{align}
We conclude that by plugging \eqref{e:term1} and \eqref{e:term2} into \eqref{e:dzkappa},
\begin{align}\label{e:dzkz}
|\del_z \kappa_t(z)|\lesssim \frac{\delta}{\delta^2+\dist(x, \{tv,\ell+tv\})^2}\lesssim \frac{ \min\{\kappa_t(x), 1-\kappa_t(x)\}}{\delta},
\end{align}
where the last inequality is from \Cref{c:nabHbound}. The claim \eqref{e:ktdiff} follows from integrating \eqref{e:dzkz} from $x$ to $x+\ri \eta$. 

The expression of the Hilbert transform \eqref{e:Hibkappa} also extends analytically to $\Hib(\kappa_t)(z)$ for $|\Im[z]|<\delta$,
\begin{align}\begin{split}\label{e:Hibkappa}
\Hib(\kappa_t)(z)=\frac{1}{2}\left(\ln \frac{(z+\ri \delta)-(\ell+tv)}{(z+\ri \delta)-tv}+\ln\frac{(z-\ri \delta)-(\ell+tv)}{(z-\ri \delta)-tv}\right).
\end{split}\end{align}
For $z=x+\ri \eta$ with $|\eta|\leq \delta/2$, we have
\begin{align*}
\left|\ln \frac{(z+\ri \delta)-(\ell+tv)}{(z+\ri \delta)-tv}\right|
\leq \ln \left| 1+\frac{\ell}{(z+\ri \delta)-tv}\right|+\pi \leq \ln (\ell/\delta)+\OO(1),
\end{align*}
and we have the same estimate for the second term on the righthand side of \eqref{e:Hibkappa}. This gives the first statement of \eqref{e:Hkappa}. To prove the second statement of \eqref{e:Hkappa}, we can rewrite \eqref{e:Hibkappa} as, 
\begin{align}\begin{split}\label{e:Hibkappa2}
\Hib(\kappa_t)(z)=\frac{1}{2}\left(\ln ((z-\ell-tv)^2+\delta^2)-\ln ((z-tv)^2+\delta^2)\right).
\end{split}
\end{align}
If $\dist(z, \{tv,\ell+tv\})\gtrsim \ell$, from \eqref{e:Hibkappa}, we have
\begin{align*}
\Hib(\kappa_t(z)\lesssim \frac{\delta}{\dist(z, \{tv, \ell+tv\})},
\end{align*}
this gives the second statement in \eqref{e:largezbound}.

For $z=x+\ri \eta$ with $|\eta|\leq \delta/2$, we have
\begin{align*}
&\phantom{{}={}}|\Im \ln ((z-\ell-tv)^2+\delta^2)|
=|\Im \ln (\ri 2\eta (x-\ell-tv)+(x-\ell-tv)^2+(\delta^2-\eta^2))|\\
&=\left|\Im[\ln (x-\ell-tv)^2+(\delta^2-\eta^2)]+\Im \ln \left(1+\frac{\ri 2\eta (x-\ell-tv)}{(x-\ell-tv)^2+(\delta^2-\eta^2)}\right)\right|
\lesssim \frac{|\eta|}{\sqrt{\delta^2-\eta^2}}\lesssim \frac{\eta}{\delta},
\end{align*}
where in the last two inequalities we used that $(x-\ell-tv)^2+(\delta^2-\eta^2)\geq 2|x-\ell-tv|\sqrt{\delta^2-\eta^2}$. We have the same estimate for the second term on the righthand side of \eqref{e:Hibkappa2}. This gives the second statement of \eqref{e:Hkappa}.
\end{proof}

\begin{proof}[Proof of \Cref{l:gtestimate}]
We recall $\kappa_t(x)$ from \eqref{e:defkappat} and its properties from \Cref{c:nabHbound} and \Cref{l:kappaHib}. We will first show the following term (from \eqref{e:ftgtx})
\begin{align*}
\ln \sin(-\pi \partial_t  \widetilde H(x,t))- \ln \sin(\pi \partial_x  \widetilde H(x,t)+\pi\partial_t  \widetilde H(x,t)),
\end{align*}
can be extended to $z\in \cD$ as
\begin{align}\label{e:logdiff}
\ln \sin(\pi \varrho v \kappa_t(z))- \ln \sin(\pi \varrho(1-v)\kappa_t(z))
=\ln \frac{\sin(\pi \varrho v \kappa_t(z))}{\sin(\pi \varrho(1-v)\kappa_t(z))}\,.
\end{align}
Recall that $\kappa_t(x)\in [0,1]$ (from \Cref{c:nabHbound}) and $\zeta\leq \varrho v\leq 1-\zeta$, we have that 
\begin{align*}
\sin(\pi \varrho v \kappa_t(x))\asymp \min \{\varrho v \kappa_t(x), 1-\varrho v \kappa_t(x)\}.
\end{align*}
Thanks to \eqref{e:ktdiff}, we have that 
\begin{align}\begin{split}\label{e:sindiff1}
|\sin(\pi \varrho v \kappa_t(z))-\sin(\pi \varrho v \kappa_t(x))| 
&\lesssim \pi \varrho v |\kappa_t(z)-\kappa_t(x)|\\
&\lesssim   \frac{|\Im[z]|}{\delta} \min\{\varrho v \kappa_t(x), \varrho v(1-\kappa_t(x))\}\\
&\lesssim \frac{|\Im[z]|}{\delta} \sin(\pi \varrho v \kappa_t(x))
\leq \frac{1}{2}\sin(\pi \varrho v \kappa_t(x)),
\end{split}\end{align}
provided that $\delta\geq C\Im[z]$ for $C>1$ large enough. Similarly, for the second term on the left of \eqref{e:logdiff}, we have
\begin{align}\begin{split}\label{e:sindiff2}
|\sin(\pi \varrho(1-v)\kappa_t(z))-\sin(\pi \varrho(1-v)\kappa_t(x))|
&\lesssim \frac{|\Im[z]|}{\delta} \sin(\pi \varrho(1-v)\kappa_t(x))\\
&\leq \frac{1}{2}\sin(\pi \varrho(1-v)\kappa_t(x))\,.
\end{split}\end{align}
The two estimates \eqref{e:sindiff1} and \eqref{e:sindiff2} together gives that 
\begin{align}\label{e:t1}
\left|\frac{\sin(\pi \varrho v \kappa_t(z))}{\sin(\pi \varrho(1-v)\kappa_t(z))}\right|
\asymp \left|\frac{\sin(\pi \varrho v \kappa_t(x))}{\sin(\pi \varrho(1-v)\kappa_t(x))}\right|.
\end{align}
If $\kappa_t(x)\leq 1/2$, then 
\begin{align}\label{e:t2}
 \left|\frac{\sin(\pi \varrho v \kappa_t(x))}{\sin(\pi \varrho(1-v)\kappa_t(x))}\right|
 \asymp \frac{\varrho v}{\varrho(1-v)}\in \left[\frac{\zeta}{1-\zeta}, \frac{1-\zeta}{\zeta}\right].
\end{align}
If $\kappa_t(x)\geq 1/2$, then using $\zeta\leq \varrho v\leq 1-\zeta$,
\begin{align}\begin{split}\label{e:t3}
\zeta  
&\lesssim \min\{\varrho v\kappa_t(x),1-\varrho v\kappa_t(x) \}\lesssim\left|\frac{\sin(\pi \varrho v \kappa_t(x))}{\sin(\pi \varrho(1-v)\kappa_t(x))}\right|\\
& \lesssim \frac{1}{\min\{\varrho(1-v)\kappa_t(x),1-\varrho(1-v)\kappa_t(x) \}}\lesssim \frac{1}{\zeta}.
\end{split}\end{align}

We conclude from plugging \eqref{e:t1}, \eqref{e:t2} and \eqref{e:t3} into \eqref{e:logdiff}
\begin{align}\label{e:logdiff2}
|\ln \sin(\pi \varrho v \kappa_t(z))- \ln \sin(\pi \varrho(1-v)\kappa_t(z))|
\leq \ln(1/\zeta)+C.
\end{align}
The above estimate \eqref{e:logdiff2} and the first statement of \Cref{l:kappaHib} together imply that we can extend $f_t(z)$ to the strip region, and 
\begin{align*}
f_t(z)=e^{-\ri\pi \varrho \kappa_t(z)}\frac{\sin(\pi \varrho v \kappa_t(z))}{\sin(\pi \varrho(1-v)\kappa_t(z))},\quad |f_t(z)|\leq \frac{C}{\zeta}.
\end{align*}
The estimate \eqref{e:logdiff2} and the second statement of \Cref{l:kappaHib} together imply that we can extend $g_t(z)$ to the strip region, and 
\begin{align*}
g_t(z)=\ln \frac{\sin(\pi \varrho v \kappa_t(z))}{\sin(\pi \varrho(1-v)\kappa_t(z))}-\Hib(\kappa_t)(z),\quad  |g_t(z)|\leq \ln (1/\zeta)+\ln(\ell/\delta)+C. 
\end{align*}
For $\dist(z, \{tv, \ell+tv\})\geq\ell$, it follows from \eqref{e:ktdiff} and \eqref{e:ktxbound} in \Cref{c:nabHbound}
\begin{align*}
|\kappa_t(z)|\lesssim \frac{C\delta}{\dist(z, \{tv, \ell+tv\})+\dist(z, \{tv, \ell+tv\})^2/\ell}.
\end{align*}
Then by Taylor expansion 
\begin{align*}
\ln \frac{\sin(\pi \varrho v \kappa_t(z))}{ \sin(\pi \varrho(1-v)\kappa_t(z))}
&=\ln\frac{\sin( v)}{ \sin(1-v)}
+\ln \frac{\sin(\pi \varrho v \kappa_t(z))}{\pi \varrho v \kappa_t(z)}
+\ln \frac{\pi \varrho(1-v)\kappa_t(z)}{ \sin(\pi \varrho(1-v)\kappa_t(z))}\\
&=\ln\frac{\sin( v)}{ \sin(1-v)}+\OO\left(\frac{\delta}{\dist(z, \{tv, \ell+tv\})}\right).
\end{align*}
This gives the claim \eqref{e:glargez}. 
Next, we estimate $\Im[g_t(z)]$ for $z=x+\ri \eta$. By taking imaginary part of \eqref{e:logdiff}, and using \eqref{e:sindiff1} and \eqref{e:sindiff2}
\begin{align*}
\Im \ln \frac{\sin(\pi \varrho v \kappa_t(z))}{\sin(\pi \varrho(1-v)\kappa_t(z))}
&=\Im \ln \frac{\sin(\pi \varrho v \kappa_t(z))}{\sin(\pi \varrho v \kappa_t(x))}-\Im \ln\frac{\sin(\pi \varrho(1-v)\kappa_t(z))}{\sin(\pi \varrho(1-v)\kappa_t(x))}\\
&=\arg \left(1+\OO\left(\frac{|\Im[z]|}{\delta} \right)\right)
-\arg \left(1+\OO\left(\frac{|\Im[z]|}{\delta} \right)\right)=\OO\left(\frac{|\Im[z]|}{\delta} \right).
\end{align*}
The above estimate together with the second statement \eqref{e:Hkappa} of \Cref{l:kappaHib}  imply that  
\begin{align*}
|\Im g_t(z)|\leq \frac{C\Im[z]}{\delta}. 
\end{align*}
$\Hib(\kappa_t)(z)$ can be computed explicitely, see   \eqref{e:Hibkappa2}, and its derivative at $z=x$ can be bounded as
 \begin{align}\label{e:der2}
\left|\del_x \Hib(\kappa_t)(x)\right|=\left|\frac{x-\ell-tv}{(x-\ell-tv)^2+\delta^2}-\frac{x-tv}{(x-tv)^2+\delta^2}\right|\lesssim \frac{1}{\dist(x, \{tv, \ell+tv\})+\delta}.
\end{align}
Since $g_t(z)$ is analytic, we have $|\del_x g_t(x)|=|\del_\eta g_t(x+\ri\eta)|_{\eta=0}|$. As a consequence of \eqref{e:sindiff1} and \eqref{e:sindiff2}, we have
\begin{align}\label{e:der1}
\left|\del_x \ln \frac{\sin(\pi \varrho v \kappa_t(x))}{\sin(\pi \varrho(1-v)\kappa_t(x))}\right|\lesssim \frac{1}{\delta}.
\end{align}
The claim $|\del_x g_t(x)|\lesssim 1/\delta$ follows from \eqref{e:der1} and \eqref{e:der2}.

Recall our construction of $g_t$ from \eqref{e:ftgtx}, since $\widetilde H(x-tv,t)$ is constant for $0\leq t\leq \ell$, so is
 $g_t(x-tv)$. By taking derivative with respect to $t$ we get
\begin{align*}
\del_t g_t-v\del_x g_t=0.
\end{align*}
Thus $|\del_t g_t|\leq v|\del_x g_t|\leq C/\delta$. 
\end{proof}

\section{Asymptotics for Macdonald Polynomials}\label{s:Macdonald}

In this section, we explain the correspondence between  non-intersecting $\theta$-Bernoulli walk ensembles and certain Macdonald ascending process \eqref{e:Mprocess}. This correspondence will be used to derive large deviation asymptotics for (skew) Macdonald polynomials. We collect some basic properties of Macdonald symmetric functions  in \Cref{s:Macdonald}. We recall  the Macdonald ascending process in \Cref{s:Macdonaldprocess}. Our main references are \cite{MR3443860} and \cite{MR1014073}. In \Cref{s:proofmain3}, we prove \Cref{t:main3}.

\subsection{Macdonald polynomials and specializations}
We use the following notations:
\begin{align}\label{e:bracketf}
(a;q)_\infty=\prod_{i=1}^\infty(1-aq^{i-1}),\quad f(u)=\frac{(tu;q)_\infty}{(qu;q)_{\infty}},\quad
\Gamma_q(x)=(1-q)^{1-x}\frac{(q;q)_\infty}{(q^x;q)_\infty}.
\end{align}

The next class of Markov chains is built out of the specializations of Macdonald polynomials. We refer to \cite[Section VI]{macdonald1998symmetric} and \cite{borodin2014macdonald} for definitions and properties of Macdonald symmetric functions and use them as a black box in this section. Macdonald symmetric functions $P$ and $Q$ are indexed by
partitions and implicitly depend on two parameters $q,t\in(0,1)$. The coefficients for the symmetric functions are in $\bQ[q,t]$. Macdonald symmetric functions $P_\bmla(\bmx;q,t)$ and $Q_\bmla(\bmx;q,t)$ are elements of the algebra $\Lambda$ of the symmetric functions in infinitely many variables $(x_i)_{i=1}^\infty$ uniquely determined by the following two properties:
\begin{enumerate}
\item $P_\bmla$, $|\bm\lambda|=m$, can be expressed in terms of the monomial symmetric functions via a strictly upper unitriangular transition matrix:
$$
P_\bmla=m_\bmla+\sum_{\bmmu<\bmla\in \bY_m}R_{\bmla\bmmu}P_{\bmmu},
$$
where $R_{\bmla \bmmu}$ are functions of $q,t$ and $\bmmu<\bmla$ is comparison in the dominance order on the set $\bY_m$ of all partitions of $m$ (equivalently, Young diagrams with $m$ boxes).
\item They are pairwise orthogonal with respect to the scalar product defined on the power sums via
$$
\langle p_\bmla,p_\bmmu\rangle_{q,t}=\delta_{\bmla\bmmu}z_\bmla(q,t), \qquad p_\lambda=\prod_{i=1}^{\infty} p_{\lambda_i}, \qquad z_\bmla(q,t)=\prod_{i\geq 1}i^{m_i}(m_i)!\prod_{i=1}^{\ell(\bmla)}\frac{1-q^{\la_i}}{1-t^{\la_i}},
$$
where $\bmla=1^{m_1}2^{m_2}\dots$, i.e.\ $m_i$ is the multiplicity of $i$ in $\bmla$, $\ell(\bmla)$ is the number of rows, and $p_k=(x_1)^k+(x_2)^k+\dots$, $k\geq 1$.
\end{enumerate}
We further define $Q_\bmla=\frac{P_\bmla}{\langle P_\bmla, P_\bmla\rangle_{q,t}}.$ Finally, the skew Macdonald polynomials $P_{\bmla/\bmmu}$ and $Q_{\bmla/\bmmu}$ are defined through the expansions:
\begin{align*}
 P_\bmla(x_1,x_2,\dots, y_1,y_2,\dots; q,t)&=\sum_{\bmmu} P_{\bmla/\bmmu}(x_1,x_1,\dots;q,t) P_\bmmu(y_1,y_2,\dots; q,t),\\
 Q_\bmla(x_1,x_2,\dots, y_1,y_2,\dots; q,t)&=\sum_{\bmmu} Q_{\bmla/\bmmu}(x_1,x_1,\dots;q,t) Q_\bmmu(y_1,y_2,\dots; q,t).
\end{align*}

\subsection{Macdonald ascending processes}\label{s:Macdonaldprocess}
Computations in the algebra of symmetric functions $\Lambda$ can be converted into numeric identities by means of \emph{specializations}, which are algebra homomorphism from $\Lambda$ to the set of complex numbers. Specialization $\rho$ is uniquely determined by its values on any set of algebraic generators of $\Lambda$ and we use $(p_k)_{k=1}^{\infty}$ as such generators. The value of $\rho$ on a symmetric function $f$ is denoted $f(\rho)$.
Given two specializations $\rho,\rho'$, we define their union $(\rho,\rho')$ through the formula:
$$
p_k(\rho,\rho')=p_k(\rho)+p_k(\rho'),\quad k\geq 1.
$$
A specialization $\rho$ is called \emph{Macdonald nonnegative} if its values on all (skew) Macdonald symmetric functions are non-negative, i.e., if for all partitions $\bm\lambda$ and $\bm\mu$,
\begin{align*}
P_{\bmla/\bmmu}(\rho;q,t)\geq 0,
\end{align*}
The description of Macdonald nonnegative specializations was conjectured in \cite{kerov2003asymptotic} and proven in \cite{matveev2019macdonald}:
\begin{theorem}[\cite{matveev2019macdonald}] \label{Theorem_Macdonald_positive}
For any fixed $q,t\in(0,1)$, Macdonald nonnegative specializations can be parameterized by triplets $(\bm\alpha=\{\al_i\}_{i\geq 1}, \bm\beta=\{\beta_i\}_{i\geq 1},\gamma)$ of nonnegative numbers satisfying $\sum_{i=1}^\infty(\al_i+\beta_i)<\infty$. The specialization $\rho$ corresponding to a triplet $(\bm\alpha, \bm \beta, \gamma)$ is defined by for $k\geq 2$
$$
 p_1(\rho)=\sum_{i=1}^\infty \alpha_i+ \frac{1-q}{1-t}\left(\gamma+ \sum_{i=1}^{\infty}\beta_i\right) , \qquad p_k(\rho)=\sum_{i=1}^{\infty} (\alpha_i)^k + (-1)^{k-1}\frac{1-q^k}{1-t^k} \sum_{i=1}^{\infty} (\beta_i)^k.
$$
\end{theorem}

Following \cite[Section 2.3]{borodin2014macdonald}, we create Markov chains out of the Macdonald-positive specializations:
\begin{definition} \label{Definition_Macdonald_ascending}
 Given two specializations $\rho$ and $\rho'$ we define the ascending transition through
\begin{equation}
\label{eq_ascending_transition}
 \bP(\bmla\mid \bmmu)= \frac{1}{\Pi(\rho;\rho')} \frac{P_\bmla(\rho; q,t)}{P_\bmmu(\rho; q,t)} Q_{\bmla/\bmmu}(\rho'; q,t),
\end{equation}
where $\bmla$ and $\bmmu$ are partitions with $\bmmu\subset\bmla$ and $\Pi(\rho;\rho')$ is the result of applying $\rho$ to the $x_i$ variables and $\rho'$ to the $y_j$ variables in the infinite product
$$
 \Pi=\prod_{i,j\geq 1} \frac{(tx_i y_j;q)_\infty}{(x_i y_j;q)_\infty}.
$$
\end{definition}

To obtain the asymptotics of Skew Macdonald polynomials, we study the following special Macdonald ascending process. We fix $N, \sfT\geq 1$, and a sequence of positive numbers $b_0, b_1, b_2, \cdots, b_{\sfT-1}$. The transition probability at time $0\leq t\leq \sfT-1$ is given in the notations of Theorem \ref{Theorem_Macdonald_positive} as
\begin{equation} \label{eq_principal_ascending}
 \rho:\, \alpha_i=t^{i-1},\, 1\leq i \leq N;\qquad \rho_\sft: \beta_1=b_\sft,
\end{equation}
with all other parameters set to $0$. This gives a  Markov process on Young diagrams
$\bmla(0), \bmla(1),\cdots, \bmla(\sfT)$
\begin{align}\label{e:Mprocess}
    \bP(\bmla(\sft+1)=\bmla|\bmla(\sft)=\bmmu)
    =\frac{1}{\Pi(\rho, \beta_1=b_t)}\frac{P_{\bmla}(\rho;q,t)}{P_{\bmmu}(\rho;q,t)}Q_{\bmla/\bmmu}(\beta_1=b_{t};q,t).
\end{align}
Then the transition probability gives the skew Macdonald polynomials
\begin{align}\label{e:MDdynamics1}
    \bP(\bmla(\sfT)=\bmla|\bmla(0)=\bmmu)
    &=\frac{1}{\Pi((\rho, \bm\beta=(b_0, \cdots, b_{\sfT-1}))}\frac{P_{\bmla}(\rho;q,t)}{P_{\bmmu}(\rho;q,t)}Q_{\bmla/\bmmu}(\bm\beta=(b_0, \cdots, b_{\sfT-1});q,t).
\end{align}

The evaluation of the Macdonald polynomial under the principle specialization $\rho=(1,t,\dots, t^{N-1})$ is explicit, see \cite[Chapter VI, (6.11)]{macdonald1998symmetric}:
\begin{align}\label{e:MDP}
P_\bmla(1,t,\dots, t^{N-1};q,t)=t^{\sum_{i=1}^{N}(i-1)\la_i}\prod_{i<j\leq N}\frac{(q^{\la_i-\la_j}t^{j-i};q)_\infty}{(q^{\la_i-\la_j}t^{j-i+1};q)_\infty}\frac{(t^{j-i+1};q)_\infty}{(t^{j-i};q)_\infty},
\end{align}
for $\bmla=\lambda_1\geq \lambda_2\dots\geq\lambda_N\geq 0$; $P_\bmla(\rho;q,t)=0$  if $\lambda_{N+1}>0$. We further use the $w_{u,v}$ automorphism of the algebra of symmetric functions $\Lambda$ defined on the power sums by:
$$
 w_{u,v}(p_k)=(-1)^{k-1} \frac{1-u^k}{1-v^k} p_k.
$$
As shown in \cite[Chapter VI, Section 7]{macdonald1998symmetric},
$$
 w_{t,q}\bigl(P_{\bmla'/\bmmu'}(\cdot; t,q)\bigr)= Q_{\bmla/\bmmu}(\cdot ; q,t),
$$
where $\bmla'$ and $\bmmu'$ are transposed Young diagrams $\bmla$ and $\bmmu$, respectively. Hence, with the specialization $\bm\beta=(b_0, b_1, \cdots, b_{\sfT-1})$, as in \eqref{eq_principal_ascending}, we have
$$
 Q_{\bmla/\bmmu}(\bm\beta=(b_0, b_1, \cdots, b_{\sfT-1});q,t)= P_{\bmla'/\bmmu'}(\bm\alpha=(b_0, b_1, \cdots, b_{\sfT-1}); t,q).
$$

The following claim states that we can encode the Macdonald ascending process \ref{e:Mprocess} of Young diagrams as an $N$-particle non-intersecting $\theta$-Bernoulli random walk.

\begin{claim}[{\cite[Corollary 3.11]{gorin2022dynamical}}]\label{c:mdensity2}
The transition probability of Definition \ref{Definition_Macdonald_ascending} under the specializations \eqref{eq_principal_ascending} is non-degenerate only for partitions $\bmla,\bmmu$ with at most $n$ parts, i.e.\ $\bmla=(\lambda_1\geq \lambda_2\geq\dots\geq\lambda_n)$ and $\bmmu=(\mu_1\geq \mu_2\geq\dots\geq\mu_n)$. Further, if we  set $t=q^{\theta}$ and identify
\begin{equation}\label{e:x1}
\bmx=(\sfx_1,\sfx_2,\dots, \sfx_{N})\in \bW_\theta^{N},\quad \sfx_i=\mu_i-\theta(i-1),\quad \sfx_i+e_i=\lambda_i-\theta(i-1), \quad1\leq i\leq n,
\end{equation}
then $e_i\in\{0,1\}$ and the transition probability is given by
\begin{align}\label{e:drift_MDdynamics}
\bP^{\sfb,q}(\bm\sfx+\bme|\bm\sfx)=\frac{1}{\Pi(\rho, \bm\beta=b_\sft)}
\prod_{1\leq i<j\leq n}\frac{q^{\sfx_i+\theta e_i}-q^{\sfx_j+\theta e_j}}{q^{\sfx_i}-q^{\sfx_j}}b_\sft^{\sum_{i=1}^ne_i}.
\end{align}
\end{claim}

The following lemma gives the asymptotics of the Macdonald polynomial $P_\bmla(1,t,\dots, t^{N-1};q,t)$.
\begin{lemma}\label{l:Prhoformula}
Given a Young diagram $\bmla\in \bY_N$, we identify it as an particle configuration
\begin{equation}
\bm\sfx=(\sfx_1,\sfx_2,\dots, \sfx_{N})\in \bW_\theta^{N},\quad \sfx_i=\la_i-\theta(i-1) \quad1\leq i\leq N,
\end{equation}
then 
\begin{align}\label{e:MDP2}
P_\bmla(1,t,\dots, t^{N-1};q,t)=q^{\sum_{i=1}^{N}(i-1)\theta(\sfx_i+(i-1)\theta)}\prod_{i<j\leq N}\frac{\Gamma_q(\sfx_i-\sfy_j)}{\Gamma_q(\sfx_i-\sfx_j+\theta)}\prod_{i=1}^N \frac{\Gamma_q(i \theta)}{\Gamma_{q}(\theta)},
\end{align}
Given sequences a sequence of Young diagrams
\begin{align*}
\bm\la^{(N)}=(\la_1^{(N)}\geq \la_2^{(N)}\geq\cdots\geq \la_N^{(N)})\in \bY_N, \quad N\geq 1,
 \end{align*}
 such that 
\begin{enumerate}
\item
There exists a constant $C>0$, $\la_1^{(N)}\leq CN$
\item  There exists a $1$-Lipschitz nondecreasing function $h:\bR\mapsto [0,\theta]$, and when $N\rightarrow \infty$
\begin{align}\label{e:hxdef}
\frac{\theta}{N}\sum_{i=1}^N\delta\left(\frac{\lambda_i^{(N)}-(i-1)\theta}{N}\right)\rightarrow \del_x h(x),
\end{align}
in distribution. Then
\end{enumerate}
\begin{align}\begin{split}\label{e:MDP3}
\lim_{N\rightarrow\infty}\frac{1}{N^2}\ln P_\bmla(1,t,\dots, t^{N-1};q,t)
=\frac{1}{2\theta}\int_{\bR^2} \ln(1-e^{\kappa|x-y|})\rd h(x)\rd h(y)-\theta \int_0^1 x\ln(1-e^{\kappa\theta x})\rd x\\
-\kappa\theta^2 \int \frac{(1-x)x}{e^{-\kappa \theta x}-1}\rd x+\frac{\kappa}{\theta}\int xh(x) \rd h(x)  +\frac{\kappa \theta^2}{3}+\oo(1). 
\end{split}\end{align}
\end{lemma}
\begin{proof}
We can reorganize \eqref{e:MDP}, and rewrite it in terms of the $q$-Gamma functions
\begin{align*}
&\phantom{{}={}}\prod_{i<j\leq N}\frac{(q^{\la_i-\la_j}t^{j-i};q)_\infty}{(q^{\la_i-\la_j}t^{j-i+1};q)_\infty}\frac{(t^{j-i+1};q)_\infty}{(t^{j-i};q)_\infty}\\
&=\prod_{i<j\leq N}\frac{(q^{\sfx_i-\sfy_j};q)_\infty}{(q^{\sfx_i-\sfx_j+\theta};q)_\infty}\frac{(q^{\theta(j-i+1)};q)_\infty}{(q^{\theta(j-i)};q)_\infty}\\
&=\prod_{i<j\leq N}\frac{\Gamma_q(\sfx_i-\sfy_j+\theta)}{\Gamma_q(\sfx_i-\sfx_j)}\frac{\Gamma_q(\theta(j-i))}{\Gamma_{q}(\theta(j-i+1))}\\
&=\prod_{i<j\leq N}\frac{\Gamma_q(\sfx_i-\sfy_j+\theta)}{\Gamma_q(\sfx_i-\sfx_j)}\prod_{i=1}^N \frac{\Gamma_q(\theta)}{\Gamma_{q}(i \theta)}.
\end{align*}
The claim \eqref{e:MDP2} follows.
For the asymptotics \eqref{e:NlimitJ1}, simply write 
$\lambda_i^{(N)}-(i-1)\theta=\sfx_i$. By our assumption \eqref{e:hxdef}, the exponent in the first term on the righthand side of \eqref{e:MDP2} is
\begin{align}\label{e:Gammat0} 
\frac{\kappa}{N^3}\sum_{i=1}^{N}(i-1)\theta(\sfx_i+(i-1)\theta)
=\frac{\kappa }{N}\sum_{i=1}^{N}\frac{\theta i}{N} \frac{\sfx_i}{N} +\frac{\kappa \theta^2}{3} +\OO(1/N)
=\frac{\kappa}{\theta}\int_\bR xh(x) \rd h(x)  +\frac{\kappa \theta^2}{3} +\oo(1).
\end{align}

We recall the following estimates for the $q$-Gamma function for $q\leq 1$, the first statement follows from \cite[Corollary]{moak1984q}, and the second follows from 
\cite[Theorem 1]{moak1984q}.
\begin{align*}
&\left|\ln\Gamma_q(z)-(z-1/2)\ln \left(\frac{q^z-1}{q-1}\right)-\frac{1}{\ln q}\int_{-\ln q}^{-z\ln q}\frac{u\rd u}{e^u-1}\right|\leq C,\quad z\geq \theta\\
&\left|\del_z \ln\Gamma_q(z)-\ln \left(\frac{q^z-1}{q-1}\right)\right|\leq \frac{C}{z},\quad z\geq \theta.
\end{align*}
So we have
\begin{align}\begin{split}\label{e:Gammat1} 
&\phantom{{}={}}\frac{1}{N^2}\sum_{i=1}^N\ln \Gamma_q(\theta)-\ln\Gamma_q(i\theta)
=-\frac{1}{N^2}\sum_{i=1}^N \left(i\theta \ln \left(\frac{q^{i\theta}-1}{q-1}\right) +\frac{1}{\ln q}\int_0^{-i\theta \ln q}\frac{u\rd u}{e^u-1}\right)+\OO(1/N)\\
&=\frac{\theta}{2}\ln\left(\frac{\kappa}{N}\right)-\theta \int_0^1 x\ln(1-e^{\kappa\theta x})\rd x -\kappa\theta^2 \int_0^1 \frac{(1-x)x}{e^{-\kappa \theta x}-1}\rd x+\OO(\ln N/N).
\end{split}\end{align}
and
\begin{align}\begin{split}\label{e:Gammat2} 
&\phantom{{}={}}\frac{1}{N^2}\sum_{i<j}\ln \frac{\Gamma(\sfx_i-\sfx_j+\theta)}{\Gamma(\sfx_i-\sfx_j)}
=\frac{1}{N^2}\sum_{i<j}\theta\ln \left(\frac{q^{\sfx_i-\sfx_j}-1}{q-1}\right)+\OO\left(\sum_{i<j}\frac{1}{\sfx_i-\sfx_j}\right)\\
&=\frac{1}{N^2}\sum_{i<j}\theta\ln \left(\frac{q^{\sfx_i-\sfx_j}-1}{q-1}\right)+\OO\left(\ln N/N\right)\\
&=-\frac{\theta}{2}\ln\left(\frac{\kappa}{N}\right)+\frac{1}{2\theta}\iint_{\bR^2} \ln(1-e^{\kappa|x-y|})\rd h(x)\rd h(y)+\OO(\ln N/N).
\end{split}\end{align}
where we used that $\sfx_i-\sfx_j\geq \theta(i-j)$ for the second inequality; in the last equality we used \Cref{l:freeentropy}. The claim  \eqref{e:MDP3} follows from plugging \eqref{e:Gammat0}, \eqref{e:Gammat1} and \eqref{e:Gammat2} into \eqref{e:MDP2}.
\end{proof}

\subsection{Proof of \Cref{t:main3}} \label{s:proofmain3}
Thanks to the relation \eqref{e:MDdynamics1}, the large deviation asymptotics of the skew Macdonald polynomials can be obtained from the large deviation principle of the dynamics \eqref{e:drift_MDdynamics},
\begin{align}\begin{split}\label{e:MDdynamics2}
    &\phantom{{}={}}\frac{1}{N^2}\ln Q_{\bmla/\bmmu}(\bm\beta=(b_0, \cdots, b_{\sfT-1});q,t)
    =\frac{1}{N^2}\ln \left(\bP(\bmla(\sfT)= \bmla|\bmla(0)=\bmmu)\Pi(\rho, \bm\beta=(b_0, \cdots, b_{\sfT-1}))\right)\\
    &+\frac{1}{2\theta}\left.\left(\int_{\bR^2}\ln(1-e^{\kappa|x-y|})\rd h(x, t)\rd h(y,t)+2\kappa\int_\bR xh(x,t) \rd h(x,t) \right)\right|_{t=T}^{t=0}+\oo(1),
\end{split}\end{align}
where we used \Cref{l:Prhoformula}.
	The dynamics \eqref{e:drift_MDdynamics}  and the non-intersecting $\theta$-Bernoulli random walks \eqref{e:mdensity2} differ by the drift $E(\{\bm\sfx(\sft)\}_{0\leq \sft\leq \sfT})D(\{\bm\sfx(\sft)\}_{0\leq \sft\leq \sfT})$ with
\begin{align*}
& D(\{\bm\sfx(\sft)\}_{0\leq \sft\leq \sfT}):=\prod_{0\le {\mathsf t}\le \mathsf{T-1}}2^N b_{\mathsf {t}}^{\sum_{i=1}^{N}e_{i}(\mathsf t)}, \\
&E(\{\bm\sfx(\sft)\}_{0\leq \sft\leq \sfT}):=\prod_{0\le {\mathsf t}\le \mathsf{T}} \prod_{i<j}  \frac{  {\mathsf{x}_{i}(\mathsf{t})}-{\mathsf{x}_{j}(\mathsf{t})}}{ {\mathsf{x}_{i}(\mathsf{t})+\theta e_{i}(\mathsf{t})}-({\mathsf{x}_{j}(\mathsf{t})+\theta e_{j}(\mathsf{t}))}}
\frac{ q^{\mathsf{x}_{i}(\mathsf{t})+\theta e_{i}(\mathsf{t})}-q^{\mathsf{x}_{j}(\mathsf{t})+\theta e_{j}(\mathsf{t})}}{  q^{\mathsf{x}_{i}(\mathsf{t})}-q^{\mathsf{x}_{j}(\mathsf{t})}},\\
&\Pi((\rho, \bm\beta=(b_0, \cdots, b_{\sfT-1}))\rd \bP^{\sfb,q}=E(\{\bm\sfx(\sft)\}_{0\leq \sft\leq \sfT})D(\{\bm\sfx(\sft)\}_{0\leq \sft\leq \sfT})\rd \bP.
\end{align*}
\Cref{t:main3} follows from plugging \Cref{l:drift} and the following lemma to \eqref{e:MDdynamics2}. We remark that all the terms involving $\kappa$ cancels out perfectly.
\begin{lemma}
Adopt the notations and assumptions of \Cref{t:main3}, the following holds
$$\frac{1}{N^{2}}\ln E(\{\bm\sfx(\sft)\}_{0\leq \sft\leq \sfT})=
 \frac{1}{2\theta}\left.\left(\iint_{\bR^2} \ln \frac{1-e^{\kappa |x-y|}}{-\kappa |x-y|} \rd h(y,t)\rd h(x,t)+2\kappa\int_{\bR} xh(x,t) \rd h(x,t)\right)\right|_{t=0}^{t=T} +\oo(1).$$
\end{lemma} 

\begin{proof}
 We observe that
\begin{align*}
\frac{q^{x}-q^{y}}{x-y}=\frac{1}{\ln q}\int_{0}^{1} q^{\alpha x+(1-\alpha)y}\rd\alpha,
\end{align*}
 so that
\begin{align*}\begin{split}
&\phantom{{}={}} \frac{  {\mathsf{x}_{i}(\mathsf{t})}-{\mathsf{x}_{j}(\mathsf{t})}}{ {\mathsf{x}_{i}(\mathsf{t})+\theta e_{i}(\mathsf{t})}-({\mathsf{x}_{j}(\mathsf{t})+\theta e_{j}(\mathsf{t}))}}\frac{ q^{\mathsf{x}_{i}(\mathsf{t})+\theta e_{i}(\mathsf{t})}-q^{\mathsf{x}_{j}(\mathsf{t})+\theta e_{j}(\mathsf{t})}}{  q^{\mathsf{x}_{i}(\mathsf{t})}-q^{\mathsf{x}_{j}(\mathsf{t})}}\\
& =\frac{q^{\theta e_j(\sft)}\int_0^1 q^{\al(\sfx_i(\sft)+\theta e_i(\sft)-\sfx_j(\sft)-\theta e_j(\sft))} q^{}\rd \al}{\int_0^1 q^{\al(\sfx_i(\sft)-\sfx_j(\sft))}\rd \al}
=q^{\theta e_j(\sft)}\mathbb E_{q^{{\mathsf{x}_{i}(\mathsf{t})-{\mathsf{x}_{j}(\mathsf{t})}}}}[q^{\alpha\theta(e_{i}(\mathsf{t})-e_{j}(\mathsf(t))}],
\end{split}\end{align*}
where for a real number $\beta$,  $\bE_{\beta}$ denotes the expectation over the variable $\alpha\in [0,1]$ given, for any bounded measurable function $f$ on $[0,1]$,  by
$$\bE_{\beta }[f(\alpha)]=\frac{ \int_{0}^{1} f(\alpha)\beta^{\alpha } \rd\alpha}{\int_{0}^{1} \beta^{\alpha }\rd\alpha}\,.$$
We recall that $q=e^{\kappa/N}$, and denote $t=\sft/N$, so that
\begin{align*}
\bE_{q^{{\mathsf{x}_{i}(\mathsf{t})-{\mathsf{x}_{j}(\mathsf{t})}}}}[q^{\alpha\theta(e_{i}(\mathsf{t})-e_{j}(\mathsf(t))}]=e^{
 \frac{\theta\kappa ( e_{i}(\mathsf{t})-e_{j}(\mathsf{t}))}{N} \bE_{e^{\kappa({x}_{i}(t)-{{x}_{j}(t)})}}[\alpha]+\OO\left(\frac{1}{N^{2}}\right)},
\end{align*}
where $\OO(1/N^{2})$ is uniform over the ${x}_{i}(t), {x}_{i}(t)$ and over $\theta$ in a bounded set.  Recall that $e_{i}(\mathsf t)=\mathsf{x}_{i}(\mathsf t+1)-\mathsf{x}_{i}(\mathsf t)$.
Hence, we deduce that if $f(y):= \bE_{e^{ y}}[\alpha]$, 
\begin{align}\begin{split}\label{dens}
E(\{\bm\sfx(\sft)\}_{0\leq \sft\leq \sfT})&=\prod_{0\leq \sft\leq \sfT-1} e^{\frac{\kappa}{N}\sum_{1\leq i\leq N}\theta (i-1)e_i(\sft)}\times\\
&\times e^{ \theta \kappa \sum_{0\le {\mathsf t}\le \mathsf{T}-1: t=\sft/N} \sum_{i<j}  f(\kappa({{x}_{i}(t)-{{x}_{j}(t)}})) (({{x}_{i}(t+1/N)-{{x}_{j}(t+1/N)}})-({{x}_{i}(t)-{{x}_{j}(t)}}))
+\OO(N)}\,.\end{split}\end{align}
For the first term on the righthand side of \eqref{dens}, we have
\begin{align}\begin{split}\label{e:tt1}
&\phantom{{}={}}\frac{\kappa}{N^3}\sum_{\sft=0}^{\sfT-1}\sum_{i=1}^{N}(i-1)\theta(\sfx_i(\sft+1)-\sfx_i(\sft))
=\frac{\kappa}{N^2}\sum_{i=1}^{N}(i-1)\theta(x_i(T)-x_i(0))\\
&
=\frac{\kappa }{N}\sum_{i=1}^{N}\frac{\theta i}{N} (x_i(T)-x_i(0)) +\OO(1/N)
=\left.\frac{\kappa}{\theta}\int_\bR xh(x,t) \rd h(x,t) \right|_{t=0}^{t=T}  +\oo(1).
\end{split}\end{align}

For the second term on the righthand side of \eqref{dens}, let 
\begin{align*}
F(x)=\int_{0}^{\kappa x} f(y) \rd y =\int_0^{\kappa x} \del_y \ln \int_0^1 e^{\al y}\rd \al \rd y=\ln \int_{0}^{1} e^{\alpha \kappa x}\rd\alpha =\ln \frac{1-e^{\kappa x}}{-\kappa x},
\end{align*}
so that $F'(x)=\kappa f(\kappa x)$, and 
\begin{align*}
\begin{split}
&\phantom{{}={}}\kappa f(\kappa({{x}_{i}(t)-{{x}_{j}(t)}})) ({{x}_{i}(t+1/N)-{{x}_{j}(t+1/N)}}-({{x}_{i}(t)-{{x}_{j}(t)}}))\\
&=F({{x}_{i}(t+1/N)-{{x}_{j}(t+1/N)}}) -F({{x}_{i}(t)-{{x}_{j}(t)}}) +\OO\left(\frac{1}{N^{2}}\right).
\end{split}
\end{align*}
where the error is again uniform since $f$ is uniformly Lipschitz as 
 $|\partial_{y }f(y) |\le 2\,.$
Therefore, for the second term on the righthand side of \eqref{dens}, we have  
\begin{align}\begin{split}\label{e:tt2}
&\phantom{{}={}}\frac{\theta }{N^2} \sum_{i<j}  F( {{x}_{i}(T )-{{x}_{j}(T)}})-   F({{x}_{i}(0)-{{x}_{j}(0)}})
  +\oo(1)\\
&=
 \frac{1}{2\theta}\left.\iint_{\bR^2} F(|x-y|) \rho(y;\bmx(t))\rho(x;\bmx(t))\rd y \rd x\right|_{t=0}^{t=T} +\OO(1/N),
 \end{split}\end{align}
 where we finally used that $F$ is uniformly Lipschitz on compacts. 
 Finally, we can see that the above right hand side is a continuous function of $H(y; \bmx(T))$ and $H(y; \bmx(0))$ equipped with the uniform topology, provided its derivative is a probability measure. The proof follows the ideas from  
 Lemma \ref{l:freeentropy}, so we omit the details.   It follows from plugging \eqref{e:tt1} and \eqref{e:tt2} into \eqref{dens}, 
 \begin{align*}\begin{split}
\frac{1}{N^2}\ln E(\{\bm\sfx(\sft)\}_{0\leq \sft\leq \sfT})
=
 \frac{1}{2\theta}\left.\left(\iint_{\bR^2} \ln \frac{1-e^{\kappa |x-y|}}{-\kappa |x-y|} \rd h(y,t)\rd h(x,t)+2\kappa\int_{\bR} xh(x,t) \rd h(x,t)\right)\right|_{t=0}^{t=T} +\oo(1).
  \end{split}\end{align*}
 
%
\end{proof}

\bibliography{References.bib}
\bibliographystyle{abbrv}

\end{document}